\documentclass[11pt]{article}
\usepackage[margin=1in]{geometry}
\usepackage[usenames,dvipsnames]{color}
\usepackage{xcolor}
\usepackage{subfigure}
 
\usepackage{listings}
\usepackage{geometry}

\usepackage{amsmath,amsthm,amsfonts,amssymb,amscd}
\usepackage{times}

\usepackage{mathtools}
\usepackage{mathabx}
\usepackage{float}
\usepackage{makeidx}
\usepackage{stmaryrd}
\usepackage{mathrsfs}
\usepackage{emptypage}
\usepackage{xfrac}
\usepackage{bbm}

\renewcommand{\epsilon}{\varepsilon}

\usepackage{graphicx}

\newcommand{\pa}{\partial}

\newcommand*{\sign}{\ensuremath{\mathrm{sgn\,}}}

\renewcommand*{\tilde}{\widetilde}

\newcommand{\X}[1]{\left \rVert #1 \right \rVert_{L^2}}

\newcommand{\Xm}[1]{\left \rVert #1 \right \rVert_{X_m}}

\newcommand{\Ltwo}[1]{\left \rVert #1 \right \rVert_{L^2}}
\newcommand{\Ltwox}[1]{\left \rVert #1 \right \rVert_{L_x^2}}
\newcommand{\Y}[1]{\left \rVert #1 \right \rVert_{L^\infty}}
\newcommand{\Lone}[1]{\left \rVert #1 \right \rVert_{L^1 }}
\newcommand{\Loneinf}[1]{\left \rVert #1 \right \rVert_{L^1 L^\infty}}
\newcommand{\Ltwoinf}[1]{\left \rVert #1 \right \rVert_{L_y^2 L_x^\infty}}
\newcommand{\Linf}[1]{\left \rVert #1 \right \rVert_{L_y^\infty L_x^\infty}}

\newenvironment{detailssth}[1]{\begin{trivlist} \item[] {\textbf{Details of #1:}}}{
                       \end{trivlist}}

\newtheorem{theorem}{Theorem}[section]
\newtheorem{lemma}[theorem]{Lemma}

\newtheorem{proposition}[theorem]{Proposition}
\newtheorem{corollary}[theorem]{Corollary}
\theoremstyle{definition}
\newtheorem{definition}[theorem]{Definition}

\newtheorem{rmk}[theorem]{Remark}
\numberwithin{equation}{section}

\def\f1r{{\frac{1}{r}}  }
\def\f1r{{\frac{1}{r}}  }

\title{{\bf Existence of analytic non-convex V-states}}
 \author{Gerard Castro-L\'opez, Javier G\'omez-Serrano}

\date{} %
\begin{document}

\maketitle
%\tableofcontents

\begin{abstract}
    V-states are uniformly rotating vortex patches of the incompressible 2D Euler equation and the only known explicit examples are circles and ellipses. In this paper, we prove the existence of 
    non-convex V-states with analytic boundary which are far from the known examples. To prove it, we use a combination of analysis of the linearized operator at an approximate solution and computer-assisted proof techniques.
\end{abstract}

\tableofcontents
\section{Introduction}
\subsection{The equations}

In this paper we are concerned with the incompressible 2D Euler equation (in vorticity form):
\begin{equation}\label{EulerEq_vorticity1}
    \begin{cases}
        \partial_t \omega + u \cdot \nabla \omega   = 0 & \text{in } \mathbb{R}^2 \times [0, \infty) \\
        \text{curl} (u) = \omega &\text{in } \mathbb{R}^2 \times [0, \infty)\\
        \nabla \cdot u = 0 & \text{in } \mathbb{R}^2 \times [0, \infty) \\
        \omega(x, 0) = \omega_0(x) & \text{in } \mathbb{R}^2. 
    \end{cases}
\end{equation}
It is a transport equation for the vorticity, which is conserved along the particle trajectories. In 1963 Yudovich \cite{Yudovich:Nonstationary-ideal-incompressible} proved uniqueness and global existence of (weak) solutions for initial data $\omega_0$ in $L^1\cap L^\infty$. We will now focus on vortex patch solutions, namely when the initial data is the characteristic function of a closed set $D_0 \subset \mathbb{R}^2$.  Throughout this paper, we will assume that $D_0$ is simply connected.

Since the vorticity is transported along the flow of $u$, the vorticity remains a characteristic function of a set $D(t)$ that changes over time.
Then, assuming $\partial D(t) = \{z(t, x) : x \in [0, 2\pi]\}$, where $z(t,x) \in \mathbb{R}^2$ is a parametrization of the boundary of the set $D(t)$ at time $t$, we can reduce the Euler equation \eqref{EulerEq_vorticity1} to an non local evolution equation for the boundary of the set $D(t)$:
\begin{equation}\label{eq_vortex_patch_perp}
    \partial_t z(t, x) \cdot \partial_x z(t, x)^\perp = -\frac{1}{2\pi}
 \int_0^{2\pi} \log (|z(t, y)- z(t, x)|) \partial_y z(t, y) \cdot \partial_x z(t, x)^\perp d y.
 \end{equation}

Vortex patches are globally well posed for boundary data in $C^{1,\theta}$. That is, for every initial boundary data $z_0(x)$ in $C^{1, \theta}$ the solution $z(t, x)$ always stays in $C^{1, \theta}$. This was proved by Chemin \cite{Chemin:persistance-structures-fluides-incompressibles}, and also by Bertozzi-Constantin \cite{Bertozzi-Constantin:global-regularity-vortex-patches}. Recently, Kiselev--Luo \cite{Kiselev-Luo:illposedness-c2-vortex-patches} have proved ill-posedness in $C^2$. From now on, we will focus on a more particular solution,  uniformly rotating vortex patches or \textit{V-states}. These are vortex patches $D(t)$ such that they rotate with constant angular velocity $\Omega$, so $D(t) = M(\Omega t) D_0$ where $\Omega$ is the angular velocity and $M$ is the counterclockwise rotation matrix
\begin{equation*}
    M (\Omega t)= 
    \begin{pmatrix}
        \cos(\Omega t) & -\sin(\Omega t) \\
        \sin(\Omega t) & \cos(\Omega t).
    \end{pmatrix}
\end{equation*}
To find which initial domains $D_0$ evolve like this, we impose $z(t, x) = M(\Omega t) z_0(x)$ in equation \eqref{eq_vortex_patch_perp} and obtain
\begin{equation}\label{eqz_vstate}
    \Omega z_0(x)^\perp \cdot \partial_x z_0(x) = -\frac{1}{2\pi} \int_0^{2\pi} \log (|z_0(x)- z_0(y)|) \partial_y z_0(y) \cdot \partial_x z_0(x)^\perp. 
\end{equation}

Parametrizing in polar coordinates the boundary  $z_0(x) = R(x)(\cos(x), \sin(x))$, we arrive to the following equation for $R(x)$:
\begin{equation}\label{eqR}
     R R' =   F[R] 
\end{equation}
where $F[R] := F_1[R] + F_2[R]$  and the $F_i$ are
\begin{equation}\label{F_def}
\begin{aligned}
    &F_1[R] \coloneqq \frac{1}{4\pi\Omega}\int_0^{2\pi} \cos(x-y) \log \left ( (R(x)-R(y))^2 + 4R(x)R(y)\sin^2 \left (\frac{x-y}{2} \right )  \right ) (R(x)R'(y)-R'(x)R(y)) dy  \\ 
    &F_2[R] \coloneqq \frac{1}{4\pi\Omega}\int_0^{2\pi} \sin(x-y) \log \left ( (R(x)-R(y))^2 + 4R(x)R(y)\sin^2 \left (\frac{x-y}{2} \right )  \right ) (R(x)R(y) + R'(x)R'(y)) dy. 
\end{aligned}
\end{equation}
We will consider solutions to \eqref{eqR} defined on the torus $\mathbb{T} = \mathbb{R}/2\pi\mathbb{Z}$.

\subsection{Historical background}

The existence of V-states has been of great importance over centuries, getting an explosion of results during the last decade. It is known since the 1800s that circles are stationary (and thus rotating for any $\Omega$). The first example of non-radial solutions was given by Kirchhoff in 1874 \cite{Kirchhoff:vorlesungen-math-physik}, proving that ellipses (now known as the \textit{Kirchhoff ellipses}) are also uniformly rotating vortex patches with angular velocity $\Omega_{a,b} = \frac{ab}{(a+b)^2}$, being $a,b$ the semiaxes of the ellipse. Other than these two families there are no other explicit solutions known.

% Numerics
The first comprehensive and methodical search for V-states was done numerically. Deem--Zabusky \cite{Deem-Zabusky:vortex-waves-stationary} coined the term and discovered numerically families of $m$-fold uniformly rotating solutions bifurcating from the circle for multiple $m \geq 3$ doing continuation together with a Newton-Raphson method. These solutions can be seen as a generalization of the Kirchhoff ellipses branch. Wu--Overman--Zabusky \cite{Wu-Overman-Zabusky:steady-state-Euler-2d}, using more computational power and improved second- and third-order algorithms were able to perform the continuation until close to the limiting shape, suggesting a 90-degree corner (see \cite{Overman:steady-state-euler-local-analysis-vstates} and \cite{Saffman-Tanveer:touching-pair-unniform-vortices} for asymptotic analysis and \cite{Wang-Zhang-Zhou:Boundary-regularity-v-states} for a calculus of variations' proof of the alternatives of the endpoint). Previously, Pierrehumbert \cite{Pierrehumbert:v-states-cusp} had predicted a cusp. Kamm \cite{Kamm:thesis-shape-stability-patches} computed secondary bifurcations from the elliptical branch, while Saffman--Szeto  \cite{Saffman-Szeto:equilibrium-shapes-equal-uniform-vortices} calculated translating V-states and their stability. Using variational energy arguments and velocity impulse diagrams,
Luzatto-Fegiz--Williamson \cite{LuzzattoFegiz-Williamson:efficient-numerical-method-steady-uniform-vortices,LuzzattoFegiz-Williamson:resonant-instability-2d-vortex,LuzzattoFegiz-Williamson:stability-2d-flows-imperfect-velocity-impulse,LuzzattoFegiz:bifurcation-stability-opposite-signed-vortex,LuzzattoFegiz-Williamson:stability-conservative-flows-steady-solutions-variational-argument,LuzzattoFegiz-Williamson:stability-elliptical-vortices-ivi-diagrams,LuzzattoFegiz-Williamson:structure-stability-von-karman-street} (cf. also \cite{Cerretelli-Williamson:new-family-vortices}) managed to establish linear stability properties without the need to actually perform a full linear stability
analysis, as well as to compute limiting V-states with high precision.

Related numerical computations include the work of Dritschel \cite{Dritschel:stability-energetics-corotating-vortices,Dritschel:general-theory-2d-vortex,Dritschel:repeated-filamentation-2d-vorticity} on co-rotating vortices, unequal vorticity equilibria and instability of disks and ellipses respectively; Mitchell-Rossi \cite{Mitchell-Rossi:evolution-kirchhoff-elliptic-vortices} on the evolution of perturbed ellipses; Dhanak \cite{Dhanak:stability-polygon-vertices} on the stability of small-area vortices arranged as a polygon; Elcrat--Fornberg--Miller \cite{Elcrat-Fornberg-Miller:stability-vortices-cylinder} on the stability of vortices past a cylinder and Elcrat--Protas \cite{Elcrat-Protas:framework-linear-stability-vortices} on revisiting the question of the linear stability analysis.

% Theory
Concerning existence of V-states, Burbea \cite{Burbea:motions-vortex-patches} gave the first proof of local bifurcation branches of $m$-fold symmetric solutions emerging from the disk at angular velocities $\Omega_m = \frac{m-1}{2m}$. Later, Hmidi--Mateu--Verdera \cite{Hmidi-Mateu-Verdera:rotating-vortex-patch} proved $C^\infty$ boundary regularity of those solutions, and Castro--Córdoba--Gómez-Serrano \cite{Castro-Cordoba-GomezSerrano:analytic-vstates-ellipses} improved the result to analytic. The existence of secondary bifurcations from the elliptical branch were proved by Castro--C\'ordoba--G\'omez-Serrano \cite{Castro-Cordoba-GomezSerrano:analytic-vstates-ellipses} (with analytic boundary regularity) and Hmidi-Mateu \cite{Hmidi-Mateu:bifurcation-kirchhoff-ellipses} (with $C^\alpha$). We also mention other scenarios such as the confined case \cite{delaHoz-Hassainia-Hmidi-Mateu:vstates-disk-euler} or the doubly connected case \cite{Hmidi-delaHoz-Mateu-Verdera:doubly-connected-vstates-euler}, where degenerations appear \cite{Hmidi-Mateu:degenerate-bifurcation-vstates-doubly-connected-euler,Hmidi-Renault:existence-small-loops-doubly-connected-euler,Wang-Xu-Zhou:degenerate-bifurcations-2fold-doubly-connected-patches}, as well as the co-rotating case \cite{Hmidi-Mateu:existence-corotating-counter-rotating,Garcia-Haziot:global-bifurcation-corotating-counter-rotating}, multipole configurations \cite{Hassainia-Wheeler:multipole-vstates-active-scalars,Garcia:Karman-vortex-street,Garcia:vortex-patch-choreography}, asymmetric vortex pairs \cite{Hassainia-Hmidi:steady-asymmetric-vortex-pairs-euler} and variational constructions \cite{Turkington:corotating-vortices}. Using global bifurcation theory, Hassainia-Masmoudi-Wheeler \cite{Hassainia-Masmoudi-Wheeler:global-bifurcation-vortex-patches} constructed global bifurcating curves (with analytic boundary) bifurcating from the circle. Their theorem outlines alternatives for the curve to end but doesn't narrow it to one due to the lack of information on solutions along the branch. See also the work of Carrillo–Mateu–Mora–Rondi–Scardia–Verdera \cite{Carrillo-Mateu-Mora-Rondi-Scardia-Verdera:dislocation-ellipses,Mateu-Mora-Rondi-Scardia-Verdera:explicit-minimisers-anisotropic-coulomb-3d} for variational techniques
applied to other anisotropic problems related to vortex patches and the very recent results by Choi--Jeong--Sim and Huang--Tong \cite{Choi-Jeong-Sim:existence-sadovskii-vortex-patch,Huang-Tong:steady-contiguous-vortex-patch-dipole} on the Sadovskii vortex patch. Other constructions of periodic or quasi-periodic patch solutions using KAM theory have been performed by Berti--Hassainia--Masmoudi \cite{Berti-Hassainia-Masmoudi:time-quasiperiodic-vortex-patches} close to ellipses, Hassainia--Hmidi--Roulley \cite{Hassainia-Hmidi-Roulley:invariant-kam-annular-vortex-patches} close to annuli and Hassainia--Roulley \cite{Hassainia-Roulley:quasiperiodic-euler-boundary} close to disks.

Another interesting direction involving V-states is to determine the range (in $\Omega$) of their existence (also known as the question of \textit{rigidity} vs \textit{flexibility}). Fraenkel \cite{Fraenkel:book-maximum-principles-symmetry-elliptic} proved using moving plane methods that the only simply-connected patch that is a solution for $\Omega = 0$ is the disk. This was later improved by Hmidi \cite{Hmidi:trivial-solutions-rotating-patches} in the cases $\Omega = \frac12$ and $\Omega < 0$ (the latter case under an additional convexity assumption). Finally, G\'omez-Serrano--Park--Shi--Yao \cite{GomezSerrano-Park-Shi-Yao:radial-symmetry-stationary-solutions}, using calculus of variations methods closed the problem, allowing them to remove the simply connected geometry assumption, and proving that the only uniformly rotating vortex patches for $\Omega \notin (0, \frac12)$ and signed vorticity are the disks, with the bounds on $\Omega$ being sharp. See also results on quantitative estimates by Park \cite{Park:quantitative-estimates-v-states} and related results by Huang \cite{Huang:rigidity-vstate-near-rankine}. In a different direction G\'omez-Serrano--Park--Shi \cite{GomezSerrano-Park-Shi:nonradial-stationary-solutions} constructed nonradial, finite energy patch solutions with sign-changing vorticity. See also the recent nonradial constructions by Enciso--Fern\'andez--Ruiz--Sicbaldi \cite{Enciso-Fernandez-Ruiz-Sicbaldi:schiffer-annuli-stationary-planar-euler,
Enciso-Fernandez-Ruiz:smooth-nonradial-stationary-euler}
and other rigidity/flexibility results by Hamel--Nadirashvili \cite{Hamel-Nadirashvili:rigidity-euler-annulus}, G\'omez-Serrano--Park--Shi--Yao \cite{GomezSerrano-Park-Shi-Yao:rotating-solutions-vortex-sheet-rigidity, GomezSerrano-Park-Shi-Yao:rotating-solutions-vortex-sheet-flexibility,GomezSerrano-Park-Shi:nonradial-stationary-solutions}, Dom\'inguez--Enciso--Peralta-Salas \cite{Dominguez-Enciso-PeraltaSalas:piecewise-smooth-stationary-euler} and Ruiz \cite{Ruiz:symmetry-compactly-supported-steady-2d-euler}.

% Stability
With respect to the stability of V-states, Love \cite{Love:stability-ellipses} proved the linear stability for ellipses of aspect ratio bigger than $\frac13$ and linear instability for ellipses of aspect
ratio smaller than $\frac13$. Since then, most of the results have been devoted to upgrade Love's results to the nonlinear setting. Wan \cite{Wan:stability-rotating-vortex-patches}, and Tang \cite{Tang:nonlinear-stability-vortex-patches} proved the nonlinear
stable case, whereas Guo–Hallstrom–Spirn \cite{Guo-Hallstrom-Spirn:dynamics-unstable-kirchhoff-ellipse} settled the nonlinear unstable one. See also \cite{Constantin-Titi:evolution-nearly-circular-vortex-patches}. We also point out to other stability results by Sideris--Vega \cite{Sideris-Vega:stability-L1-patches}, Choi--Jeong 
\cite{Choi-Jeong:stability-instability-kelvin-waves} and Cao--Wan--Wang \cite{Cao-Wan-Wang:nonlinear-stability-patches-bounded-domains}.

% \begin{figure}[t]
%     \centering
%     \includegraphics[width=6cm]{Net2/Images/zabusky.png}
%     \caption{Numerical solutions from Deem-Zabusky \cite{Deem1978}}
%     \label{fig:Zabusky}
% \end{figure}

The main result of this paper, Theorem \ref{main_theorem}, is the first positive existence result that gives quantitative information about the solution outside  the perturbative region around the circle or ellipses. The difficulty of finding solutions far from those regions is what makes the problem hard: global bifurcation or calculus of variations' techniques fails to provide quantitative information and local bifurcation doesn't work far from trivial solutions. We hope our method to be applicable to other situations and use the Main Theorem from this paper to showcase its power. In particular, the concrete choice of the $m=6$ branch is important since the branches corresponding to $m=2,3$ seem to be always convex so any general argument which doesn't use any additional information about $m$ is doomed to fail.

\subsection{Computer-assisted proofs}

One of the fundamental pillars of the proof relies on a computer-assisted argument. This type of arguments have accumulated a lot of momentum over the last few years, leading to the possibility to tackle more complex scenarios and problems. The main paradigm is to substitute floating point numbers on a computer by rigorous bounds, which are then propagated through every operation that the computer performs taking into account any error made throughout the process.

In our concrete case, we will mostly use the computer to perform three tasks:

\begin{enumerate}
    \item Compute certified eigenvalue bounds of matrices of moderate sizes (Lemma \ref{I+Ainvert})
    \item Compute operator bounds which can be recasted into integrals of complicated functions (Lemma \ref{L-L_Fbound})
    \item Quantify the error (the \textit{defect}) made by the approximation (Lemma \ref{CE0})
\end{enumerate}

Towards the first bullet point, somewhat similar techniques were applied in the context of computer-assisted proofs for eigenvalue problems in polygons by G\'omez-Serrano--Orriols \cite{GomezSerrano-Orriols:negative-hearing-shape-triangle}, Dahne--G\'omez-Serrano--Hou \cite{Dahne-GomezSerrano-Hou:counterexample-payne}, Bogosel--Bucur \cite{Bogosel-Bucur:polygonal-faber-krahn-validated}, whereas the reduction from operator bounds into complicated integrals has been implemented in other settings by Castro--C\'ordoba--G\'omez-Serrano \cite{Castro-Cordoba-GomezSerrano:global-smooth-solutions-sqg}, Chen-Hou-Huang \cite{Chen-Hou-Huang:blowup-degregorio}, Dahne \cite{Dahne:highest-cusped-waves-fkdv} or Dahne--G\'omez-Serrano \cite{Dahne-GomezSerrano:highest-wave-burgershilbert}.

In a broader context (computer-assisted proofs in PDE), the reduction of the proof of existence to a fixed-point argument has also been particularly successful in the context of \textit{radii polynomials}, developed in \cite{vandenBerg-Lessard:chaotic-braided-solutions-swift-hohenberg,Day-Lessard-Mischaikow:validated-continuation-equilibria-pde,Gameiro-Lessard-Mischaikow:validated-continuation-large-parameter-ranges-pde} and later used among others by van den Berg--Breden--Lessard--van Veen \cite{vandenBerg-Breden-Lessard-vanVeen:periodic-orbits-ns} in the context of periodic solutions of Navier-Stokes, or by Castelli--Gameiro--Lessard \cite{Castelli-Gameiro-Lessard:rigorous-numerics-ill-posed-pde} for Boussinesq.

We would also like to highlight other works that have successfully completed a computer-assisted proof for problems in fluid mechanics such as Guo--Had\v zi\'c--Jang--Schrecker \cite{Guo-Hadzic-Jang-Schrecker:gravitational-collapse-stars-self-similar} for Euler-Poisson; Kobayashi \cite{Kobayashi:global-uniqueness-stokes} for Stokes' extreme waves; Chen--Hou--Huang \cite{Chen-Hou-Huang:blowup-degregorio} for De Gregorio; Castro--C\'ordoba--G\'omez-Serrano \cite{Castro-Cordoba-GomezSerrano:global-smooth-solutions-sqg} for SQG;  Cadiot \cite{Cadiot:proofs-existence-stability-capillary-gravity-whitham} for Whitham; Arioli--Koch, Figueras--De la Llave, Gameiro--Lessard, Figueras--Gameiro--Lessard--De la Llave, Zgliczynski, Zgliczynski--Mischaikow \cite{Arioli-Koch:cap-stationary-ks,Figueras-DeLaLLave:cap-periodic-orbits-kuramoto,Gameiro-Lessard:periodic-orbits-ks,Figueras-Gameiro-Lessard-DeLaLLave:framework-cap-invariant-objects,Zgliczynski:periodic-orbit-kuramoto,Zgliczynski-Mischaikow:rigorous-numerics-kuramoto} for Kuramoto--Shivasinsky; Buckmaster--Cao-Labora--G\'omez-Serrano \cite{Buckmaster-CaoLabora-GomezSerrano:implosion-compressible} for compressible Euler and Navier-Stokes; and Arioli--Gazzola--Koch, Bedrossian--Punshon-Smith \cite{Arioli-Gazzola-Koch:uniqueness-bifurcation-ns,Bedrossian-PunshonSmith:chaos-stochastic-2d-galerkin-ns} for Navier-Stokes.

Another important set of works, more generally in (elliptic or parabolic) PDE are the classical works of Plum
\cite{Plum:H2-estimates-elliptic-bvp,Plum:numerical-existence-nonlinear-elliptic-bvps} and Nakao
\cite{Nakao:numerical-approach-elliptic-problems,Nakao:nonlinear-parabolic-problems} as well as more modern ones such as Fazekas--Pacella--Plum 
\cite{Fazekas-Pacella-Plum:nonradial-lane-emden-ball}, van den Berg--H\'enot--Lessard \cite{vandenBerg-Henot-Lessard:radial-solutions-semilinear-elliptic}, Jaquette--Lessard--Takayasu 
\cite{Jaquette-Lessard-Takayasu:global-dynamics-nonconservative-nls} or Breden--Engel 
\cite{Breden-Engel:cap-chaos-hopf-systems}. 

We also refer the reader to the books \cite{Moore-Bierbaum:methods-applications-interval-analysis,Tucker:validated-numerics-book} and to the survey \cite{GomezSerrano:survey-cap-in-pde} and the book \cite{Nakao-Plum-Watanabe:cap-for-pde-book} for a more specific treatment of computer-assisted proofs in 
PDE.

Finally, we would like to mention that our implementation of the computer-assisted proof was done using the Arb \cite{Johansson:Arb} library in C/C++.

\subsection{Main theorem, sketch of the proof and organization of the paper}
\label{subsec:main_result}

The main theorem of this work is the following:

\begin{theorem}\label{main_theorem}
    There exists an analytic solution $R(x)$ of \eqref{eqR} such that it parametrizes a vortex patch $D\subset \mathbb{R}^2$ which is non-convex and has $6$-fold symmetry. See Figure \ref{fig:patch}.
\end{theorem}

\begin{figure}[h]
    \centering    \includegraphics[width=5cm]{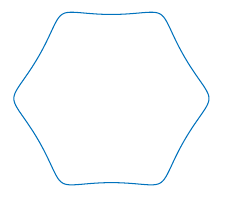}
    \caption{The boundary $\partial D$ is a curve contained in the plotted line.}
    \label{fig:patch}
\end{figure}

\begin{proof}[Proof of Theorem \ref{main_theorem}]
    Theorem \ref{Gcontractive} proves the existence of a fixed point for the functional $G$, and Proposition \ref{Unfolding_u_to_R} proves that given a fixed point of $G$ we can construct a solution $R \in H^1$ to our main equation \eqref{eqR}.
    Finally, using Proposition \ref{Ranalytic}, we prove that $R$ is analytic and Proposition \ref{NonConvex} proves the non-convexity. The proof of the theorem follows.
\end{proof}

% \begin{corollary}\label{main_corollary}
%     There exists a $\delta > 0$ such that for any angular velocity in $(\Omega-\delta, \Omega+\delta)$ there exists an analytic solution of equation \eqref{eqR} with 6-fold symmetry. %, where $\Omega$ is the angular velocity of the solution given by Theorem \ref{main_theorem}.
% \end{corollary}

% \begin{proof}
%     Let $\tilde{\Omega}$ be sufficiently close to $\Omega$ and $R_0$ be the approximate solution given in Section \ref{Sec:Prelim}. Since all the functionals are continuous in $\Omega$ and the fixed point conditions are open, then we can apply Theorem \ref{Gcontractive} to $(\tilde{\Omega}, R_0)$ if $|\Omega - \tilde{\Omega}|$ is small enough. 
    
% \end{proof}

% \color{orange}
% \begin{corollary}
%     There exists a $\delta > 0$ such that for any angular velocity in $(\Omega - \delta, \Omega + \delta)$ there exists an analytic of equation \eqref{eqR} such that its associated vortex patch is non-convex and has 6-fold symmetry.
% \end{corollary}
% \begin{proof}
%     All the bounds in the proof of the main theorem \ref{main_theorem} are open $(<)$ and continuous on $\Omega$, therefore, if we consider $\Tilde{\Omega}$ close enough to $\Omega$ the same proof holds.
% \end{proof}

%\color{red} NEW VERSION
\begin{corollary}
\label{main_corollary}
    There exists a $\delta > 0$ such that for any angular velocity in $(\Omega - \delta, \Omega + \delta)$ there exists an analytic solution of equation \eqref{eqR} such that it parametrizes a non-convex patch with 6-fold symmetry.
\end{corollary}
\begin{proof}
    Let $\tilde{\Omega}$ be sufficiently close to $\Omega$ and $R_0$ be the approximate solution given in Section \ref{Sec:Prelim}. Since all the functionals are continuous in $\Omega$ and the non-convexity and fixed point conditions are open, we can apply Theorem \ref{main_theorem} to $(\tilde{\Omega}, R_0)$ if $|\Omega - \tilde{\Omega}|$ is small enough. 
\end{proof}

%\color{black}

We summarize the main steps of the proof of Theorem \ref{main_theorem} below:

The first step is to recast the problem as a perturbative one, close to an explicit approximate even, $2\pi/m$ periodic solution $R_0$. In practice $R_0$ is an analytic function represented by a finite sum $R_0(x) = \sum_{j=1}^{N_0} a_j \cos(jmx)$ and was found doing numerical continuation of the branch emanating from the circle. Once $R_0$ is fixed, we write  $R = R_0 + \int_0^x \Tilde{u}$ where $u : [0, \pi/m] \longrightarrow \mathbb{R}$ and $\Tilde{u}$ is the odd, $2\pi/m$ periodic extension of $u$. We will plug this expression into \eqref{eqR}, and solve for $u$ (cf. \eqref{equ}). This setup ensures that we are modding out the scaling and rotation symmetries of the equation and maintain the $m$-fold symmetry. Conversely, we prove that if $u$ solves equation \eqref{equ}, then $R = R_0 + \int_0^x \Tilde{u}$ is a solution of \eqref{eqR}.

Looking at \eqref{equ}, the equation that $u$ satisfies is of the form  $$Lu(x) = N_L[u](x) + \delta(x)$$ for all $x \in [0, \pi/m]$, with $Lu := u(x) + \int K(x,y) u(y) dy$, $\delta(x)$ proportional to the error of the approximate solution and $N_L$ being non-linear terms (cf. Definitions \ref{Lu_def}, \ref{NL_def}). Note that $K$ depends on $R_0$. The main idea is to use a fixed-point argument and set the system $u = L^{-1}( N_L[u] + \delta)$, provided that $L$ is invertible and $\delta$ is small. A simple bound of $\|L^{-1}\|$ can be given by $(1-\|K\|)^{-1}$ if $\|K\| < 1$. This type of bounds was used in the papers \cite{Enciso-GomezSerrano-Vergara:convexity-cusped-whitham,Dahne-GomezSerrano:highest-wave-burgershilbert,Dahne:highest-cusped-waves-fkdv}, however we will not have that luxury here since $\| K\|_\infty > 2$, and even $\| K\|_2 > 1.3$. Instead we approximate $K$ by a finite rank approximation $K_F$ and consider the operator $L_F u = u(x) + \int K_F(x,y)u(y) dy$. Bounding $L_F^{-1}$ and the difference $Lu-L_Fu$ allows us to invert $L$ as a Neumann series and bound $L^{-1}$. This is the first part of Section \ref{Sec:Fixed_Point}. In order to apply the fixed point theorem we need to show that $N_L$ satisfies a Lipschitz condition in a ball of radius $\varepsilon$. This  is accomplished in the second part of Section \ref{Sec:Fixed_Point} and Section \ref{Sec:Estimates}. Finally, using a bootstrap argument we upgrade the regularity of $R$ from $H^1$ to $C^\infty$, which in turn is upgraded subsequently to analytic using a free boundary elliptic problem approach.  Moreover, due to the strict non-convexity of the patch parametrized by $R_0$ and the smallness of $u$ we can conclude that the true solution $R$ is non-convex as well. This is all done in Section \ref{Sec:Regularity}.

\subsection{Acknowledgements}

We acknowledge the CFIS Mobility Program for the partial funding of this research work, particularly Fundaci\'o Privada Mir-Puig, CFIS partners, and donors of the crowdfunding program.
GCL has been partially supported by a MOBINT grant from Generalitat de Catalunya and by an Erasmus+ grant co-funded by the European Union. 
 JGS  has been partially supported
by the MICINN (Spain) research grant number PID2021–125021NA–I00,  and by the AGAUR project 2021-SGR-0087 (Catalunya). GCL and JGS were partially supported by NSF under grants DMS-2245017 and DMS-2247537. We thank Brown University  for computing
facilities. We also wholeheartedly thank the Department of Mathematics at Brown University for hosting GCL during the development of this project. Part of this research was conducted using computational resources and services at the
Center for Computation and Visualization, Brown University. 

\section{Preliminaries}
\label{Sec:Prelim}

Throughout the rest of the paper we will work with the following approximate solution $R_0$:
\begin{align}\label{def_R0_m_N0_Omega_epsilon}
        R_0(x) := \sum_{k = 0}^{N_0} c_k \cos(kmx) \quad \text{with } \ m = 6, \ N_0 = 30 ,   
\end{align}
where the coefficients $\{c_k\}$ are explicit, given in Appendix \ref{ck_coefs}. Moreover, we will fix $\Omega = \frac{1537}{3750}$, and $\varepsilon = 2 \cdot 10^{-5}$.

% \begin{definition}
%     Let the approximate solution $R_0$ be defined by the following truncated fourier series
%     \begin{align*}
%         R_0(x) := \sum_{k = 0}^{N_0} c_k \cos(kmx) \quad \text{with } \ m = 6, \ N_0 = 30 ,   
%     \end{align*}
%     where the coefficients $\{c_k\}_0^{N_0}$ are explicit numbers given in the appendix \ref{ck_coefs}.
% \end{definition}

% \begin{definition}\label{def_Omega}
%     Let the angular velocity be $\Omega = 0.409866666666667$.
%     \color{red} segur?? \color{black}
% \end{definition}

% \begin{definition}\label{def_epsilon}
%     Let $\epsilon = 3 \cdot 10^{-5}$.
% \end{definition}
% \color{red} mirar esto
% \color{black}
% We announce a Lemma that is going to be proved later in \ref{boundR0}.
% \begin{restatable*}{lemma}{mRMR} \label{boundR0}
%     The approximate solution is bounded by above and below by $m_{R_0}< R_0(x) < M_{R_0} $, with $m_{R_0} = 0.941$ and $M_{R_0} = 1.0925$.
% \end{restatable*}

\subsection{Functional Spaces setup}
Let us now define the following spaces:
\begin{align*}
    X_m := L^2 \left(\left [0, \frac{\pi}{m} \right ] \right ),  \quad X_m^{\epsilon} := \left \{ u \in X_m : \| u \|_{L^2 \left ([0, \frac{\pi}{m}]\right )}\leq \epsilon \right \}. 
\end{align*}
For any function $u \in X_m$ let $\Tilde{u}$ be its odd, $2\pi/m$ periodic extension. For these extensions we define the following spaces: 
\begin{align*}
    \Tilde{X}_m^{odd} := \left \{ w \in L^2 \left ( \mathbb{T} \right ) : w(-x)= -w(x) \text{ and } w(x+2\pi/m) = w(x) \right \}, \\
    \Tilde{X}_m^{odd, \epsilon} := \left \{ w \in  \Tilde{X}_m^{odd} : \frac{1}{\sqrt{2m}}\| w \|_{L^2 (\mathbb{T})} \leq \epsilon \right \}.
\end{align*}

We will work with perturbations of $R_0$ given $v = \int_0^x \Tilde{u}$ which belong to $H^1(\mathbb{T})$, but we can also define a more precise space where they belong:
\begin{equation*}
    \Tilde{Y}_m^{even, \epsilon} := \left \{ v(x) = \int_0^x w(y)dy : w \in  \Tilde{X}_m^{odd, \epsilon} \right \}.
\end{equation*}
% \begin{rmk}
%     If $v \in \Tilde{Y}_m^{even, \epsilon}$ then by the Sobolev embedding $\Y v \leq \epsilon\sqrt{\frac{\pi}{m}} $
% \end{rmk}

\subsection{Equation for the perturbation}
In this subsection we will write the equation that $u$ satisfies, where $R = R_0 + \int_0^x \Tilde{u}$, and $R$ solves $\eqref{eqR}$. Here we will only do formal computations to motivate the equation $Lu = N_L[u] + \frac{E_0}{K_1}$, and later  we will prove the converse, namely that if $u$ satisfies $Lu = N_L[u] + \frac{E_0}{K_1}$ then $R$ is a solution of the original equation.

Denoting $v = \int_0^x \Tilde{u}$ and substituting $R = R_0+ v$ in equation \eqref{eqR}, the equation for $v$ becomes 
\begin{equation}\label{eqv}
    T[v] = E[v]+ N[v]
\end{equation}
with $T, E, N$ given by 
    \begin{equation}\label{OPv} 
        \begin{aligned}
            T[v](x) &\coloneqq 
            v'(x)\left (  R_0(x) + \int \left ( R_0(y)C[v] - R_0'(y)S[v] \right ) dy \right )+
            v(x) \left ( R_0'(x) + \int \left ( -R_0'(y)C[v] - R_0(y)S[v] \right ) dy \right ) \\
            &+ \int \left ( -R_0(x)S[v] +R_0'(x)C[v] \right ) v(y)dy + \int \left (- R_0'(x)S[v] - R_0(x)C[v] \right ) v'(y)dy, \\
            \\
            E[v](x) &\coloneqq -R_0(x)R'_0(x) +   
            \int C[v](R_0(x)R_0'(y) - R_0(y)R_0'(x) )dy +   
            \int S[v] (R_0(x)R_0(y) + R_0'(x)R_0'(y)) dy, \\
            \\
            N[v](x) &\coloneqq -v(x)v'(x) +  
            \int C[v] (v(x)v'(y)- v(y)v'(x))dy + 
            \int S[v] (v(x)v(y) + v'(x)v'(y)) dy,
        \end{aligned}
    \end{equation}
    where
    \begin{align}
        A[v](x,y) &\coloneqq (R_0(x)-R_0(y) + v(x)-v(y))^2 + 4(R_0(x)+v(x))(R(y)+v(y))\sin^2 \left ( \frac{x-y}{2}\right ) \label{defA},
        \\
        C[v](x,y) &\coloneqq \frac{1}{4\pi\Omega} \cos(x-y) \log (A[v](x,y)), \quad
        S[v](x,y) \coloneqq \frac{1}{4\pi\Omega} \sin(x-y) \log (A[v](x,y)). \label{defCS}
    \end{align}
\begin{rmk}
    We can also write $A[v]$ as 
    \begin{equation}\label{defA2}
        A[v](x,y) = (R_0(x)+v(x))^2 + (R_0(y)+v(y))^2 - 2(R_0(x)+v(x))(R(y)+v(y))  \cos(x-y). 
    \end{equation}
\end{rmk}

%For this functionals to make sense we need $R = R_0+v$ to be bounded and positive (remember it is a radial component in polar coordinates).

We have the following bounds on the approximate solution and on $R$:
\begin{lemma} \label{boundR0}
    The approximate solution is bounded by above and below by $m_{R_0}< R_0(x) < M_{R_0} $, with $m_{R_0} := 0.941$ and $M_{R_0} := 1.0925$. Let $|\lambda| \leq 1$ and let $R_{\lambda} = R_{0} + \lambda v$, where $v \in \Tilde{Y}_m^{even, \epsilon}$. Then we have the bound $m_R < R_{\lambda}(x) < M_R$ with $m_R := 0.94, M_R := 1.1$.
\end{lemma}
\begin{proof}
The proof of the statement about $R_0$ is computer-assisted and can be found in the supplementary material. We refer to Appendix \ref{appendix:implementation} for the implementation details. The proof of the statement about $R_\lambda$ follows from $\Y v \leq \epsilon \sqrt{\pi/m}$.
\end{proof}

\begin{lemma}
\label{Ineq_NonConvexity}
The approximate solution $R_0$ satisfies the following condition:
\begin{equation}\label{Ineq_NonConvexty_eq}
R_0(\pi/m)+\epsilon \sqrt{\pi/m} < R_0(0) \cos(\pi/m)
\end{equation}
\end{lemma}
\begin{proof}
The proof is computer-assisted and can be found in the supplementary material. We refer to Appendix \ref{appendix:implementation} for the implementation details.
\end{proof}

Moreover, we also have the following bounds on the derivatives of $R_0$:

\begin{lemma}\label{MdR0}
    The derivative of the approximate solution is bounded by $|R_0'(x)| < M_{R_0'}$ with $M_{R_0'} = 0.52$.
\end{lemma}

\begin{proof}
The proof is computer-assisted and can be found in the supplementary material. We refer to Appendix \ref{appendix:implementation} for the implementation details.
\end{proof}

\begin{lemma}\label{MddR0}
    The second derivative of the approximate solution is bounded by $|R_0''(x)| < M_{R_0''}$ with $M_{R_0''} = 8.7$.
\end{lemma}

\begin{proof}
The proof is computer-assisted and can be found in the supplementary material. We refer to Appendix \ref{appendix:implementation} for the implementation details.
\end{proof}

\begin{lemma}\label{MdddR_MddddR}
    The third and fourth derivatives of the approximate solution are uniformly bounded by the following constants $|R_0^{(k)}| < M_{R_0^{(k)}}$ for $k = 3,4$ with $ M_{R_0^{(3)}} = 106, M_{R_0^{(4)}} = 4000$.
\end{lemma}
\begin{proof}
    The proof is computer-assisted and can be found in the supplementary material. We refer to Appendix \ref{appendix:implementation} for the implementation details.
\end{proof}

We also have the following bounds on the defect and its derivatives:
\begin{lemma}\label{CE0}
    The approximate solution error satisfies $\Xm{E[0](x)} = \Xm{R_0R_0'- F_1[R_0]-F_2[R_0]} \leq C_{E_0}$, with $C_{E_0} := 3\cdot 10^{-8}$. 
\end{lemma}

\begin{proof}
The proof is computer-assisted and can be found in the supplementary material. We refer to Appendix \ref{appendix:implementation} for the implementation details.
\end{proof}

\begin{lemma}\label{ddE0_bound}
    The second derivative of the error of the approximate solution is uniformly bounded by 50: $|E[0]''(x)| < 50$.
\end{lemma}
\begin{proof}
    The proof is computer-assisted and can be found in the supplementary material. We refer to Appendix \ref{appendix:implementation} for the implementation details.
\end{proof}

% \begin{lemma}\label{mRMR_not0_bound}
%     If $R = R_0 + v$ with $v \in \Tilde{Y}_m^{even, \epsilon}$ then it is bounded by above and below $m_R < R(x) < M_R$ with $m_R = 0.94, M_R = 1.1$.
% \end{lemma}
% \begin{proof}
%     It immediately follows from the bounds $m_{R_0}, M_{R_0}$ on $R_0$ \eqref{boundR0}, and that $\Y v \leq \epsilon \sqrt{\pi/m}$.
% \end{proof}

% \color{red}

% Probably remove -- leave it for now just in case 

% \begin{lemma} \label{Abounds}
%     If $v \in \Tilde{Y}_m^{even, \epsilon}$ the functional $A[v](x,y)$ satisfies the following bounds 
%     \begin{enumerate} 
%         \item $A[v](x,y) \geq 4 m_R^2 \sin^2 \left ( \frac{x-y}{2}\right )$,
%         \item $A[v](x,y) \geq (R(x)-R(y))^2 $,
%         \item $A[v](x,y) \geq 4 \left |(R(x)-R(y))m_R \sin \left ( \frac{x-y}{2}\right ) \right |$,
%         \item $A[v](x,y) \leq (M_R-m_R)^2 + 4 M_R^2 \sin^2 \left ( \frac{x-y}{2}\right ) $,
%         \item $A[v](x,y) \leq 4 M_R^2$,
%     \end{enumerate}
%     where $R = R_0+v$.
% \end{lemma}
% \begin{proof}
%     The first, second  and fourth inequality follow straight from the first definition of $A[v]$ \eqref{defA}. For the third one we need to use its definition and that $a^2+b^2 \geq 2ab$. The fifth one follows from the second definition of $A[v]$ \eqref{defA2}.
% \end{proof}

% \color{black}

In our setting where $R_0$ is $2\pi/m$ periodic and even, we first prove that the equation preserves those properties and hence it is enough to look at $x \in [0,\pi/m]$.%if we perturb maintaining the symmetries the equation is fully determined by the $[0, \pi/m]$ interval. This is a consequence of next proposition. 
We start with the following lemma:

\begin{lemma} \label{Asimmetries}
    If $v \in \Tilde{Y}_m^{even, \epsilon}$ the functional $A[v](x,y)$ satisfies the following properties
    \begin{enumerate} 
        \item $A[v](x+ \frac{2\pi}{m}, y) = A[v](x, y-\frac{2\pi}{m})$,
        \item $A[v](-x,y) = A[v](x,-y)$.
    \end{enumerate}
\end{lemma}
\begin{proof}
By the second definition of $A$ \eqref{defA2}, $A[v](x,y)$ only depends on $R_0(x), R_0(y), v(x), v(y), \cos(x-y)$. Both $R_0$ and $v$ are $2\pi/m$ periodic and even, and the lemma follows.%the first one by definition and the second one by hypothesis of the Lemma. Using these properties we only need to write down the definition of $A[v]$ to see both equality's.
\end{proof}

\begin{proposition}\label{TEN_symmetries}
    If $v \in \Tilde{Y}_m^{even, \epsilon}$ the functionals $T, E, N$ maintain the symmetries of $v'$
    \begin{enumerate}
        \item $T[v], E[v], N[v]$ are $2\pi/m$ periodic.
        \item $T[v], E[v], N[v]$ are odd.
    \end{enumerate}
\end{proposition}
\begin{proof}
By definition \eqref{OPv}, the three functions $T[v], E[v], N[v]$ can be written in the following form
\begin{eqnarray*}
    \sum_{i,j = 0}^1 a_{i,j} f_i(x)f_j'(x) + \int_0^{2\pi} \log(A[v](x,y)) \left ( \cos(x-y) (b_{i,j}f_i'(x)f_j(y) + c_{i,j} f_i(x)f_j'(y))\right. \\
    \left.+ \sin(x-y) (d_{i,j}f_i(x)f_j(y) + e_{i,j} f_i'(x)f_j'(y))\right )dy,
\end{eqnarray*}
where $a, b, c, d, e$ are some matrices of coefficients and $f_0 = R_0, f_1 = v$. We are going to prove both items for functions of this form.

Since $R_0, v$ are $2\pi/m$ periodic, $f_i, f_i'$ are $2\pi/m$ periodic and $f_i$ is even and $f_i'$ is odd, implying that the terms $a_{i,j}f_i(x)f_j'(y)$ are $2\pi/m$ periodic, odd and moreover
\begin{align*}
    &f_i(x+2\pi/m)f_j'(x+ 2\pi/m) = f_i(x)f_j'(x), \\
    &f_i(-x)f_j'(-x) = f_i(x)f_j'(-x) = -f_i(x) f_j'(x).
\end{align*}
For the integral term let us define the following auxiliary functions
\begin{equation*}
    F_{i,j}(x,y) =  \cos(x-y) (b_{i,j}f_i'(x)f_j(y) + c_{i,j} f_i(x)f_j'(y)) + \sin(x-y) (d_{i,j}f_i(x)f_j(y) + e_{i,j} f_i'(x)f_j'(y))
\end{equation*}
Using the again the periodicity of $f_i, f_i'$ we have
\begin{align*}
    F_{i,j}(x+2\pi/m, y) &= \cos(x+ \frac{2\pi}{m}-y) \left (b_{i,j}f_i'(x+\frac{2\pi}{m})f_j(y) + c_{i,j} f_i(x+\frac{2\pi}{m})f_j'(y)\right )  \\
    &+\sin(x+\frac{2\pi}{m}-y) \left (d_{i,j}f_i(x+\frac{2\pi}{m})f_j(y) + e_{i,j} f_i'(x+\frac{2\pi}{m})f_j'(y)\right) \\
    &= \cos(x-(y- \frac{2\pi}{m})) \left (b_{i,j}f_i'(x)f_j(y- \frac{2\pi}{m}) + c_{i,j} f_i(x)f_j'(y- \frac{2\pi}{m}) \right )  \\
    &+\sin(x-(y- \frac{2\pi}{m})) \left (d_{i,j}f_i(x)f_j(y- \frac{2\pi}{m}) + e_{i,j} f_i'(x)f_j'(y- \frac{2\pi}{m})\right)  = F_{i,j}(x, y- \frac{2\pi}{m}),
\end{align*}
and using the evenness and oddness of $f_i, f_i'$ respectively
\begin{align*}
     F_{i,j}(-x,y) &=  \cos(-x-y) \left (b_{i,j}f_i'(-x)f_j(y) + c_{i,j} f_i(-x)f_j'(y) \right ) \\
     &+ \sin(-x-y)  \left ( d_{i,j}f_i(-x)f_j(y) + e_{i,j} f_i'(-x)f_j'(y) \right ) \\
     &=  \cos(x+y)) \left (-b_{i,j}f_i'(x)f_j(y) - c_{i,j} f_i(x)f_j'(-y) \right )  \\
     &-\sin(x+y)  \left ( d_{i,j}f_i(x)f_j(-y) + e_{i,j} f_i'(x)f_j'(-y) \right ) \\
     &= -F_{i,j}(x,-y).
\end{align*}
Using the properties we just proved for $F_{i,j}$ and the properties of Lemma \ref{Asimmetries} for $A[v]$, changing variables inside the integral we see that
\begin{align*}
    \sum_{i,j = 0}^1 \int_{0}^{2\pi} \log\left (A[v](x+2\pi/m,y)\right ) F_{i,j}\left (x+2\pi/m,y\right ) dy &= \sum_{i,j = 0}^1 \int_{0}^{2\pi} \log\left (A[v](x,y-2\pi/m)\right ) F_{i,j}\left (x,y-2\pi/m\right ) dy  \\
    &= \sum_{i,j = 0}^1 \int_{0}^{2\pi} \log\left (A[v](x,z)\right ) F_{i,j}\left (x,z \right ) dz, 
\end{align*}
and
\begin{align*}
    \sum_{i,j = 0}^1 \int_{0}^{2\pi} \log\left (A[v](-x,y)\right ) F_{i,j}\left (-x,y\right ) dy &= \sum_{i,j = 0}^1 \int_{0}^{2\pi} \log\left (A[v](x,-y)\right )(- F_{i,j}\left (x,-y\right ) )dy \\
    &= - \sum_{i,j = 0}^1 \int_{0}^{2\pi} \log\left (A[v](x,z)\right ) F_{i,j}\left (x,z\right ) dz.
\end{align*}
\end{proof}
In order to transform equation \eqref{eqv} into a fixed point equation, we will extract the linear terms. At this stage we have split the equation into the terms $T, E, N$, due to their different natures. In particular, $E[0]$ is the defect made by the approximate solution, and $N$ is purely nonlinear.
%The functional $E$ evaluated at $0$, $E[0]$, is the error of the approximate solution $R_0$. The functional $T$ is the "natural" linear terms we would expect from using the distributive property in $R= R_0+v$, and the functional $N$ contains only $o(v^2)$ terms. 
We now define the linearized operators $T_1, E_1$ as the Fréchet derivatives of $T, E$ at $0$.%, this is going to be clear when we do the non linear estimates $T-T_1$, $E-E[0]-E_1$ in Section \ref{Sec:Estimates}.

\begin{definition} \label{T1E1_def}
    The linearized operators $T_1, E_1$ are
    \begin{align*}
        T_1[v](x) &\coloneqq 
        v'(x)\left (  R_0(x) + \int_0^{2\pi} \left ( R_0(y)C[0] - R_0'(y)S[0] \right ) dy \right )\\
        &+
        v(x) \left ( R_0'(x) + \int_0^{2\pi} \left ( -R_0'(y)C[0] - R_0(y)S[0] \right ) dy \right ) \\
        &+ \int_0 ^{2\pi}\left ( -R_0(x)S[0] +R_0'(x)C[0] \right ) v(y)dy + \int_0^{2\pi} \left (- R_0'(x)S[0] - R_0(x)C[0] \right ) v'(y)dy, \\
        \\
        E_1[v](x) &:= %\int \left . \d_\lambda C[\lambda v] \right |_{\lambda = 0} (R_0(x)R_0'(y) - R_0(y)R_0'(x) )dy +   
            %\int \left . \d_\lambda S[\lambda v] \right |_{\lambda = 0}  (R_0(x)R_0(y) + R_0'(x)R_0'(y)) dy = \\
        \frac{1}{4\pi \Omega}\int_0^{2\pi}  \left [ \frac{\cos(x-y) \left ( R_0(x)R_0'(y) - R_0(y)R_0'(x) \right ) + \sin(x-y) \left ( R_0(x)R_0(y) + R_0'(x)R_0'(y) \right )}{A[0](x,y)}  \right ] \\
        & \times \left [ \left ( 2(R_0(x)-R_0(y)) + 4R_0(y) \sin^2\left ( \frac{x-y}{2}\right ) \right )v(x) \right. \\
        & \left.+ \left (-2(R_0(x)-R_0(y)) + 4R_0(x) \sin^2\left ( \frac{x-y}{2}\right ) \right )  v(y) \right ] dy \\
        &= : v(x) \int_0^{2\pi} -M_1(x,y) dy + \int_0^{2\pi} -M_2(x,y) v(y) dy.
    \end{align*}
    It is clear from the definition that if $v \in \Tilde{Y}_m^{even, \epsilon}$, then $T_1[v], E_1[v]$ are $2\pi/m$ periodic and odd.
\end{definition}

% \begin{corollary}\label{T1E1_symmetries}
%     If $v \in \Tilde{Y}_m^{even, \epsilon}$, then $T_1[v], E_1[v]$ are $2\pi/m$ periodic and odd.
% \end{corollary}
% \begin{proof}
%     This is a consequence of Proposition \ref{TEN_symmetries}. $T_1, E_1$ are the Fréchet derivatives of $T,E$ at $0$, so they inherit the oddness and $2\pi/m$ periodicity.
% \end{proof}

We can rewrite equation \eqref{eqv} into 
% \begin{equation}\label{eqvLin}
%     T_1[v]-E_1[v] = E[v]-E_1[v] + N[v] - (T[v]-T_1[v]).
% \end{equation}
% \color{red} new suggested split
\begin{equation}\label{eqvLin}
    T_1[v]-E_1[v] = (E[v]-E_1[v] - E[0]) + N[v] - (T[v]-T_1[v]) + E[0].
\end{equation}
Let us note that every summand on the RHS is nonlinear in $v$, with the exception of the last one which can be bounded by a small constant.
In order to invert the linear operator on the LHS we have to use the symmetries in $v$ that come from $\Tilde{u}$ and write it in terms of an operator acting on $u$. The choice of $u$ is motivated due to the need to mod out the dilating and rotating symmetries, which would yield multiple solutions, in a way that considering the problem in terms of $u$ there is only one solution near $u=0$. 
We now rewrite $T_1$ and $E_1$ taking into account the aforementioned symmetries.
\begin{proposition} \label{TEsym}
If $v \in \Tilde{Y}_m^{even, \epsilon}$ then 
\begin{align*}
    T_1[v](x) &= 
    v'(x)\left (  R_0(x) + \int_{0}^{2\pi} \left ( R_0(y)C[0] - R_0'(y)S[0] \right ) dy \right ) \\
    & +
    v(x) \left ( R_0'(x) + \int_{0}^{2\pi} \left ( -R_0'(y)C[0] - R_0(y)S[0] \right ) dy \right ) \\
    &+ \int_{0}^{\frac{\pi}{m}} \left ( -R_0(x)S^{s,e}(x,y) +R_0'(x)C^{s,e}(x,y) \right ) v(y)dy + \int_{0}^{\frac{\pi}{m}} \left (- R_0'(x)S^{s,o}(x,y) - R_0(x)C^{s,o}(x,y) \right ) v'(y)dy, \\
    \\
    E_1[v](x) &= v(x) \int_0^{2\pi} -M_1(x,y) dy + \int_{0}^{\frac{\pi}{m}} -M_2^{s,e}(x,y) v(y) dy,
\end{align*}
where
\begin{align*}
    C^{s,e}(x,y) &\coloneqq \sum_{j = 0}^{m-1} C[0](x,y+\frac{2\pi j}{m}) + C[0](x,-y-\frac{2\pi j}{m}), \\
    C^{s,o}(x,y) &\coloneqq \sum_{j = 0}^{m-1} C[0](x,y+\frac{2\pi j}{m}) - C[0](x,-y-\frac{2\pi j}{m}), \\
    S^{s,e}(x,y) &\coloneqq \sum_{j = 0}^{m-1} S[0](x,y+\frac{2\pi j}{m}) + S[0](x,-y-\frac{2\pi j}{m}), \\
    S^{s,o}(x,y) &\coloneqq \sum_{j = 0}^{m-1} S[0](x,y+\frac{2\pi j}{m}) - S[0](x,-y-\frac{2\pi j}{m}), \\    
    M_2^{s,e}(x,y) &\coloneqq \sum_{j = 0}^{m-1} M_2(x,y+\frac{2\pi j}{m}) + M_2(x,-y-\frac{2\pi j}{m}).
\end{align*}
The first superscript $s$ denotes the symmetrized kernel and the second one ${e,o}$ refers to the parity respect to the second variable: \{even, odd\}.
\end{proposition}
 \begin{proof}
 The proof follows from the symmetries of $v,v'$ and their periodicity.
%     The proof is just changing variables on the integrals and using the symmetries $v(-y) = v(y)$, $v'(-y) = -v'(y)$ and periodicity $v(y+2\pi/m) = v(y)$, $v'(y+2\pi/m) = v'(y)$.
 \end{proof}

We further simplify the notation of the linear part of equation \eqref{eqvLin} by writing, for $v \in \Tilde{Y}_m^{even, \epsilon}$:
\begin{align*}
   T_1[v](x) - E_1[v](x) = K_1(x)v'(x) + K_2(x)v(x) + \int_0^{\frac{\pi}{m}} K_3^{s,e}(x,y)v(y) dy + \int_0^{\frac{\pi}{m}} K_4^{s,o}(x,y)v'(y) dy,
\end{align*}
where 
\begin{align}
    K_1(x) &\coloneqq R_0(x) + \int_0^{2\pi} R_0(y)C[0](x,y) - R_0'(y)S[0](x,y) dy, \label{defK1} 
    \\
    K_2(x) &\coloneqq R_0'(x) + \int_0^{2\pi} -R_0'(y)C[0](x,y) - R_0(y)S[0](x,y) + M_1(x,y)dy, \label{defK2} 
    \\
    K_3^{s,e}(x,y) &\coloneqq -R_0(x)S^{s,e}(x,y) +R_0'(x)C^{s,e}(x,y) + M_2^{s,e}(x,y), \label{defK3} \\
    K_4^{s,o}(x,y) &\coloneqq - R_0'(x)S^{s,o}(x,y) - R_0(x)C^{s,o}(x,y). \label{defK4} 
\end{align}

% \begin{definition}\label{T1-E1def}
%     If $v \in \Tilde{Y}_m^{even, \epsilon}$, the linear part of the equation \eqref{eqvLin} can be written as 
%     \begin{align*}
%         \mathbf{(T_1-E_1)[v](x)} = K_1(x)v'(x) + K_2(x)v(x) + \int_0^{\frac{\pi}{m}} K_3^{s,e}(x,y)v(y) dy + \int_0^{\frac{\pi}{m}} K_4^{s,o}(x,y)v'(y) dy,
%     \end{align*}
%     with
%     \begin{align*}
%         K_1(x) &\coloneqq R_0(x) + \int R_0(y)C[0](x,y) - R_0'(y)S[0](x,y) dy, \\
%         \\
%         K_2(x) &\coloneqq R_0'(x) + \int -R_0'(y)C[0](x,y) - R_0(y)S[0](x,y) + M_1(x,y)dy, \\ 
%         \\
%         K_3^{s,e}(x,y) &\coloneqq -R_0(x)S^{s,e}(x,y) +R_0'(x)C^{s,e}(x,y) + M_2^{s,e}(x,y),\\
%         \\
%         K_4^{s,o}(x,y) &\coloneqq - R_0'(x)S^{s,o}(x,y) - R_0(x)C^{s,o}(x,y). \\
%     \end{align*}
% \end{definition}

%\color{blue} 
%The last step is to write the linear operator $T_1-E_1$ in terms of $u$ only, thanks to Proposition \eqref{TEsym} all the integration intervals where $\Tilde{u}$ appears are reduced to $[0, \pi/m]$ so we won't need the extension $\Tilde{u}$. The idea behind the definition of this operator $L$ is that $K_1 Lu  = (T_1-E_1)[\int_0^x \Tilde{u}]$ for $x \in [0, \pi/m]$. We are dividing by $K_1$ so we have the operator in the form Identity + Compact and we can invert it.

%\color{red} Suggested:

The last step is to write the linear operator in terms of $u$ (and not $\Tilde{u}$) only, by virtue of Proposition \ref{TEsym}. After performing this reduction, we finish off by dividing by $K_1$ to get a linear operator of the form Identity $+$ Compact and be able to invert it.

\color{black}
\color{black}

%We announce a Lemma proved later in \ref{K1bound} to emphasize that dividing by $K_1(x)$ in the following definitions  \ref{Lu_def} \ref{NL_def} is not a problem.

In order to perform the division by $K_1$, we have the following Lemma:

\begin{lemma}
\label{K1bound}
    The function $K_1(x)$  satisfies the bound $C_{K_1} < |K_1(x)|$, with $C_{K_1} := 0.1$.
\end{lemma}
\begin{proof}
The proof is computer-assisted and can be found in the supplementary material. We refer to Appendix \ref{appendix:implementation} for the implementation details.
\end{proof}

We finish this section with the rest of the operator notation, and the setup of the final equation for $v$.

\begin{definition} \label{Lu_def}
Given $u \in X_m$ we define $L u$ as
\begin{align*}
    Lu(x) = u(x) + \int_0^{\frac{\pi}{m}} K(x,y) u(y) dy,    
\end{align*}
where
\begin{align*}
    K(x,y) = \frac{1}{K_1(x)} \left (K_2(x) \mathbb{I}_{0\leq y \leq x} + K_3^{s,o}(x,y) + K_4^{s,o}(x,y) \right ),
\end{align*}
with 
\begin{align*}
        K_3^{s,o}(x,y) &= \int_y^{\frac{\pi}{m}} K_3^{s,e}(x,z) dz,
\end{align*}
for all $x \in [0, \pi/m]$.
\end{definition}

%The nonlinear part $N_L[u]$ of the equation for $u$ is defined through its extension.
%\color{red} don't understand what "extension" exactly the previous sentence talks about? \color{black} \color{green} It refers to $\Tilde{u}$, maybe change the sentence, not well explained \color{black}

\begin{definition}\label{NL_def}
Given $u \in X_m^\epsilon$ we define $N_L[u]$ as
\begin{align*}
    N_L[u](x) &\coloneqq \frac{1}{K_1(x)} \left ( T_1[v] - T[v] + (E[v]-E_1[v] - E[0]) + N[v]\right ), 
\end{align*}
with $\displaystyle v(x) = \int_0^x \Tilde{u}(y)dy$.
\end{definition}

Finally, the reduced problem is to find a $u \in X_m^\epsilon$ such that
\begin{equation}\label{equ}
    Lu (x) = N_L[u](x) + \frac{E[0](x)}{K_1(x)} \quad \forall x \in \left [0, \frac{\pi}{m} \right ].
\end{equation}

\section{Solving the Fixed Point equation}\label{Sec:Fixed_Point}
In this section we will prove through a fixed point theorem argument that there exists a (unique) solution to \eqref{equ}. We will also show that any solution $u$ to \eqref{equ} will yield a solution $R$ of \eqref{eqR} after setting $R = R_0 + \int_{0}^{x}\Tilde{u}(y) dy$.
The main theorem is divided int two separate parts: first we have to invert the linear operator $L$, and second we have to do nonlinear estimates in $N_L+ \frac{E[0]}{K_1}$ to show $L^{-1}(N_L[\cdot]+ \frac{E[0]}{K_1})$ is a contraction in an adequate ball. The main theorem of this section is the following and its proof will be postponed to the end of the section:

\begin{theorem}\label{Gcontractive}
    Let $G := L^{-1} (N_L[\cdot]+ \frac{E[0]}{K_1}) : X_m^{\epsilon} \rightarrow X_m^{ \epsilon}$, then
    \begin{enumerate}
        \item $G$ is well defined: $G(X_m^\epsilon) \subseteq X_m^\epsilon$.
        \item $G$ is contractive.
    \end{enumerate}
\end{theorem}
Then, via Banach's fixed point theorem we have:

 \begin{corollary}
     There exists a unique $u \in X_m^\epsilon$ that satisfies $u = G[u]$.
 \end{corollary}

\subsection{Inverting the linear operator}
The main idea to prove the invertibility of $L$ is to first approximate $L$ by $L_F = $ Identity + Finite Rank, then prove the invertibility of $L_F$ and finally prove that the approximation error $L-L_F$ is small enough, making $L$ invertible via a Neumann series.

\begin{definition}
    Let $\{e_n(x)\}_n$ be the normalized Fourier basis of $X_m = L^2([0, \frac{\pi}{m}])$, more precisely
    \begin{align*}
        e_1(x) = \sqrt{\frac{m}{\pi}}, \quad e_{2k}(x)  = \sqrt{\frac{2m}{\pi}} \sin(2mkx), \  k \geq 1, \quad e_{2k+1} (x) = \sqrt{\frac{2m}{\pi}} \cos(2mkx),  \ k \geq 1. 
    \end{align*}
    Let $N = 201$, and $E_N = \text{span} \{e_n\}_{n=1}^N$ be the subspace generated by the first $N$ vectors. Similarly, let $E_N^\perp$ be its orthogonal subspace. We will also define $L_F = I+ \mathcal{K}_F: X_m \mapsto X_m$ where $\mathcal{K}_F : E_N \rightarrow E_N$ is given by
\begin{equation}\label{LF_def}
        \mathcal{K}_F[u] = \int_0^\frac{\pi}{m} K_F(x,y) u(y) dy \ \text{ with } \ K_F(x,y) := \sum_{k,l=1}^{N} A_{k,l} e_k(x)e_l(y), 
    \end{equation}
    and $A_{k,l}$ is the matrix defined in Appendix \ref{Akl_coefs}.
\end{definition}

\begin{lemma}\label{I+Ainvert}
    The matrix $I+A$ is invertible and satisfies $\|(I+A)^{-1}\|_2 \leq C_2$, with $C_2 := 8.8$.
\end{lemma}
\begin{proof}
The proof is computer-assisted and can be found in the supplementary material. We refer to Appendix \ref{appendix:implementation} for the implementation details.    
\end{proof}

\begin{lemma}
\label{LFinvertible}
    The operator $L_F$ is invertible and $\| L_F^{-1}\|_2 < C_2$ with $C_2 := 8.8$.
\end{lemma}
\begin{proof}
Let $P_N$ be the projection operator onto $E_N$. Using that $\mathcal{K}_F = P_N \mathcal{K}_F = \mathcal{K}_F  P_N$, $L_F$ decouples in $E_N \bigoplus E_N^\perp$ as
\begin{align*}
    L_F = 
    \begin{pmatrix}
        I_{E_N} + A & \mathbf{0} \\
        \mathbf{0} & I_{E_N}^\perp
    \end{pmatrix}
\end{align*}
then as $E_N$ is a finite dimensional vector space, to invert $I_{E_N}+A$ we have to invert the corresponding matrix, and the identity in $E_N^\perp$ is trivially inverted, so
\begin{align*}
    L_F^{-1}f := (I_{E_N}+A)^{-1} P_N f + (I-P_N)f. 
\end{align*}
We can conclude that $L_F$ is invertible. Moreover
\begin{align*}
    \X{L_F^{-1} f}  = \X{P_N L_F^{-1} f } + \X{(I-P_N) L_F^{-1} f} &\leq \| (I_{E_N} + A)^{-1}\|_2 \X{P_N f} + \X{(I-P_{N})f} \\
    &\leq \max \{ \| (I_{E_N} + A)^{-1}\|_2, 1 \} \X{f},
\end{align*}
hence its norm is bounded by $C_2$ because by Lemma \ref{I+Ainvert}, $ \| (I+A)^{-1}\|_2 \leq C_2$ and $ 1 < C_2$.
\end{proof}

\begin{lemma}\label{L-L_Fbound}
The error of approximating the operator $L$ by $L_F$ satisfies $\| L-L_F\|_2 \leq C_3$, with $C_3 := 0.085$.
\end{lemma}
\begin{proof}
The proof is computer-assisted and can be found in the supplementary material. We refer to Appendix \ref{appendix:implementation} for the implementation details.    
\end{proof}

We can now state and prove the main result of this section

\begin{proposition}\label{Linvertible}
    The linear operator $L: X_m \rightarrow X_m$ is invertible. Moreover  $\|L^{-1}\|_2 \leq C_1 = 35$.
\end{proposition}
\begin{proof}
    Using that $L_F$ is invertible, we can write
    \begin{equation*}
        L = L_F(I + L_F^{-1}(L-L_F)).
    \end{equation*}
    We can then invert $I+L_F^{-1}(L-L_F)$ using a Neumann series because due to Lemmas \ref{LFinvertible}, \ref{L-L_Fbound} we have that $\|L_F^{-1}(L-L_F)\|_2 \leq C_2 C_3 < 1$. As $L_F$ is also invertible, we can conclude that $L$ is invertible and
    \begin{align*}
        \| L^{-1}\|_2 = \| (I+ L_F^{-1}(L-L_F))^{-1} L_F^{-1}\|_2 \leq \frac{C_2}{1-C_2C_3}  = \frac{8.8}{0.252}< 35 =C_1. 
    \end{align*}
    %By definition $C_3 = \frac{1}{2C_2}$ so $C_1 = \frac{C_2}{1-C_2C_3}= 2C_2$.
\end{proof}

\subsection{Nonlinear terms}

In this subsection we will perform the nonlinear estimates on the full operator, provided we have them on each of the individual pieces. We defer those individual estimates to Section \ref{Sec:Estimates}.

\begin{proposition}\label{Final_nonlinear_bounds_NL}
    The map $(N_L[\cdot]+ \frac{E[0]}{K_1}): X_m^\epsilon \rightarrow X_m$  satisfies:
    \begin{enumerate}
        \item $\Xm{(N_L[u]+ \frac{E[0]}{K_1})} \leq \epsilon_0 + C_5 \Xm{u}^2$,
        \item $\Xm{(N_L[u_1]+ \frac{E[0]}{K_1})  - (N_L[u_0]+ \frac{E[0]}{K_1})} \leq \epsilon C_6  \Xm{u_1-u_0}$, 
    \end{enumerate}
    with $\epsilon_0 := \frac{C_{E_0}}{C_{K_1}}, C_5 := \frac{C_T+C_E+C_N}{C_{K_1}}$ and $C_6:= 2\frac{C_N'+C_T'+C_E'}{C_{K_1}}$. 
    The constants $C_T, C_E, C_N, C_N', C_T', C_E'$ will be defined in Propositions \ref{T-T_1 v}, \ref{Ev-E0-E1v}, \ref{Nv}, \ref{Nv1Nv2bound}, \ref{T-T_1v1v2}, \ref{Ev1-E1v1-Ev2+E1v2}, $C_{E_0}$ is defined in Proposition \ref{CE0} and $C_{K_1}$ is defined in  Proposition \ref{K1bound}.
\end{proposition}

\begin{proof}

    \textbf{Proof of (1): }
    Writing $v = \int_0^x \Tilde{u}$ we get
    \begin{align}
        \Xm{N_L[u]} &\leq \Y{\frac{1}{K_1}}\left (  \Xm{(T-T_1)  [v]} + \Xm{(E-E[0]-E_1)[v]} + \Xm{N[v]} \right ) \nonumber\\
        &\leq  \frac{C_T+C_E+C_N}{C_{K_1}} \Xm{u}^2 \label{boundNL}, 
    \end{align}
    where we have used that for all $v$ in $\Tilde{Y}_m^{even,\epsilon}$ we have the bounds  
    \begin{align*}
        \Xm{N[v]} \leq C_N \Xm{v'}^2, \quad 
        \Xm{E[v]-E[0]-E_1[v]} \leq C_{E}\Xm{v'}^2, \quad  \Xm{(T-T_1) [v]} \leq C_{T}\Xm{v'}^2,
    \end{align*}
    proved in Propositions \ref{Nv}, \ref{T-T_1 v}, \ref{Ev-E0-E1v}, and we have also used
    \begin{align}\label{boundE0_2}
        \Xm{E[0]} \leq C_{E_0} \quad
        |K_1(x)| \geq C_{K_1}   
    \end{align}
    proved in Lemmas \ref{CE0} and \ref{K1bound}. Combining \eqref{boundNL} and \eqref{boundE0_2} gives the desired bound.
    \\
    \\
    \textbf{Proof of (2): }
    For all $v_0, v_1$ in $\Tilde{Y}_m^{even,\epsilon}$ we have
    \begin{align*}
        \Xm{N[v_1] - N[v_0]} &\leq C'_N 2\epsilon \Xm{v_1'-v_0'}, 
        \\
        \Xm{(T-T_1) [v_1] - (T-T_1) [v_0]} &\leq C'_{T} 2\epsilon \Xm{v_1'-v_0'},
        \\
        \Xm{(E[v_1]-E[0]-E_1[v_1]) - (E[v_0]-E[0]-E_1[v_0])} &\leq C'_{E}2\epsilon\Xm{v_1'-v_0'} 
    \end{align*}
    by Propositions \ref{Nv1Nv2bound}, \ref{T-T_1v1v2}, \ref{Ev1-E1v1-Ev2+E1v2}, using $\|v_i\|_2 \leq \epsilon$ appropriately. 
    Adding the three inequalities concludes the proof.
\end{proof}

\subsection{Proof of Fixed Point Existence}
Now we prove that $G$ is contractive, Theorem \ref{Gcontractive}:
\begin{proof}[Proof of Theorem \ref{Gcontractive}]\, \\ \\
    \textbf{Proof of (1):}
    We estimate the norm of $Gu$: 
    \begin{align*}
        \Xm{Gu} \leq \Xm{L^{-1}} \Xm{(N_L+ \frac{E[0]}{K_1}) u} \leq C_1 \epsilon_0 + C_1 C_5 \Xm{u}^2.
    \end{align*}
    We need to see that $\Xm{Gu} \leq \epsilon$, which is true iff
    \begin{align*}
        C_1 \epsilon_0 + C_1 C_5 \epsilon^2 \leq \epsilon 
        \iff \frac{1-\sqrt{1-4C_1C_5C_1\epsilon_0}}{2C_1C_5} \leq \epsilon \leq \frac{1+\sqrt{1-4C_1C_5C_1\epsilon_0}}{2C_1C_5}.
    \end{align*}
    Both inequalities are satisfied because of Lemma \ref{FixedPointConstantChecking}.
    \\
    \\
    \textbf{Proof of (2): }
    We estimate $Gu_1 - Gu_2$:
    \begin{align*}
        \Xm{Gu_1 - Gu_2}  &= \Xm{L^{-1} ((N_L+ \frac{E[0]}{K_1}) u_1 - (N_L+ \frac{E[0]}{K_1}) u_2)} \leq \Xm{L^{-1}} \Xm{N_L u_1 - N_L u_2}  \\
        &\leq C_1 C_6 \epsilon \Xm{u_1 - u_2}.
    \end{align*}
    We want the Lipschitz constant to be less than 1, which is true iff
    \begin{equation*}
        C_1 C_6 \epsilon  < 1 \iff \epsilon < \frac{1}{C_1 C_6}. 
    \end{equation*}
    Our choice of $\epsilon$ also satisfies this by Lemma \ref{FixedPointConstantChecking}, so we can conclude that $G: X_m^\epsilon \rightarrow X_m^\epsilon$ is a contraction.
\end{proof}

\subsection{Unfolding the solution}

Let $\displaystyle R(x) = R_0(x) + \int_0^x \Tilde{u}$ where $u \in X_m^\epsilon$ is a solution to $u = G[u]$.
The aim of this section is to prove that $R$ is an $H^1$ solution of \eqref{eqR}.

\begin{proposition}\label{Unfolding_u_to_R}
    If $u$ is the fixed point of $G$ given by Theorem \ref{Gcontractive}, then $R = R_0+ \int_0^x \Tilde{u}$ is a solution of the main equation \eqref{eqR}.
\end{proposition}

\begin{proof}

We start by proving that if $u = Gu$ (i.e. if $Lu = N_L[u] + \frac{E[0]}{K_1(x)}$), then $v = \int_0^x \Tilde{u}$ satisfies equation \eqref{eqv}: $$T[v] = E[v] + N[v].$$

By \eqref{NL_def}
    \begin{align*}
        K_1(x)N_L[u](x) = (T_1 - T + E- E[0]-E_1 + N)[v](x) \quad \forall x \in \left [ 0, \frac{\pi}{m}\right ].
    \end{align*}
    We will prove that $K_1(x)Lu(x) = (T_1-E_1)[v](x)$ for all $x \in [0, \pi/m]$. Indeed, by \eqref{Lu_def} and integrating by parts the third term 
    \begin{align*}
        K_1(x)Lu(x) &= K_1(x)u(x) + \int_0^\frac{\pi}{m} K_2(x) \mathbbm{1}_{0\leq y\leq x} u(y) dy + \int_0^\frac{\pi}{m} \left ( \int_y^\frac{\pi}{m} K_3^{s,e}(x,z) dz\right) u(y) dy + \int_0^\frac{\pi}{m} K_4^{s,o}(x,y) u(y) dy \\
        &= K_1(x)u(x) + K_2(x)\int_0^x u(y) dy + \left ( \int_y^\frac{\pi}{m} K_3^{s,e}(x,z) dz\right) \left ( \int_0^y u(z) dz \right ) \bigg \rvert_{y=0}^{y=\frac{\pi}{m}}  \\
        &-\int_0^\frac{\pi}{m}\left (- K_3^{s,e}(x,y) \right ) \left (\int_0^y u(z) \right ) dz dy + \int_0^\frac{\pi}{m} K_4^{s,o}(x,y) u(y) dy \\
        &= K_1(x) \Tilde{u}(x) + K_2(x)\int_0^y \Tilde{u}(y) dy + \int_0^\frac{\pi}{m} K_3^{s,e}(x,y) \left (\int_0^y \Tilde{u}(z) \right ) dy + \int_0^\frac{\pi}{m} K_4^{s,o}(x,y) u(y) dy \\
        &= (T_1-E_1)[v](x) \quad \forall x \in \left [ 0, \frac{\pi}{m}\right],
    \end{align*}
    where we have used that $u(y) = \Tilde{u}(y)$ for all $y$ in $[0, \pi/m]$ and that $u$ is only evaluated within this interval in the above integrals.

    Using what we just proved
    \begin{equation*}
        (T_1-E_1)[v](x) = K_1(x)Lu(x) = K_1(x) N_L[u](x) + E[0](x) = (T_1-T + E-  E_1 + N)[v](x) \quad \forall x \in \left [ 0, \frac{\pi}{m}\right],
    \end{equation*}
    which implies
    \begin{equation*}
        T[v] = E[v] + N[v] \quad \forall x \in \left [ 0, \frac{\pi}{m} \right ],
    \end{equation*}
    and since $v$ belongs to $\Tilde{Y}_m^{even, \epsilon}$, by Proposition \ref{TEN_symmetries}, this equality can be extended to all $x$ in $\mathbb{T}$. Setting $R = R_0 + v$, a straightforward computation yields
    \begin{align*}
    RR'-F_1[R]-F_2[R] &= - E[v] + T[v] - N[v].
    \end{align*}
    Therefore, if $T[v]- E[v]-N[v] = 0$ for all $x \in \mathbb{T}$ then $RR'-F_1[R]-F_2[R] = 0$ for all $x \in \mathbb{T}$, and this concludes the proof.

%    By definition of $G$, $u = Gu$ iff $Lu = N_L[u]$. Then by Proposition \ref{Unfolding1} we know that $v = \int_0^x \Tilde{u}$ satisfies equation \eqref{eqv} $T[v] = E[v] + N[v]$. Then by Proposition \ref{Unfolding2} $R = R_0 + v$ is a solution of the original equation \ref{eqR}.
\end{proof}

\section{Nonlinear Estimates}\label{Sec:Estimates}

In this section we are going to prove the nonlinear estimates used in the previous section to close the fixed point.
\begin{lemma}\label{vLinftyv'Ltwobound}
    For all $v \in \Tilde{Y}_m^{even, \epsilon}$ we have $\Y{v} \leq \frac{\sqrt{\pi}}{m \sqrt{2}} \X{v'}$.
\end{lemma}
\begin{proof}
    Using that the $L^\infty$ norm of $v$ in $[-\pi, \pi]$ is the same as the one in $[0, \pi/m]$, because $v$ is even and $2\pi/m$ periodic: 
    \begin{equation*}
        \| v\|_{L^\infty} \leq \sqrt{\frac{\pi}{m}} \Xm{v'} = \frac{\sqrt{\pi}}{m\sqrt{2}} \Ltwo{v'}.
    \end{equation*}
\end{proof}

\begin{definition}\label{pm_def}
    For any $x \in \mathbb{R}$ let $p_m(x) := \text{dist} \left (x, \{ \frac{2\pi}{m}k\}_{k\in \mathbb{Z}} \right )$. 
\end{definition}
%\begin{rmk}\label{pm_def2}
%    Another equivalent definition is $p_m(x) := \text{dist} \left (x, \{ \frac{2\pi}{m}k\}_{k\in \mathbb{Z}} \right )$.
%\end{rmk}

\begin{lemma}\label{bound_v(x)-v(y)}
    For all $v \in \Tilde{Y}_m^{even, \epsilon}$ we have $|v(x)-v(y)| \leq \frac{\X{v'}}{\sqrt{2m}} \sqrt{|p_m(x)-p_m(y)|}$.
\end{lemma}
\begin{proof}
    By definition, $v$ is $2\pi/m$ periodic and even, so $v(x) =v(p_m(x))$. Then
    \begin{align*}
        |v(x)-v(y)| = |v(p_m(x)) - v(p_m(y))| \leq \Xm{v'} \sqrt{|p_m(x)-p_m(y)|} = \frac{\X{v'}}{\sqrt{2m}} \sqrt{|p_m(x)-p_m(y)|}
    \end{align*}
    where we have used that $p_m(x), p_m(y) \in [0, \pi/m]$.
\end{proof}

\begin{lemma}\label{bound2_v(x)-v(y)}
    For all $v \in \Tilde{Y}_m^{even, \epsilon}$ we have $|v(x)-v(y)| \leq \frac{\X{v'}}{\sqrt{2m}} \sqrt{|p_m(x-y)|}$.
\end{lemma}
\begin{proof}
    Using Lemma \ref{bound_v(x)-v(y)} we only need to prove $|p_m(x)-p_m(y)|\leq p_m(x-y)$. Since $p_m$ is $2\pi/m$-periodic, we can assume that $x, y$ belong to $[-\pi/m, \pi/m]$. Then $|p_m(x)-p_m(y)| = |\sign(x)x-\sign(y)y|$. We split the argument in two cases
    \begin{enumerate}
        \item If $\sign(x) = \sign(y)$, then $|p_m(x)-p_m(y)| = |x-y| = p_m(x-y)$ because $x-y \in [-\pi/m, \pi/m]$.
        \item If $\sign(x) \neq \sign(y)$, we may assume $x \geq 0$ and $y \leq 0$. Then
        \begin{align*}
            p_m(x-y) &= \min\{x-y, y+2\pi/m - x\} = \min \{|x| + |y|, |y+\pi/m| + |\pi/m - x|\} \geq \\
            &\geq |x+y| = |p_m(x)-p_m(y)|
        \end{align*}
        where we have used the sign of $x, y$, that $y + \pi/m \geq 0$ and that $\pi/m- x \geq 0$.
    \end{enumerate}
\end{proof}

\begin{lemma} \label{f(x)-f(y)}
    For all functions $f$ which are $2\pi/m$-periodic and belonging to $W^{1,\infty}$ we have $|f(x)-f(y)| \leq \Y{f'} p_m(x-y)$.
\end{lemma}
\begin{proof}
    Since $f$ is $2\pi/m$ periodic, for all $x,y$ there exists $x', y'$ such that $|x'-y'| \leq \frac{\pi}{m}$, $f(x)=f(x'), f(y)=f(y')$ and $|x'-y'| = p_m(x-y)$. Then
    \begin{align*}
        |f(x)-f(y)| =|f(x')-f(y)'| \leq \Y{f'} |x'-y'| = \Y{f'} p_m(x-y).
    \end{align*}
\end{proof}

\subsection{$L^2$ estimates of the functional}

In this subsection we will focus on proving the Lemmas and Propositions leading to the estimates of the first half of Proposition \ref{Final_nonlinear_bounds_NL}. Let us start by recalling the definitions of $C[v](x,y)$ and $S[v](x,y)$:

\begin{align*}
        C[v](x,y) &\coloneqq \frac{1}{4\pi\Omega} \cos(x-y) \log (A[v](x,y)), \quad
        S[v](x,y) \coloneqq \frac{1}{4\pi\Omega} \sin(x-y) \log (A[v](x,y)).
\end{align*}

where
\begin{align*}
        A[v](x,y) &\coloneqq (R_0(x)-R_0(y) + v(x)-v(y))^2 + 4(R_0(x)+v(x))(R(y)+v(y))\sin^2 \left ( \frac{x-y}{2}\right ) \\
        & = (R_0(x)+v(x))^2 + (R_0(y)+v(y))^2 - 2(R_0(x)+v(x))(R(y)+v(y))  \cos(x-y). 
\end{align*}

We have the following Lemma:

\begin{lemma}\label{CSbounds}
    For all $v \in \Tilde{Y}_m^{even, \epsilon}$ we have the following bounds on $C[v], S[v]$:
    %\color{red} specify which space is in which variable. Perhaps better to write $\sup_{x} \|C[v](x,y)\|_{L^2_y}$ ? 
 \color{black}
    \begin{enumerate}
        \item \label{CL2bound} $\sup_x \|{C[v]}(x,y)\|_{L_y^2} \leq b_1 :=  \frac{1}{4\pi \Omega} \sqrt{ 16 t_1 + 8\log^2(2M_R) \left ( \frac{\pi}{2} - \frac{\sin(2t_2) + 2t_2}{4} \right )} $, 
        \item \label{SLinfbound}$\sup_x \|S[v](x,y)\|_{L_y^\infty } \leq   b_2 := \max \left \{\frac{1}{\Omega e2^{5/2}m_R} , \frac{1}{4\pi \Omega}\log \left ((M_R-m_R)^2+3M_R^2 \right ) ,  \frac{\sqrt{3}}{8\pi\Omega} \log \left ((M_R-m_R)^2+4M_R^2\right ) \right \} $, 
    \end{enumerate}
    where $t_1 := 2 \arcsin{\left (\frac{1}{2 m_R}\right )}$ and $t_2 := 2 \arcsin{\left (\frac{\sqrt{1-(M_R-m_R)^2}}{2 M_R}\right )}$. 
  % Note $\|C[v]\|_{L_x^2 L_y^\infty} = \|C[v]\|_{L_y^2 L_x^\infty} \leq b_1$.
\end{lemma}
\begin{proof}
    Using the definition of $A[v]$, we get
    \begin{equation}\label{logAbounds}
        |\log(A[v](x,y))| \leq 
        \begin{cases}
            -\log \left (4m_R^2\sin^2 \left (\frac{x-y}{2}\right ) \right ) & 
            \text{ if } A[v](x,y) \leq 1 \\
            \min \{ \log \left ((M_R-m_R)^2 + 4M_R^2\sin^2 \left (\frac{x-y}{2}\right ) \right ), \log(4M_R^2)\} & \text{ if } A[v](x,y) \geq 1
        \end{cases}
    \end{equation}
    where we have used that if $m_A < A < M_A$ then $|\log(A)| \leq -\log(m_A)$ when $A \leq 1$ and $|\log(A)| \leq \log(M_A)$ when $A\geq 1$. 

    We now use the new variable $z = x-y$. The objective is to find a bound independent of $x$ so then we can drop the $\sup_x$ part of the norm in each estimate due to $L^2_y = L^2_z$ and $L^\infty_y = L^\infty_z$.
    
    \textbf{Proof of \eqref{CL2bound}:} 
    First we find estimates of the sets of $z \in [-\pi, \pi]$ such that $A[v](x,x-z)$ is bigger or less than 1. Let us consider the following regions:
    \begin{align*}
        I_1& =\left \{ z : A[v](x,x-z) \leq 1\right \} \subseteq \left \{ z : 4m_R^2\sin^2 \left (\frac{z}{2}\right ) \leq 1\right \} = \left \{ |z| \leq t_1 := 2 \arcsin{\left (\frac{1}{2 m_R}\right )}\right\}, \\
        I_2 &= \left \{ z : A[v](x,x-z) \geq 1\right \} \subseteq \left \{ z : (M_R-m_R)^2 + 4M_R^2\sin^2 \left (\frac{z}{2}\right ) \geq 1\right \} \\
        & = \left \{ \pi \geq |z| \geq t_2 := 2 \arcsin{\left (\frac{\sqrt{1-(M_R-m_R)^2}}{2 M_R}\right )}\right\}. \\
    \end{align*}
    We will also use that 
    \begin{equation*}
        \sin\left (\frac{z}{2} \right ) \geq \frac{\sin\left (\frac{t_1}{2} \right )}{t_1}z  = \frac{z}{2m_R t_1}
    \end{equation*}
    whenever $z \in [0,t_1]$ because $\sin(z/2)$ is concave in $[0, \pi]$. 
    We now bound the integral:
    \begin{align*}
        \int_{-\pi}^\pi \cos(z)^2 \log^2 (A[v](x,x-z)) dz 
        &\leq \int_{I_1} \cos(z)^2 \log^2 \left (4m_R^2\sin^2 \left (\frac{z}{2}\right ) \right )dz  + \int_{I_2} \cos^2(z) \log^2(4M_R^2)dz   
        \\
        &\leq 2\int_{0}^{t_1} \cos(z)^2 \log^2 \left (4m_R^2\sin^2 \left (\frac{z}{2}\right ) \right )dz  + 2\int_{t_2}^\pi \cos^2(z) \log^2(4M_R^2)dz 
        \\
        &\leq  2\int_{0}^{t_1} \log^2 \left (\frac{z^2}{t_1^2}\right )dz  + 2\log^2(4M_R^2)\int_{t_2}^\pi \cos^2(z)dz  \\
        &\leq 2t_1 \int_0^1 \log^2(y^2) dy + 2\log^2(4M_R^2) \left ( \frac{\pi}{2} - \frac{\sin(2t_2) + 2t_2}{4} \right )dz \\
        &\leq 16 t_1 + 8\log^2(2M_R) \left ( \frac{\pi}{2} - \frac{\sin(2t_2) + 2t_2}{4} \right ),
    \end{align*}
    
    where we have integrated by explicitly $\log^2(y^2)$ and $\cos^2(z)$. We can conclude %\color{red} fix notation norms? \color{black}
    \begin{equation*}
        \sup_x\|C[v]\|_{L_y^2} \leq \frac{1}{4\pi \Omega} \sqrt{ 16 t_1 + 8\log^2(2M_R) \left ( \frac{\pi}{2} - \frac{\sin(2t_2) + 2t_2}{4} \right )} = b_1.
    \end{equation*}

    \textbf{Proof of \eqref{SLinfbound}:} From the bounds in \eqref{logAbounds} we can deduce
    \begin{equation*}
        |\log(A[v](x,y))| \leq \max \left \{-\log \left (4m_R^2\sin^2 \left (\frac{x-y}{2}\right ) \right ),  \log \left ((M_R-m_R)^2 + 4M_R^2\sin^2 \left (\frac{x-y}{2}\right ) \right ) \right \}.
    \end{equation*}
    Also, by periodicity and symmetry of $S[v]$,  $\| S[v](x,y)\|_{L_y^\infty} = \| S[v](x,x-z)\|_{L_z^\infty[0, \pi]}$. Combining these, we have reduced our problem to computing the maximum of two one-variable functions in $[0,\pi]$:
    
    \begin{align*}
        4\pi \Omega \Linf{S[v]} & \leq  \max \left \{ \max_{z \in [0, \pi]} \left \{ -\sin(z) \log \left (4m_R^2\sin^2 \left (\frac{z}{2}\right ) \right )\right \} ,\right. \\ & \left.\max_{z \in [0, \pi]} \left \{ \sin(z) \log \left ((M_R-m_R)^2+4M_R^2\sin^2 \left (\frac{z}{2}\right ) \right )\right \} \right \}.
    \end{align*}
    Starting with the first one, we realize that the maximum over $[0, \pi]$ is actually the same as the maximum as $[0, \pi/2]$ because for all $z \in [\pi/2, \pi]$ the function is negative as $4m_R^2 \sin^2(z/2) \geq 4m_R^2 \sin^2(\pi/4) \geq 1$ makes the $\log$ change sign. This last inequality is satisfied because $m_R \geq 1/\sqrt{2}$. Also
    \begin{align*}
        \max_{z \in [0, \pi/2]} \left \{ -2\sin(z) \log \left (2m_R\sin \left (\frac{z}{2}\right ) \right )\right \} \leq \max_{z \in [0, \pi/2]} \left \{ -2 z \log \left (2m_R \frac{\sqrt{2}}{\pi} z\right ) \right \} = \frac{\pi}{e\sqrt{2}m_R},
    \end{align*}
    where in the last step we have explicitly calculated the maximum.

    For the second function we split the interval at $2\pi/3$ and we bound the maximum in $[0, 2\pi/3]$ and in $[2\pi/3, \pi]$, using that the function inside the logarithm is increasing and the $\log$ is a monotone function and that $\sin$ is decreasing in $[2\pi/3, \pi]$ we get:
    \begin{align*}
        &\max_{z \in [0, \pi]} \left \{ \sin(z) \log \left ((M_R-m_R)^2+4M_R^2\sin^2 \left (\frac{z}{2}\right ) \right )\right \}  \\
        &\leq \max \left \{\log \left ((M_R-m_R)^2+4M_R^2\sin^2 \left (\frac{2\pi/3}{2}\right ) \right ) , \sin(2\pi/3) \log \left ((M_R-m_R)^2+4M_R^2\sin^2 \left (\frac{\pi}{2}\right ) \right ) \right \}  \\
        &= \max \left \{\log \left ((M_R-m_R)^2+3M_R^2 \right ) ,  \frac{\sqrt{3}}{2} \log \left ((M_R-m_R)^2+4M_R^2\right ) \right \},
    \end{align*}

    where the first term of the first inequality corresponds to the bound for whenever $z \in [0,2\pi/3]$ and the second term to $z \in [2\pi/3,\pi]$. 
    We can conclude 
    \begin{align*}
        & \Linf{S[v]} 
        \leq \frac{1}{4\pi \Omega}\max \left \{\frac{\pi}{e\sqrt{2}m_R} , \log \left ((M_R-m_R)^2+3M_R^2 \right ) ,  \frac{\sqrt{3}}{2} \log \left ((M_R-m_R)^2+4M_R^2\right ) \right \}.
    \end{align*}
    
\end{proof}

\begin{proposition} \label{Nv}
    If $v \in \Tilde{Y}_m^{even, \epsilon}$ then
    \begin{align*}
        \Xm{N [v]} \leq C_N \Xm{v'}^2
    \end{align*}
    with $C_N := \sqrt{\frac{\pi}{m}}  + \frac{2^{3/2}\pi }{\sqrt{m}} b_1+ \left ( \frac{2 \pi^{5/2}}{m^{3/2}}+ 2\sqrt{\pi m} \right ) b_2 $. 
\end{proposition}
\begin{proof}
    By periodicity and oddness of $N[v]$ due to Proposition \ref{TEN_symmetries}, it is enough to show that $\Ltwo{N [v]} \leq \frac{C_N}{\sqrt{2m}} \Ltwo{v'}^2$
    Recalling \eqref{OPv}, using the triangle inequality, and using the $L^\infty-\dot{H}^1$ inequality of Lemma \ref{vLinftyv'Ltwobound} we get the desired estimate
    
    \begin{equation}\label{Nbound}
    \begin{aligned} 
        &\Ltwo{N[v]} \\
        &\leq \Ltwo{vv'} + \Ltwo{\int_0^{2\pi} C[v] v(x)v'(y) dy} + \Ltwo{\int_0^{2\pi} C[v] v(y)v'(x)dy} \\
        & + \Ltwo{\int_0^{2\pi} S[v] v(x)v(y) dy} +\Ltwo{\int_0^{2\pi} S[v] v'(x)v'(y) dy} \\
        &\leq \Y v \Ltwo{v'} + \Y v \Ltwo{v'} \Ltwoinf{C[v]} \Ltwo 1 + \Y v \Ltwo{v'} \Loneinf{C[v]} + \\
        &+ \Y v ^2 \Linf{S[v]} \Lone 1 \Ltwo 1 + \Ltwo{v'}^2 \Ltwo 1 \Linf{S[v]} \\
        &\leq \Ltwo{v'}^2 \left ( \frac{\sqrt{\pi}}{m\sqrt{2}} + \frac{\sqrt{\pi}}{m\sqrt{2}} b_1 \sqrt{2\pi}  + \frac{\sqrt{\pi}}{m\sqrt{2}} \sqrt{2\pi} b_1 + \frac{\pi}{2m^2} b_2 2\pi \sqrt{2\pi} + \sqrt{2\pi} b_2\right )  \\
        &= \Ltwo{v'}^2 \left ( \frac{\sqrt{\pi}}{m\sqrt{2}}  + \frac{2\pi}{m} b_1 + \left ( \frac{\sqrt{2} \pi^{5/2}}{m^2} + \sqrt{2\pi} \right ) b_2\right ) \leq \frac{C_N}{\sqrt{2m}} \Ltwo{v'}^2,
    \end{aligned}
    \end{equation}
    where we have also used that $\Ltwoinf{C[v]} \leq b_1$ and $\Linf{S[v]} \leq b_2$ as proved in Lemma \ref{CSbounds}.
\end{proof}

%\color{red} domain bounds on integral everywhere \color{black}

\begin{lemma} \label{CSvCS0linbound}
    For all $v \in \Tilde{Y}_m^{even, \epsilon}$ and $w \in L^2(\mathbb{T})$ we have the following bounds on the linearization error of $C[v], S[v]$ : 
    \begin{enumerate}
        \item \label{Cvlinboundy} $\Ltwo{\int_0^{2\pi} |(C[v]-C[0]) w(y)| dy } \leq    b_3^C \Ltwo{v'} \Ltwo{w}$,
        \item \label{Cvlinboundx}$\Ltwo{\int_0^{2\pi} |(C[v]-C[0]) w(x)| dy } \leq  b_3^C \Ltwo{v'} \Ltwo{w}$,
        \item \label{Svlinboundy} $\Ltwo{\int_0^{2\pi} |(S[v]-S[0]) w(y)| dy } \leq  b_3^S \Ltwo{v'} \Ltwo{w}$,
        \item \label{Svlinboundx}$\Ltwo{\int_0^{2\pi} |(S[v]-S[0]) w(x)| dy } \leq b_3^S \Ltwo{v'} \Ltwo{w}$,
    \end{enumerate} 
     with $b_3^C := \frac{1}{\Omega} \left ( \frac{I_1^C(m)}{2^{7/2} \sqrt{\pi} m_R \sqrt{m}} + \frac{\sqrt{2}}{\sqrt{\pi}m_R m}\right )$  and $  b_3^S := \frac{1}{\Omega} \left ( \frac{I_1^S(m)}{2^{7/2} \sqrt{\pi} m_R \sqrt{m}} + \frac{\sqrt{2}}{\sqrt{\pi}m_R m}\right ) $ and $I_1^C, I_1^S$ are explicit integral bounds defined in Lemma \ref{Integralbound1}:
     \begin{align*}
         \int_{-\pi}^\pi \frac{|\cos(z)|\sqrt{p_m(z)}}{|\sin\left (\frac{z}{2}\right )|}dz  \leq I_1^C(m), \quad \int_{-\pi}^\pi \frac{|\sin(z)|\sqrt{p_m(z)}}{|\sin\left (\frac{z}{2}\right )|}dz  \leq I_1^S(m) 
     \end{align*}
    
\end{lemma}
\begin{proof}
    By definition $C[v]-C[0] = \cos(x-y) (\log(A[v](x,y))-\log(A[0](x,y))$, and similarly for $S$ changing $\cos(x-y)$ by $\sin(x-y)$. It is then enough to control the difference of $\log(A)$. Using the mean value theorem
    \begin{equation}\label{logAvlogA0}
        \log(A[v](x,y))-\log(A[0](x,y)) = \int_0^1 \left (\frac{d}{d \lambda} \log(A[\lambda v](x,y)) \right ) d\lambda = \int_0^1 \left (\frac{\frac{d}{d \lambda}  A[\lambda v](x,y) }{A[\lambda v](x,y)} \right ) d\lambda  
    \end{equation} 
    Computing the derivative
    %\color{red} careful here since $\lambda^*$ depends on $x,y$. 
    
    %Sorry not a big fan of this proof. Can we just write $\log(A[v]/A[0]) = \log(1+blah)$ and bound it by $|blah|$? I think this is exactly what you're doing. Let's talk about it.\color{black}
    \begin{equation}\label{d_lambdaA}
    \begin{aligned}
        \frac{d}{d \lambda} A[\lambda v]|(x,y) &= 2\left (R_0(x)-R_0(y) +  \lambda(v(x)-v(y))\right )(v(x)-v(y)) + \\
        &+ 4\sin^2\left (\frac{x-y}{2}\right ) \left ((R_0(x)+\lambda v(x))v(y) + (R_0(y)+\lambda v(y))v(x)\right ) =: a_1 + a_2,
    \end{aligned}
    \end{equation}
    we will bound both terms $a_1/A, a_2/A$ separately for every $\lambda \in [0,1]$.  

    For the second one, we bound $A$ from below. We cancel terms with $a_2$ and only then we use $R_0+\lambda v > m_R$, which is true because of Lemma \ref{boundR0}, getting 
    \begin{equation}\label{a2/A}
    \begin{aligned}
        \left |\frac{a_2}{A[\lambda v]}\right | &\leq \frac{a_2(x,y)}{4 \sin^2\left (\frac{x-y}{2}\right ) ((R_0(x)+\lambda v(x))(R_0(y)+\lambda v(y))} \leq \frac{1}{m_R} (v(y)+ v(x)) \leq \frac{2}{m_R} \| v \|_{L^\infty} \leq  \frac{\sqrt{2\pi}}{m_R m} \Ltwo{v'}.
    \end{aligned}
    \end{equation}
    For the first term $a_1$ we have to use another lower bound for $A$. As before, we first cancel terms with $a_1$ and then use $R_0+ \lambda v > m_R$. Moreover, we set $y = x-z$ because we are going to change variables on the integral. We get
    \begin{equation}\label{a1/A}
    \begin{aligned}
         \left |\frac{a_1}{A[\lambda v]}\right | &\leq \left |\frac{2\left (R_0(x)-R_0(x-z) +  \lambda(v(x)-v(x-z))\right )(v(x)-v(x-z))}{4m_R\left (R_0(x)-R_0(x-z) +  \lambda(v(x)-v(x-z))\right ) \sin\left (\frac{z}{2}\right )}\right |  \\
         &\leq \frac{\left |v(x)-v(x-z) \right | }{2m_R\left | \sin\left (\frac{z}{2}\right )\right  |}  \\
         &\leq \frac{\X{v'} \frac{1}{\sqrt{2m}} \sqrt{|p_m(x)-p_m(x-z)|}}{2m_R |\sin\left (\frac{z}{2}\right )|}   \\
         &\leq \frac{\X{v'}\sqrt{\pi}}{2^{3/2} m_R\sqrt{m} } \frac{\sqrt{p_m(z)}}{|\sin\left (\frac{z}{2}\right )|},
    \end{aligned}
    \end{equation}
    where in the last inequality step we used Lemmas \ref{bound_v(x)-v(y)} and \ref{bound2_v(x)-v(y)} to bound $|v(x)-v(y)|$. Going back to \eqref{logAvlogA0} we get
    \begin{equation}\label{dlambdaA/A_bound}
        |\log(A[v]) - \log(A[0])|(x,x-z) \leq  \int_0^1 \left |\left (\frac{\frac{d}{d \lambda}  A[\lambda v](x,x-z) }{A[\lambda v](x,x-z)} \right ) \right | d\lambda  \leq  \frac{\Ltwo{v'} \sqrt{2\pi}}{m_R m}  + \frac{\X{v'}\sqrt{\pi}}{2^{3/2} m_R\sqrt{m} } \frac{\sqrt{p_m(z)}}{|\sin\left (\frac{z}{2}\right )|}. 
    \end{equation}
    
    We will only prove items \ref{Cvlinboundy} and \ref{Svlinboundy}, since \ref{Cvlinboundx} and \ref{Svlinboundx} are analogous. We start by \ref{Cvlinboundy}, changing variables in the integral $z = x-y$, using the inequality we just proved and Minkowski's inequality, we obtain
    \begin{equation} \label{prooflogAvlogA0}
    \begin{aligned}
        &\Ltwo{\int_0^{2\pi} 4\pi \Omega|(C[v]-C[0]) w(y)| dy } = \Ltwox{\int_{-\pi}^\pi |\cos(z)(\log(A[v](x,x-z))-\log(A[0](x,x-z))) w(x-z)| dz }  \\
        &\leq \Ltwox{\int_{-\pi}^\pi \left (\frac{\X{v'}\sqrt{\pi}}{2^{3/2} m_R\sqrt{m} } \frac{|\cos(z)|\sqrt{p_m(z)}}{|\sin\left (\frac{z}{2}\right )|}  +   \frac{|\cos(z)|\sqrt{2\pi}}{m_R m} \Ltwo{v'} \right ) |w(x-z)| dz }  \\
        &\leq \int_{-\pi}^\pi \left (\frac{\sqrt{\pi}}{2^{3/2} m_R\sqrt{m} } \frac{|\cos(z)|\sqrt{p_m(z)}}{|\sin\left (\frac{z}{2}\right )|}  +  \frac{|\cos(z)|\sqrt{2\pi}}{m_R m}  \right ) \X{v'} \Ltwo w dz  \\
        &\leq  \left (\frac{\sqrt{\pi}}{2^{3/2} m_R\sqrt{m} } \int_{-\pi}^\pi \frac{|\cos(z)|\sqrt{p_m(z)}}{|\sin\left (\frac{z}{2}\right )|}dz  +  \frac{2^{5/2}\sqrt{\pi}}{m_R m}  \right ) \X{v'} \Ltwo w  \\
        & \leq 4\pi \Omega \ b_3^{C} \Ltwo{v'}\Ltwo w,
    \end{aligned}
    \end{equation}
    where in the last line we have used Lemma \ref{Integralbound1} to bound the integral.
    Following the same procedure but with a $\sin(z)$ we get the same bound but with the constant $I_1^S(m)$ bounding the integral instead of $I_1^C(m)$.
\end{proof}

%\color{red} The previous proposition needs some reworking \color{black}

\begin{proposition} \label{T-T_1 v}
    If $v \in \Tilde{Y}_m^{even, \epsilon}$ then 
    \begin{align*}
        \Xm{(T-T_1) [v]} \leq C_T \Xm{v'}^2
    \end{align*}
    with $C_T := (b_3^C+b_3^S)(M_{R_0}+ M_{R_0'}) (\sqrt{2m} + \frac{\pi \sqrt{2}}{\sqrt{m}} )$ and $M_{R_0'}$ is defined in Lemma \ref{MdR0}.
\end{proposition}
\begin{proof}
    By periodicity and oddness of $T,T_1$, it is enough to show that $\Ltwo{(T-T_1) [v]} \leq \frac{C_T}{\sqrt{2m}} \Ltwo{v'}^2 $.
    Recalling Definitions \ref{OPv} and \ref{T1E1_def} we get
    \begin{align*}
         \Ltwo{(T-T_1) [v]} &\leq \Ltwo{\int_0^{2\pi}  v'(x)R_0(y)(C[v]-C[0])  dy} +  \Ltwo{\int_0^{2\pi}  v'(x)R_0'(y)(S[v]-S[0])  dy}  \\
         &+ \Ltwo{\int_0^{2\pi} v(x)R_0'(y)(C[v]-C[0]) dy} + \Ltwo{\int_0^{2\pi}  v(x)R_0(y)(S[v]-S[0])  dy} \\
         &+ \Ltwo{\int_0^{2\pi}  v(y)R_0(x)(S[v]-S[0])  dy} + \Ltwo{\int_0^{2\pi}  v(y)R_0'(x)(C[v]-C[0])  dy}  \\
         &+ \Ltwo{\int_0^{2\pi}  v'(y)R_0'(x)(S[v]-S[0])  dy} + \Ltwo{\int_0^{2\pi}  v'(y)R_0(x)(C[v]-C[0])  dy}. 
    \end{align*}
    
    Using $L^\infty$ bounds on $R,R'$ and using Lemma  \ref{CSvCS0linbound}:
    \begin{align*}
        \Ltwo{(T-T_1) [v]} 
        &\leq \left ( M_{R_0} \Ltwo{v'} + M_{R_0'} \Ltwo{v'} + M_{R_0'} \Ltwo{v} + M_{R_0} \Ltwo{v} \right ) (b_3^C+b_3^S) \Ltwo{v'}  \\
        &\leq (b_3^C+b_3^S) \left (  M_{R_0} + M_{R_0'} + M_{R_0'} \Ltwo 1 \frac{\sqrt{\pi}}{m\sqrt{2}} + M_{R_0} \Ltwo 1 \frac{\sqrt{\pi}}{m\sqrt{2}} \right ) \Ltwo{v'}^2  \\
        &= (b_3^C+b_3^S)(M_{R_0}+ M_{R_0'}) (1 + \frac{ \pi}{m})\Ltwo{v'}^2 \leq \frac{C_T}{\sqrt{2m}} \Ltwo{v'}.
    \end{align*}
\end{proof}

%\color{red} Same issues as in 4.9 \color{black}

\begin{lemma}\label{CSvCS0secondorderbound}
    For all $v \in \Tilde{Y}_m^{even, \epsilon}$ we have

    \begin{equation*}
        \Ltwo{\int_{-\pi}^\pi \left |\left (\log(A[v])-\log(A[0]) - \frac{d}{d\lambda} \log(A[\lambda v])|_{\lambda = 0} \right )(x,x-z) h_\alpha(z) \right | dz  } \leq 4 \pi \Omega \ b_4^\alpha \Ltwo{v'}^2 
    \end{equation*}
    with $h_C(z) :=  \cos(z) p_m(z)$, $h_S(z) := \sin(z)$. And
    \begin{equation*}
        b_4^\alpha := \frac{1}{8\Omega} \left ( \frac{1/\sqrt{\pi} + \sqrt{\pi}/2}{2^{3/2}mm_R^2} I_2^\alpha(m) + \frac{3\sqrt{2\pi}}{m^2m_R^2} I_3^\alpha(m)+ \frac{\sqrt{2\pi}}{m_R^2m^{3/2}} I_4^\alpha(m) \right )
    \end{equation*}
    with $I_2^\alpha, I_3^\alpha, I_4^\alpha$ defined in Lemmas \ref{Integralbound2}, \ref{Integralbound3}, \ref{Integralbound4}.
\end{lemma}
\begin{proof}
    We begin by computing the error of linearizing
    \begin{equation}\label{logAvlogA0dlogAv}
        \left(\log(A[v])-\log(A[0]) - \frac{d}{d\lambda} \log(A[\lambda v])|_{\lambda = 0} \right)(x,y) = \int_0^1 \int_0^{\lambda'}\frac{d^2}{d\lambda^2} \log(A[\lambda v](x,y)) d\lambda d\lambda'. 
    \end{equation}
    Computing the derivative
    \begin{equation}\label{TaylorlogA}
        \frac{d^2}{d\lambda^2} \log(A[\lambda v]) = \frac{\frac{d^2}{d\lambda^2} A[\lambda v] }{A[\lambda v]} - \left (\frac{\frac{d}{d\lambda} A[\lambda v]}{A[\lambda v]}\right )^2.
    \end{equation}
    We will bound both terms separately. Starting at the second term we can use the bounds \eqref{a2/A}, \eqref{a1/A} to get
    \begin{align*}
        \left (\frac{\frac{d}{d\lambda} A[\lambda v] }{A[\lambda v]} \right )^2 (x,x-z) \leq \left (\frac{\X{v'}\sqrt{\pi}}{2^{3/2} m_R\sqrt{m} } \frac{\sqrt{p_m(z)}}{|\sin\left (\frac{z}{2}\right )|}  +  \frac{\X{v'}\sqrt{2\pi}}{m_R m}  \right )^2
    \end{align*}

    Now we will proceed to the first term, for which we have to compute a further derivative of the first derivative we computed in \eqref{d_lambdaA}:
    \begin{align*}
        \frac{d^2}{d\lambda^2} A[\lambda v]  = 2(v(x)-v(y))^2 + 8 v(x)v(y) \sin^2 \left ( \frac{x-y}{2}\right ).
    \end{align*}
    Then we use the bounds from Lemma \ref{bound2_v(x)-v(y)} for $|v(x)-v(x-z)|$, and the bounds from Lemma \ref{boundR0}, yielding 
    \begin{align*}
        \left | \frac{\frac{d^2}{d\lambda^2} A[\lambda v] }{A[\lambda v]}  (x,x-z)\right | &\leq \left  | \frac{2(v(x)-v(x-z))^2}{A[\lambda v](x,x-z)} \right | + \left | \frac{8 v(x)v(x-z) \sin^2 \left ( \frac{z}{2
        }\right )}{A[\lambda^*v](x,x-z)}  \right |\\
        &\leq \left | \frac{\X{v'}^2 p_m(z)}{4mm_R^2\sin^2 \left ( \frac{z}{2} \right ) } \right | + \left | \frac{8 \Y{v}^2 \sin^2 \left ( \frac{z}{2}\right )}{4m_R^2\sin^2 \left ( \frac{z}{2} \right ) }  \right |.
    \end{align*}
    Putting both bounds together and recalling equality \eqref{logAvlogA0dlogAv} we get the following bound:
    \begin{align*}
        &\left |\log(A[v])-\log(A[0]) - \frac{d}{d\lambda} \log(A[\lambda v])|_{\lambda = 0} \right |(x,x-z) \leq \\
        &\leq \int_0^1 \int_0^{\lambda'} \left | \frac{\frac{d^2}{d\lambda^2} A[\lambda v] }{A[\lambda v]} (x,x-z)\right | + \left |\frac{\frac{d}{d\lambda} A[\lambda v] }{A[\lambda v]}(x,x-z) \right |^2  d\lambda d\lambda' \\
        &\leq \left ( \left (\frac{\X{v'}\sqrt{\pi}}{2^{3/2} m_R\sqrt{m} } \frac{\sqrt{p_m(z)}}{|\sin\left (\frac{z}{2}\right )|}  +  \frac{\X{v'}\sqrt{2\pi}}{m_R m}  \right )^2 +  \frac{\X{v'}^2 \left | p_m(z)\right |}{4mm_R^2\sin^2 \left ( \frac{z}{2} \right ) }  +  \frac{8 \Y{v}^2 \sin^2 \left ( \frac{z}{2}\right )}{4m_R^2\sin^2 \left ( \frac{z}{2} \right ) }\right ) \int_0^1 \int_0^{\lambda'} 1 d\lambda d\lambda' . 
    \end{align*}
    Using that $\int_0^1 \int_0^{\lambda'} 1 d\lambda d\lambda' = \frac{1}{2}$ we can proceed with the final bound

    \begin{align*}
        &\Ltwox{\int_{-\pi}^\pi \left |\left (\log(A[v])-\log(A[0]) - \frac{d}{d\lambda} \log(A[\lambda v])|_{\lambda = 0} \right )(x,x-z) h_\alpha(z) \right | dz  } \leq \\
        &\leq  \Ltwox{\frac{1}{2}\int_{-\pi}^{\pi} \left (\frac{\X{v'}\sqrt{\pi}}{2^{3/2} m_R\sqrt{m} } \frac{\sqrt{p_m(z)}}{|\sin\left (\frac{z}{2}\right )|}  +  \frac{\X{v'}\sqrt{2\pi}}{m_R m}  \right )^2 |h_\alpha(z)| dz } \\
        &+ \Ltwox{\frac{1}{2} \int_{-\pi}^\pi \left | \frac{\X{v'}^2 p_m(z)}{4mm_R^2\sin^2 \left ( \frac{z}{2} \right ) } \right | |h_\alpha(z)| dz } + \Ltwox{\frac{1}{2} \int_{-\pi}^\pi \left | \frac{8 \Y{v}^2 \sin^2 \left ( \frac{z}{2}\right )}{4m_R^2\sin^2 \left ( \frac{z}{2} \right ) }  \right | |h_\alpha(z)| dz } \\
        &\leq  \frac{1}{2} \X{v'}^2 \sqrt{2\pi} \left (\frac{\pi}{2^3 m_R^2m } \int_{-\pi}^{\pi} \frac{p_m(z)|h_\alpha(z)|}{\sin^2\left (\frac{z}{2}\right )} dz  +   \frac{2\pi}{m_R^2 m^2} \int_{-\pi}^\pi |h_\alpha(z)| dz+ \frac{\pi}{m_R^2 m^{3/2}} \int_{-\pi}^\pi \frac{|h_\alpha(z)|\sqrt{p_m(z)}}{|\sin\left (\frac{z}{2}\right )|} dz   \right ) \\
        &+ \frac{\X{v'}^2\X{1}}{8mm_R^2} \int_{-\pi}^\pi \frac{p_m(z) |h_\alpha(z)|}{\sin^2 \left ( \frac{z}{2} \right )} dz +  \frac{\pi \X{v'}^2 \X{1} }{2m^2 m_R^2} \int_{-\pi}^\pi |h_\alpha(z)| dz\\
        &\leq \frac{1}{2}\left (\frac{\pi^{3/2}}{2^{5/2} m_R^2m } I_2^\alpha(m)  +   \frac{(2\pi)^{3/2}}{m_R^2 m^2} I_3^\alpha(m)+ \frac{\sqrt{2}\pi^{3/2}}{m_R^2 m^{3/2}} I_4^\alpha(m)   \right ) \X{v'}^2  \\
        &+ \frac{1}{2}\left ( \frac{\sqrt{\pi}}{2^{3/2}m m_R^2} I_2^\alpha(m) + \frac{\sqrt{2} \pi^{3/2}}{m^2m_R^2} I_3^\alpha(m)\right ) \X{v'}^2 \\
        &= 4\pi \Omega  \left ( \frac{1/\sqrt{\pi} + \sqrt{\pi}/2}{2^{3/2}mm_R^2} I_2^\alpha(m) + \frac{3\sqrt{2\pi}}{m^2m_R^2} I_3^\alpha(m)+ \frac{\sqrt{2\pi}}{m_R^2m^{3/2}} I_4^\alpha(m) \right ) \frac{1}{8\Omega} \Ltwo{v'}^2 = 4\pi \Omega \ b_4^\alpha \Ltwo{v'}^2,
    \end{align*}
    where we have used Lemmas \ref{Integralbound2}, \ref{Integralbound3}, \ref{Integralbound4} to bound the integrals in terms of $I_i^\alpha$.

    \end{proof}

\begin{proposition} \label{Ev-E0-E1v}
    If $v \in \Tilde{Y}_m^{even, \epsilon}$ then
    \begin{align*}
        \Xm{E[v]- E[0]-E_1[v]} \leq C_E \Xm{v'}^2
    \end{align*}
    with $C_E := \sqrt{2m} \left ( M_{R_0}M_{R_0''} + M_{R_0'}^2 \right ) b_4^C+ \sqrt{2m} \left (M_{R_0'}^2+M_{R_0}^2 \right )  b_4^S $ and $M_{R_0'}, M_{R_0''}$ are defined in Lemmas \ref{MdR0}, \ref{MddR0}.
\end{proposition}
\begin{proof}
    By periodicity and oddness of $E, E_1, v$, it is enough to prove $\Ltwo{E[v]- E[0]-E_1[v]} \leq \frac{C_E}{\sqrt{2m}} \Ltwo{v'} ^2$. Recalling Definitions \ref{OPv} and \ref{T1E1_def}:
    \begin{align*}
        &\Ltwo{E[v]- E[0]-E_1[v]} \\
        &\leq \Ltwo{\int_{-\pi}^\pi (C[v] - C[0] - \frac{d}{d\lambda} C[\lambda v] |_{\lambda = 0} )(x,x-z) (R_0(x)R_0'(x-z) - R_0(x-z)R_0'(x) )dz}  +\\
        &+ \Ltwo{\int_{-\pi}^\pi (S[v] - S[0] - \frac{d}{d\lambda} S[\lambda v] |_{\lambda = 0} )(x,x-z) (R_0(x)R_0(x-z) + R_0'(x)R_0'(x-z)) dz}.  
    \end{align*}
    First we get an extra cancellation for the first term
    \begin{align*}
        |R_0(x)R_0'(x-z) - R_0(x-z)R_0'(x)| 
        &\leq |R_0(x)| |R_0'(x-z) - R_0'(x)| + |R_0'(x)| |R_0(x-z) - R_0(x)| \\
        &\leq M_{R_0} M_{R_0''} p_m(z) + M_{R_0'}^2 p_m(z),
    \end{align*}
    where we have used that both $R,R'$ are $2\pi/m$ periodic, even and odd respectively and have its derivatives bounded so we can apply Lemma \ref{f(x)-f(y)}. On the second term we can directly apply $L^\infty$ bounds to $R_0, R_0'$.
    
    Applying Lemma \ref{CSvCS0secondorderbound} with $h_\alpha(z) = \cos(z)p_m(z)$ and $h_\alpha =\sin(z)$ we get
    \begin{align*}
        &\Ltwo{E[v]- E[0]-E_1[v]}\\ 
        &\leq \frac{1}{4\pi \Omega}\Ltwo{\int_{-\pi}^\pi \left |\left (\log(A[v])-\log(A[0]) - \frac{d}{d\lambda} \log(A[\lambda v])|_{\lambda^* = 0} \right )(x,x-z) \cos(z)p_m(z) \right | dz  } \left ( M_{R_0}M_{R_0''} + M_{R_0'}^2 \right ) \\
        &+ \frac{1}{4\pi \Omega}\Ltwo{\int_{-\pi}^\pi \left |\left (\log(A[v])-\log(A[0]) - \frac{d}{d\lambda} \log(A[\lambda v])|_{\lambda^* = 0} \right )(x,x-z) \sin(z) \right | dz  } \left (M_{R_0}^2 + M_{R_0'}^2 \right )  \\
        &\leq \left ( M_{R_0}M_{R_0''} + M_{R_0'}^2 \right ) b_4^C \X{v'}^2+ \left (M_{R_0}^2 + M_{R_0'}^2 \right ) b_4^S \X{v'}^2
        \leq \frac{C_E}{\sqrt{2m}} \Ltwo v ^2. 
    \end{align*}
\end{proof}

\subsection{Lipschitz estimates of the functional}

This subsection is devoted to perform Lipschitz estimates of the functional. We start with the following Lemma:

\begin{lemma} \label{Sv1Sv2Linfbound}
    If $v_0,v_1 \in \Tilde{Y}_m^{even, \epsilon}$ then
    \begin{equation}
        | S[v_0](x,y) - S[v_1](x,y)| \leq b_5 \Ltwo{v_0'-v_1'},
    \end{equation}
    where 
    $b_5:=\frac{1}{\Omega m_R m 2^{5/2}} \left ( 1+ \frac{2}{\sqrt{\pi}}\right ) $. 
\end{lemma}
\begin{proof}
    The proof is similar to Lemma \ref{CSvCS0linbound}. As $4 \pi \Omega S[v](x,y) = \sin(x-y)\log(A[v](x,y))$ we want to estimate $\log(A[v_1])- \log(A[v_0])$. Defining $v_\lambda = (1-\lambda)v_0 + \lambda v_1$, we can write 
    \begin{equation} \label{logAv1logAv2}
        \left (\log(A[v_1])-\log(A[v_0]) \right )(x,y)= \int_0^1\frac{d}{d\lambda}  \log(A[v_\lambda])(x,y) d\lambda = \int_0^1\frac{\frac{d}{d\lambda}  A[ v_\lambda]}{A[ v_\lambda]}(x,y) d\lambda.
    \end{equation}

    %We begin by introducing $v_\lambda(x) := v_1(x) + \lambda (v_2(x)-v_1(x))$ with $\lambda \in [0,1]$, which belongs to $\Tilde{Y}_m^{odd, \epsilon}$ because it is odd and $2\pi/m$ periodic, as it is a linear combination of $v_1,v_2$, also $v_\lambda (x) = \int_0^x v_\lambda'$ (in other words, $v_\lambda(0) = 0$) and the only thing left to see is the $\epsilon$ bound on its derivative
    %\begin{equation*}
     %   v_\lambda' = (1-\lambda)v_1' + \lambda v_2' \implies \Xm{v_\lambda'} \leq (|1-\lambda| + \lambda )\epsilon = \epsilon.
    %\end{equation*}
    %As a consequence, defining $R_\lambda := R_0+v_\lambda$, we then can use the bounds $m_R < R_\lambda(x) < M_R$.
    %\color{red} this is false as stated \color{black}
    
    %Now using the mean value theorem
    %\begin{align*}
    %\log(A[v_2])-\log(A[v_1]) = \frac{d}{d\lambda} \left. \log(A[v_\lambda]) \right |_{\lambda^*} = \frac{\frac{d}{d\lambda} \left . A[v_\lambda] \right |_{\lambda^*}}{A[ v_{\lambda^*}]}  \quad \lambda^* \in [0,1].
    %\end{align*}
        Since the space $\Tilde{Y}_m^{ even, \epsilon}$ is convex, for all $\lambda \in [0,1]$ we have $v_\lambda \in \Tilde{Y}_m^{even, \epsilon}$. Defining $R_\lambda := R_0 + v_\lambda$, we compute the derivative as in \eqref{d_lambdaA}
    \begin{equation}\label{d_lambdaAv1v2}
    \begin{aligned}
        \frac{d}{d\lambda} A[v_\lambda]|_\lambda(x,y) &= 2\left (R_\lambda(x)-R_\lambda(y)\right )(v_1(x)-v_1(y) + v_0(y)-v_0(x)) + \\
        &+ 4\sin^2\left (\frac{x-y}{2}\right ) \left (R_\lambda(x)(v_1(y)-v_0(y)) + R_\lambda(y)(v_1(x)-v_0(x))\right ) =: a_1 + a_2.
    \end{aligned}
    \end{equation}

    For all $\lambda \in [0,1]$, we proceed to bound $|\sin(x-y)|| a_1/A[v_{\lambda}]|$ and $|\sin(x-y) ||a_2/A[v_{\lambda}]|$ separately, using the exact same bounds as in \eqref{a2/A} via Lemma \ref{boundR0}:
    \begin{align*}
        |\sin(x-y)| \left |\frac{a_2}{A[v_{\lambda}]} \right | \leq \frac{\sqrt{2\pi}}{m m_R} \X{v_1'-v_0'}. 
    \end{align*}
    For $a_1$ we need to use the cancellation of $\sin(x-y)$. Changing variables to $z = x-y$ and using the bound \eqref{a1/A}:
    \begin{align*}
        \left |\sin(z) \frac{a_1}{A[v_{\lambda}]} \right | \ \leq \ \X{v_1'-v_0'} \frac{\sqrt{\pi}}{2^{3/2}m_R\sqrt{m}} \left |\frac{\sin(z)\sqrt{p_m(z)}}{\sin(\frac{z}{2})} \right |
        \ \leq \ \frac{\pi}{m_R m \sqrt{2}} \X{v_1'-v_0'}, 
    \end{align*}
    where we have used $p_m(z) \leq \pi/m$. Recalling \eqref{logAv1logAv2}, adding up both of the bounds we got and taking into account the $4\pi \Omega$ factor dividing we get
    \begin{equation*}
        \left |S[v_1](x,y)-S[v_0](x,y)\right | =  \left |\int_0^1 \frac{\sin(x-y)}{4\pi\Omega}\frac{\frac{d}{d\lambda}  A[ v_\lambda]}{A[ v_\lambda]}(x,y) d\lambda \right |\leq \frac{1}{\Omega m_R m 2^{5/2}} \left ( 1+ \frac{2}{\sqrt{\pi}}\right ) \X{v_1'-v_0'}.  
    \end{equation*}
\end{proof}

\begin{proposition}\label{Nv1Nv2bound}
    If $v_0,v_1 \in \Tilde{Y}_m^{even, \epsilon}$ then
    \begin{align*}
        \Xm{N[v_1]-N[v_0]} &\leq C'_{N}(\Xm{v_1'} + \Xm{v_0'})\Xm{v_1'-v_0'}
    \end{align*}
    with $C_N' := 2 \left (\sqrt{2\pi} b_3^C \epsilon + \frac{\pi \sqrt{2}}{\sqrt{m}} b_1 \right ) + 2\sqrt{m\pi}\left (1+ \frac{\pi^2}{m^2} \right ) \left ( b_5 \sqrt{2m}\epsilon + b_2 \right ) + \sqrt{\frac{\pi}{m}}$. 
\end{proposition}
\begin{proof}
    By periodicity and oddness of $N[v_i]$, it is enough to prove $$\X{N[v_1]-N[v_0]} \leq \frac{C_N'}{\sqrt{2m}} (\X{v_1'} + \X{v_0'})\X{v_1'-v_0'}.$$ We start by splitting the integral as in the proof of Proposition \ref{Nv}:
    \begin{equation}\label{Nineq}
    \begin{aligned}
        &\Ltwo{N[v_1]-N[v_0]} \leq \Ltwo{\int_0^{2\pi} C[v_1] v_1(x)v_1'(y)- C[v_0] v_0(x)v_0'(y)dy} \\
        & + \Ltwo{\int_0^{2\pi} C[v_1] v_1(y)v_1'(x) - C[v_0] v_0(y)v_0'(x)dy} + \Ltwo{\int_0^{2\pi} S[v_1] v_1(x)v_1(y)-S[v_0] v_0(x)v_0(y) dy} \\
        & + \Ltwo{\int_0^{2\pi} S[v_1] v_1'(x)v_1'(y)-S[v_0] v_0'(x)v_0'(y) dy} + \Ltwo{v_1v_1'-v_0v_0'}. 
    \end{aligned}
    \end{equation}
    We now bound each term one by one, starting with the last one
    \begin{equation}\label{Nineq1}
        \Ltwo{v_1v_1'-v_0v_0'} \leq \Y{v_1} \Ltwo{v_1'-v_0'} + \Y{v_1-v_0}\Ltwo{v_0'} \leq \frac{\sqrt{\pi}}{m\sqrt{2}} (\Ltwo{v_1'}+ \Ltwo{v_0'}) \Ltwo{v_1'-v_0'}.
    \end{equation}
    Now we proceed to the first and second term, which are bounded the same way.

    \begin{align*}
        &\Ltwo{\int_0^{2\pi} C[v_1] v_1(x)v_1'(y)- C[v_0] v_0(x)v_0'(y)dy} \\
        &\leq \Ltwo{\int_0^{2\pi} ( C[v_1] - C[v_0])v_1(x)v_1'(y) dy} + \Ltwo{\int_0^{2\pi} C[v_0](v_1(x)-v_0(x))v_1'(y)dy} \\
        & + 
        \Ltwo{\int_0^{2\pi} C[v_0]v_0(x)(v_1'(y)-v_0'(y))dy} \\
        &\leq \sqrt{\frac{\pi}{m}} b_3^C \epsilon (\Ltwo{v_1'}+ \Ltwo{v_0'})\Ltwo{v_0'-v_1'}+ \frac{\pi}{m} b_1 \Ltwo{v_1'-v_0'} \Ltwo{v_1'}  + \frac{\pi}{m} b_1 \Ltwo{v_0'} \Ltwo{v_1'-v_0'} \\
        &\leq \left (\sqrt{\frac{\pi}{m}} b_3^C \epsilon + \frac{\pi}{m} b_1 \right )(\Ltwo{v_1'}+ \Ltwo{v_0'})\Ltwo{v_1'-v_0'}.
    \end{align*}
    For the second and the third we used the bounds of the second term in \eqref{Nbound} but changing $v$ to $v_1-v_0, v_0$ respectively and $v'$ to $v_1', v_1'-v_0'$ respectively. For the first term we have used a modified version of Lemma \ref{CSvCS0linbound}, instead of having $C[v]-C[0]$ we have $C[v_1]-C[v_0]$ so using $v_\lambda = (1-\lambda)v_0 + \lambda v_1$ and following the same procedure as in Lemma \ref{Sv1Sv2Linfbound} we obtain the same bounds than in Lemma \ref{CSvCS0linbound} but with $\X{v_1'-v_0'}$ instead of $\X{v'}$:
    
    \begin{equation}\label{Cv1Cv2bound}
    \begin{aligned}
        \Ltwo{\int_0^{2\pi} ( C[v_1] - C[v_0])v_1(x)v_1'(y) dy} \ & \leq \ \Y{v_1} b_3^C \X{v_1'-v_0'} \X{v_1'} \ \\
        &\leq \ \sqrt{\frac{\pi}{m}}b_3^C \epsilon \left ( \X{v_1'}+\X{v_0'}\right )\X{v_1'-v_0'},
    \end{aligned}
    \end{equation}
    where we have used Lemma \ref{vLinftyv'Ltwobound} for the $L^\infty-\dot{H}^1$ bound, that $\X{v_1'} \leq \sqrt{2m}\epsilon$ and that $\X{v_1'} \leq {\X{v_1'}+\X{v_0'}}$.

    The second term of the original inequality \eqref{Nineq} is bounded by the same value in a very similar way.
    
    Now we have to bound the third and fourth term of \eqref{Nineq}, we are going to bound the fourth term and after use the same inequality for the third adding the necessary constants. By triangular inequality:
    \begin{align*}
   &\Ltwo{\int_0^{2\pi} S[v_1] v_1'(x)v_1'(y)- S[v_0] v_0'(x)v_0'(y)dy} \\ 
        &\leq \Ltwo{\int_0^{2\pi} (S[v_1] - S[v_0])v_1'(x)v_1'(y) dy} + \Ltwo{\int_0^{2\pi} S[v_0](v_1'(x)-v_0'(x))v_1'(y)dy} \\
        & + \Ltwo{\int_0^{2\pi} S[v_0]v_0'(x)(v_1'(y)-v_0'(y))dy} \\
        &\leq b_5  \Ltwo{v_1'-v_0'} \Ltwo 1 \Ltwo{v_1'}^2 + b_2 \Ltwo 1 \Ltwo{v_1'-v_0'} \Ltwo{v_1'}   + b_2 \Ltwo 1 \Ltwo{v_0'} \Ltwo{v_1'-v_0'} \\
        &\leq \sqrt{2\pi} \left ( b_5 \sqrt{2m}\epsilon + b_2  \right ) \left ( \Ltwo{v_1'} + \Ltwo{v_0'} \right ) \Ltwo{v_1'-v_0'}.
    \end{align*}   For the first term we used the uniform bound from Lemma \ref{Sv1Sv2Linfbound}. For the second and third bound we used the $L^\infty$ bound on $S[v_i]$ found in Lemma \ref{CSbounds}. We have also used again that $\X{v_1'} \leq \sqrt{2m} \epsilon$ and that $\X{v_1'}\leq \X{v_1'}+\X{v_0'}$.

    For the third term in \eqref{Nineq} we use the same bound but end up getting $v_i$ instead of $v_i'$. Using that 
    \begin{equation*}
        \left ( \Ltwo{v_1} + \Ltwo{v_0} \right ) \Ltwo{v_1-v_0} \leq 2 \pi \left ( \Y{v_1} + \Y{v_0} \right ) \Y{v_1-v_0} \leq \frac{\pi^2}{m^2} \left ( \Ltwo{v_1'} + \Ltwo{v_0'} \right ) \Ltwo{v_1'-v_0'}
    \end{equation*}
    we get the bound for the third term.

    Adding all the constants together 
    \begin{align*}
        &\Ltwo{N[v_1]-N[v_0]}  \\
        &\leq \left ( 2\left (\sqrt{\frac{\pi}{m}} b_3^C \epsilon + \frac{\pi}{m} b_1 \right ) + \sqrt{2\pi} \left ( b_5 \sqrt{2m} \epsilon + b_2 \right ) + \frac{\pi^2}{m^2} \sqrt{2\pi}\left ( b_5 \sqrt{2m}\epsilon + b_2  \right ) + \frac{\sqrt{\pi}}{m \sqrt{2}}\right ) \\
        & \times \left ( \Ltwo{v_1'} + \Ltwo{v_0'} \right ) \Ltwo{v_1'-v_0'}   \\
        &\leq \frac{C_N'}{\sqrt{2m}} \left ( \Ltwo{v_1'} + \Ltwo{v_0'} \right ) \Ltwo{v_1'-v_0'}, 
    \end{align*}

    as we wanted to prove.
\end{proof}

\begin{proposition}\label{T-T_1v1v2}
    If $v_0,v_1 \in \Tilde{Y}_m^{even, \epsilon}$ then
    \begin{align*}
        \Xm{(T-T_1) [v_1] - (T-T_1) [v_0]} &\leq C'_{T}(\Xm{v_1'} + \Xm{v_0'}) \Xm{v_1'-v_0'}
    \end{align*}
    with $ C'_{T} := 2( b_3^C+b_3^S) \left (\sqrt{2m}M_{R_0} + M_{R_0'} \frac{\pi\sqrt{2}}{\sqrt{m}} \right ) $ and $M_{R_0'}$ is defined in \ref{MdR0}.
\end{proposition}
\begin{proof}
    By periodicity and oddness of $T,T_1, v_i'$, it is enough to prove that $\Ltwo{(T-T_1) [v_1]-(T-T_1) [v_0]} \leq \frac{C_T}{\sqrt{2m}} \Ltwo{v_1'-v_0'}^2 $ . By definition and triangular inequality
    \begin{align*}
         &\Ltwo{(T-T_1) [v_1] - (T-T_1) [v_2]} \leq \Ltwo{\int_0^{2\pi}  v_1'(y)R_0(x)(C[v_1]-C[0]) -  v_0'(y)R_0(x)(C[v_0]-C[0]) dy}  \\
         &+\Ltwo{\int_0^{2\pi}  v_1'(x)R_0(y)(C[v_1]-C[0]) -  v_0'(x)R_0(y)(C[v_0]-C[0]) dy} +  6\text{ more terms}.\\
    \end{align*}
    From the "6 more terms" two terms come from changing $v_i' R_0$ by $v_i R_0'$, and the other four come from changing $C$ by $S$. We are going to bound one term, and all other seven terms are analogous with the same bounds except for the use of $M_{R_0}$ or $M_{R_0'}$, as well as taking into account that to go from $\Ltwo{v}$ to $\Ltwo{v'}$ we have some constants to add and use $b_3^C$ or $b_3^S$ accordingly. 

    We start bounding the first term, adding and subtracting $R_0(x)C[v_1]v_0'(y)$ 
    \begin{align*}
         \text{First term} \leq \Ltwo{\int_0^{2\pi}  R_0(x)(C[v_1]-C[0])(v_1'(y)-v_0'(y)) dy} + \Ltwo{\int_0^{2\pi}   v_0'(y)R_0(x)(C[v_0]-C[v_1]) dy}.
    \end{align*}
    In order to bound the first summand we can apply Lemma \ref{CSvCS0linbound} to bound
    \begin{equation}\label{Cv1Cv0_v1'-v0'}
        \Ltwo{\int_0^{2\pi}  R_0(x)(C[v_1]-C[0])(v_1'(y)-v_0'(y)) dy} \leq M_{R_0} b_3^C \Ltwo{v_1'} \Ltwo{v_1'-v_0'}. 
    \end{equation}
    For the second summand we would need to compute the difference $$\log(A[v_1])-\log(A[v_0]) = \int_0^1 \frac{d}{d\lambda}  \log(A[v_0 + \lambda (v_1 - v_0)])  d\lambda,$$ but we have done this in equation \eqref{Cv1Cv2bound} of Proposition \ref{Nv1Nv2bound}. The only difference is that we have $R_0(x)$ instead of $v_1(x)$, so we use $M_{R_0}$ to do the $L^\infty$ bound

    \begin{equation}\label{v_0'R0_Cv0-Cv1}
        \Ltwo{\int_0^{2\pi}   v_0'(y)R_0(x)(C[v_0]-C[v_1]) dy} \leq M_{R_0} b_3^C \Ltwo{v_0'-v_1'} \Ltwo{v_0'}. 
    \end{equation}
    Adding both inequalities, the total bound of the first term is:
    \begin{equation*}
        \Ltwo{\int_0^{2\pi}  v_1'(y)R_0(x)(C[v_1]-C[0]) -  v_0'(y)R_0(x)(C[v_0]-C[0]) dy} \leq M_{R_0} b_3^C (\Ltwo{v_1'} + \Ltwo{v_0'})\Ltwo{v_1'-v_0'}. 
    \end{equation*}
    For the other seven terms we can follow the same exact procedure, but if we have a $v_i(x)R_0'(y)$ factor instead of $v_i'(x) R_0(y)$ the bound we get is $b_3^\alpha M_{R_0'} \frac{\pi}{m}$ instead of $b_3^\alpha M_{R_0}$, because in these cases we have
    \[
        \Ltwo{v} \leq \frac{\pi}{m} \Ltwo{v'}
    \]
    for $v = v_1-v_0$ in \eqref{Cv1Cv0_v1'-v0'} and $v = v_0$ in \eqref{v_0'R0_Cv0-Cv1}.
    
    %\begin{align*}
    %    &b_3^\alpha M_{R_0'} \Ltwo{v_1'} \Ltwo{v_1-v_0} \leq  b_3^\alpha M_{R_0'} \frac{\pi}{m} \Ltwo{v_1'} \Ltwo{v_1'-v_0'}, \\
    %    &b_3^\alpha M_{R_0'} \Ltwo{v_1} \Ltwo{v_1'-v_0'} \leq  b_3^\alpha M_{R_0'} \frac{\pi}{m} \Ltwo{v_1'} \Ltwo{v_1'-v_0'},
    %\end{align*}
    We have two terms of each type, so the final bound is
    \begin{align*}
        \Ltwo{(T-T_1) [v_1] - (T-T_1) [v_0]} &\leq  2\left ( b_3^C+b_3^S \right )\left (M_{R_0} + M_{R_0'} \frac{\pi}{m} \right )  \left (\Ltwo{v_1'} + \Ltwo{v_0'}\right ) \Ltwo{v_1'-v_0'} \\
        &\leq \frac{C_T'}{\sqrt{2m}} \left (\Ltwo{v_1'} + \Ltwo{v_0'}\right ) \Ltwo{v_1'-v_0'}.
    \end{align*}
\end{proof}

\begin{lemma}\label{CSv1CSv2secondorderbound}
    For all $v_0,v_1 \in \Tilde{Y}_m^{even, \epsilon}$ we have
    \begin{equation*} 
    \begin{aligned}
        &\Ltwo{\int_{-\pi}^{\pi} \left |\left ( \log(A[v_1])- \frac{d}{d\lambda} \log(A[\lambda v_1])|_{\lambda=0} - \left (\log(A[v_0])   - \frac{d}{d\lambda}\log(A[\lambda v_0])|_{\lambda = 0} \right )\right )(x,x-z) h^\alpha(z) \right | dz  } \\
        &\leq 4 \pi \Omega \ b_6^\alpha (\X{v_0'} + \X{v_1'})\Ltwo{v_1'-v_0'},  
    \end{aligned}
    \end{equation*}
    with $h^C(z) :=  \cos(z) p_m(z)$, $h^S(z) := \sin(z)$, and
    %\begin{equation*}
     %   b_6^\alpha := \left ( \frac{1}{2^{9/2}\sqrt{\pi} m^{3/2} m_R^2} + \frac{M_R^2 }{2^{5/2}m^{3/2}m_R^4}\right ) I_5^\alpha(m) +\frac{M_R}{2^{7/2}m^{3/2}m_R^3} I_4^\alpha(m) + \frac{\sqrt{\pi}}{2^{3/2} m^2m_R^2}\left ( 1 + \frac{M_R^2}{m_R^2} \right) I_3^\alpha(m)
    %\end{equation*}   
    \begin{equation*}
        b_6^\alpha := \left ( \left ( \frac{1}{2^{11/2} \sqrt{\pi}m m_R^2} + \frac{M_R^2 }{2^{7/2}m^{3/2}m_R^4}\right )I_5^\alpha(m) + \frac{\sqrt{\pi}(m_R^2+M_R^2)}{2^{5/2}m^2 m_R^4} I_3^\alpha(m) + \frac{M_R}{2^{9/2}m^{3/2}m_R^3} I_4^\alpha(m) \right ),
    \end{equation*}
    where $I_3^\alpha, I_4^\alpha, I_5^\alpha$ are defined in Lemmas \ref{Integralbound3}, \ref{Integralbound4}, \ref{Integralbound5}.
\end{lemma}

\begin{proof}
    We begin by computing the pointwise difference error of the linearization, defining $v_\mu = (1-\mu)v_0 + \mu v_1$ then
    \begin{equation}\label{logAv1dlogAv1-logAv2dlogAv2}
    \begin{aligned}
        &\left (\log(A[v_1])- \frac{d}{d\lambda} \log(A[\lambda v_1])|_{\lambda=0} \right )(x,y) - \left (\log(A[v_0])   - \frac{d}{d\lambda}\log(A[\lambda v_0])|_{\lambda = 0} \right )(x,y) =\\
        &=  \int_0^1 \frac{d}{d\mu} \left ( \log(A[v_\mu])- \frac{d}{d\lambda} \log(A[\lambda v_\mu])|_{\lambda=0} \right )(x,y) d\mu \\
        &= \int_0^1 \left (\frac{\frac{d}{d\mu} A[v_\mu]|_\mu}{A[v_\mu]} - \frac{\frac{d^2}{d\mu d\lambda}A[\lambda v_\mu] |_{\lambda = 0, \mu = \mu}}{A[0]}\right ) (x,y)d\mu \\
        &= \int_0^1 \left (\frac{A[0]\frac{d}{d\mu} A[v_\mu] |_{\mu} - A[v_\mu] \frac{d^2}{d\mu d\lambda} A[\lambda v_\mu]|_{0, \mu}}{A[v_\mu] A[0]} \right )(x,y) d\mu.
    \end{aligned}
    \end{equation}
We proceed to compute the corresponding derivatives of $A$. Doing the same calculation as in \eqref{d_lambdaAv1v2} and changing $\lambda$ by $\mu$ we have 
\begin{equation}
    \begin{aligned}
        \frac{d}{d\mu} A[v_\mu]|_\mu &= 2\left (R_\mu(x)-R_\mu(y)\right )(v_1(x)-v_1(y) + v_0(y)-v_0(x)) + \\
        &+ 4\sin^2\left (\frac{x-y}{2}\right ) \left (R_\mu(x)(v_1(y)-v_0(y)) + R_\mu(y)(v_1(x)-v_0(x))\right ).
    \end{aligned}
\end{equation}
For the term $\frac{d^2}{d\mu d\lambda} A[\lambda v_\mu]|_{0,\mu}$, the first derivative $\frac{d}{d\lambda} A[\lambda v_\mu]|_0$ is the expression obtained in \eqref{d_lambdaA} changing $v$ by $v_\mu$ and $R_\lambda$ by $R_0$. Deriving the resulting expression of $\frac{d}{d\lambda} A[\lambda v_\mu]|_0$ with respect to $\mu$, we get
\begin{equation}
    \begin{aligned}
        \frac{d^2}{d\mu d\lambda} A[\lambda v_\mu] |_{0, \mu} &= 2\left (R_0(x)-R_0(y)\right )(v_1(x)-v_1(y) + v_0(y)-v_0(x)) + \\
        &+ 4\sin^2\left (\frac{x-y}{2}\right ) \left (R_0(x)(v_1(y)-v_0(y)) + R_0(y)(v_1(x)-v_0(x))\right ),
    \end{aligned}
\end{equation}
where we have used that $\frac{d}{d\mu} v_\mu = v_1 -v_0$. Putting it together
\begin{align*}
    &\left (A[0]\frac{d}{d\mu} A[v_\mu]|_\mu  - A[v_\mu] \frac{d^2}{d\mu d\lambda} A[\lambda v_\mu] |_{0, \mu}\right ) (x,y) = \\
    &= 2(v_1(x)-v_1(y)-v_0(x)+v_0(y)) \underbrace{\left ( A[0](R_\mu(x)-R_\mu(y) )-A[v_\mu](R_0(x)-R_0(y))\right )}_{S_1} + \\
    &+ 4\sin^2 \left ( \frac{x-y}{2}\right ) \left  [ (v_1(y)-v_0(y)) \underbrace{(A[0]R_\mu(x)-A[v_\mu]R_0(x))}_{S_2} + (v_1(x)-v_0(x))\underbrace{(A[0]R_\mu(y)-A[v_\mu]R_0(y))}_{S_3} \right ].
\end{align*}
For every $\mu \in [0,1]$ we are going to begin by bounding $S_1, S_2, S_3$ separately. Starting by the first one, arranging the terms appropriately
\begin{align*}
    |S_1(x,y)| &= \left | 2(R_0(x)-R_0(y))(R_\mu(x)-R_\mu(y))[R_\mu(x)-R_\mu(y) -R_0(x)+R_0(y) ] \right .+\\
    &+ \left . 4 \sin^2\left (\frac{x-y}{2}\right ) \left [R_\mu(y)R_0(y)(R_0(x)-R_\mu(x)) + R_\mu(x)R_0(x)(R_\mu(y)-R_0(y))\right ] \right |\\
    &\leq 2|R_0(x)-R_0(y)||R_\mu(x)-R_\mu(y)|\frac{\X{v_\mu'}}{\sqrt{2m}} \sqrt{p_m(x-y)} +  4 \sin^2\left (\frac{x-y}{2}\right ) M_R^2 \frac{\sqrt{2\pi}}{m} \X{v_\mu'}\sqrt{p_m(x-y)},
\end{align*}
where we have used that $R_\mu-R_0 = v_\mu$, the bound from Lemma \ref{bound2_v(x)-v(y)} applied to $v_\mu(x)-v_\mu(y)$ and the $L^\infty-\dot{H}^1$ bound from Lemma \ref{vLinftyv'Ltwobound} applied to $v_\mu$. We have also used that $R_\mu$ satisfies $m_R < R_\mu < M_R$ via Lemma \ref{boundR0}.

Dividing by $A[v_\mu] A[0]$ and lower bounding it, we get
\begin{equation}\label{S1/A_bound}
    \frac{|S_1|}{A[v_\mu]A[0]}(x,y) \leq \frac{\X{v_\mu'}\sqrt{p_m(x-y)}}{\sqrt{2m}8m_R^2 \sin^2\left (\frac{x-y}{2}\right )} + \frac{M_R^2 \X{v_\mu'} \sqrt{2\pi} \sqrt{p_m(x-y)}}{m 4 m_R^4 \sin^2\left (\frac{x-y}{2}\right ) }.
\end{equation}

Rewriting $S_2$ appropriately
\begin{align*}
    |S_2(x,y)| &= \left |v_\mu(x)(R_0(x)-R_0(y))^2 + R_0(x)(R_0(x)-R_0(y)+R_\mu(x)-R_\mu(y))(R_0(x)-R_0(y)-R_\mu(x)+R_\mu(y)) \right . \\
    &+ \left .4 \sin^2\left (\frac{x-y}{2}\right )  R_0(x)R_\mu(x)\left (R_\mu(y)-R_0(y)\right ) \right | \\
    &\leq \frac{\sqrt{\pi}}{m\sqrt{2}}\X{v_\mu'} (R_0(x)-R_0(y))^2 + M_R(|R_0(x)-R_0(y)|+|R_\mu(x)-R_\mu(y)|) \frac{\X{v_\mu'}}{\sqrt{2m}} \sqrt{p_m(x-y)} + \\
    &+4 \sin^2\left (\frac{x-y}{2}\right ) M_R^2 \frac{\sqrt{\pi}}{m\sqrt{2}} \X{v_\mu'}.
\end{align*}
Dividing by $A[v_\mu] A[0]$ and lower bounding it, we get
\begin{equation}\label{S2/A_bound}
    \frac{|S_2|}{A[v_\mu]A[0]}(x,y) \leq \frac{\sqrt{\pi} \X{v_\mu'}}{m\sqrt{2} 4 m_R^2 \sin^2\left (\frac{x-y}{2}\right )} + \frac{M_R \X{v_\mu'} \sqrt{p_m(x-y)}}{\sqrt{2m} 16 m_R^3 \sin^3\left (\frac{x-y}{2}\right )} + \frac{M_R^2 \sqrt{\pi} \X{v_\mu'}}{m  4 \sqrt{2} m_R^4 \sin^2\left (\frac{x-y}{2}\right )}.
\end{equation}
The bound for $|S_3|/(A[v_\mu]A[0])$ is the same as the one for $S_2$ because it only changes one factor $R_0(x), R_\mu(x)$ by $R_0(y),R_\mu(y)$ and we bound both by $M_R$.

Putting together the equality \eqref{logAv1dlogAv1-logAv2dlogAv2} with the bounds \eqref{S1/A_bound}, \eqref{S2/A_bound} and also using Lemmas \ref{bound2_v(x)-v(y)}, \ref{vLinftyv'Ltwobound} for $v_1-v_0$ appropriately, we get
\begin{equation}\label{logAv1dlogAv1-logAv0dlogAv0_pointwise}
\begin{aligned}
    &\left (\log(A[v_1])- \frac{d}{d\lambda} \log(A[\lambda v_1])|_{\lambda=0} \right )(x,x-z) - \left (\log(A[v_0])   - \frac{d}{d\lambda}\log(A[\lambda v_0])|_{\lambda = 0} \right )(x,x-z) \\
    &\leq \int_0^1 \left( \frac{2 \X{v_1'-v_0'}}{\sqrt{2m}} \sqrt{p_m(z)} \frac{|S_1|}{A[v_\mu]A[0]} + \frac{\sqrt{\pi}}{m\sqrt{2}} \X{v_1'-v_0'} 4 \sin^2 \left ( \frac{z}{2}\right ) \left ( \frac{|S_2|+  |S_3|}{A[v_\mu]A[0]}\right )\right )(x,x-z) d\mu \\
    &\leq   \left ( \frac{p_m(z)}{\sin^2\left ( \frac{z}{2}\right )} \left ( \frac{1}{2^3 m m_R^2} + \frac{M_R^2 \sqrt{\pi}}{2m^{3/2}m_R^4}\right )  + \frac{\pi(m_R^2+M_R^2)}{m^2 m_R^4} + \frac{\sqrt{p_m(z)}}{\left |\sin\left ( \frac{z}{2}\right )\right |}\frac{M_R\sqrt{\pi}}{4m^{3/2}m_R^3}\right )\Ltwo{v_1'-v_0'} \int_0^1\Ltwo{v_\mu'}  d\mu.
\end{aligned}
\end{equation}
We can then use that $\Ltwo{v_\mu'} \leq (1-\mu) \Ltwo{v_0'} + \mu \Ltwo{v_1'}$ to get
\[
    \int_0^1 \Ltwo{v_\mu'} d\mu \leq \Ltwo{v_0'} \int_0^1 (1-\mu) d\mu + \Ltwo{v_1'} \int_0^1 \mu d\mu  = \frac{1}{2} (\Ltwo{v_0'} + \Ltwo{v_1'}) .
\]

Returning to the bound we wanted to prove, using the obtained bound \eqref{logAv1dlogAv1-logAv0dlogAv0_pointwise}\begin{align*}
    &\Ltwo{\int_{-\pi}^\pi \left |\left ( \frac{A[0]\frac{d}{d\mu} A[v_\mu] |_{\mu} - A[v_\mu] \frac{d^2}{d\mu d\lambda} A[\lambda v_\mu]|_{0, \mu}}{A[v_\mu] A[0]}\right )(x,x-z) h^\alpha(z) \right | dz  } \\
    &\leq  \left \|\left ( \frac{1}{2^3 m m_R^2} + \frac{M_R^2 \sqrt{\pi}}{2m^{3/2}m_R^4}\right ) \int_{-\pi}^{\pi} \frac{p_m(z)|h^\alpha(z)|}{\sin^2\left ( \frac{z}{2}\right )} dz + \frac{\pi(m_R^2+M_R^2)}{m^2 m_R^4}\int_{-\pi}^\pi |h^\alpha(z)|dz  \right .\\
    & \left .+ \frac{M_R\sqrt{\pi}}{4m^{3/2}m_R^3} \int_{-\pi}^\pi \frac{h^\alpha(z)\sqrt{p_m(z)}}{\left | \sin \left ( \frac{z}{2}\right )\right |} dz  \right \|_{L^2}  \frac{1}{2} (\X{v_0'}+\X{v_1'})\X{v_1'-v_0'}\\
    &\leq \left ( \left ( \frac{\sqrt{\pi}}{2^{7/2} m m_R^2} + \frac{M_R^2 \pi}{2^{3/2}m^{3/2}m_R^4}\right )I_5^\alpha(m) + \frac{\pi^{3/2}(m_R^2+M_R^2)}{\sqrt{2}m^2 m_R^4} I_3^\alpha(m) + \frac{M_R\pi}{2^{5/2}m^{3/2}m_R^3} I_4^\alpha(m) \right ) \\
    & \times (\X{v_0'} + \X{v_1'}) \X{v_1'-v_0'} \\
    &= 4\pi \Omega b_6^\alpha
\end{align*}
where we used the bounds for the integrals from Lemmas \ref{Integralbound3}, \ref{Integralbound4}, \ref{Integralbound5}. \end{proof}

\begin{proposition} \label{Ev1-E1v1-Ev2+E1v2}
    If $v_0,v_1 \in \Tilde{Y}_m^{even, \epsilon}$ then
    \begin{align*}
        \Xm{E[v_1]-E_1[v_1]-E[v_0]+E_1[v_0]} \leq C_E' ( \Xm{v_0'} + \Xm{v_1'}) \Xm{v_1'-v_0'}
    \end{align*}
    with $C_E' := \sqrt{2m} \left ( M_{R_0}M_{R_0''} + M_{R_0'}^2 \right ) b_6^C+ \sqrt{2m} \left (M_{R_0'}^2+M_{R_0}^2 \right )  b_6^S $ and $M_{R_0'}, M_{R_0''}$ are defined in Lemmas \ref{MdR0}, \ref{MddR0}.
\end{proposition}
\begin{proof}
    By periodicity and oddness of $E, E_1, v_i'$, it is enough to prove $\X{E[v_1]-E_1[v_1]-E[v_0]+E_1[v_0]} \leq \frac{C_E'}{\sqrt{2m}} ( \X{v_0'} + \X{v_1'}) \X{v_1'-v_0'}$. Recalling Definitions  \ref{OPv} and \ref{T1E1_def}, and using the triangle inequality
    \begin{align*}
        &\Ltwo{E[v_1]-E_1[v]+E[v_0]-E_1[v_0]} \\
        &\leq \Ltwo{\int (C[v_1] - \frac{d}{d\lambda} C[\lambda v_1] |_{\lambda = 0} -C[v_0] + \frac{d}{d\lambda} C[\lambda v_0] |_{\lambda = 0} )(x,x-z) (R_0(x)R_0'(x-z) - R_0(x-z)R_0'(x) )dz} \\
        &+ \Ltwo{\int (S[v_1] - \frac{d}{d\lambda} S[\lambda v_1] |_{\lambda = 0}  -S[v_0] + \frac{d}{d\lambda} S[\lambda v_0] |_{\lambda = 0})(x,x-z) (R_0(x)R_0(x-z) + R_0'(x)R_0'(x-z)) dz} . 
    \end{align*}
    We are going to proceed in the same way as in Proposition \ref{Ev1-E1v1-Ev2+E1v2}. First we get an extra cancellation in the first term:
    \begin{align*}
        |R_0(x)R_0'(x-z) - R_0(x-z)R_0'(x)| 
        &\leq |R_0(x)| |R_0'(x-z) - R_0'(x)| + |R_0'(x)| |R_0(x-z) - R_0(x)| \\
        &\leq M_{R_0} M_{R_0''} p_m(z) + M_{R_0'}^2 p_m(z),
    \end{align*}
    where we have used that both $R,R'$ are $2\pi/m$ periodic, even and odd respectively and have their derivatives bounded in order to apply Lemma \ref{f(x)-f(y)}. %On the second term we can directly apply $L^\infty$ bounds to $R, R'$.
    
    Applying Lemma \ref{CSv1CSv2secondorderbound} with $h_\alpha(z) = \cos(z)p_m(z)$ and $h_\alpha =\sin(z)$ we get
    \begin{align*}
        &\Ltwo{E[v_1]-E_1[v_1]-E[v_0]+E_1[v_0]} \\ 
        &\leq \frac{1}{4\pi \Omega}\Ltwo{\int \left |\left (\log(A[v_1])-  \frac{d}{d\lambda} \log(A[\lambda v_1])|_{\lambda = 0} -\log(A[v_0]) + \frac{d}{d\lambda} \log(A[\lambda v_0])|_{\lambda = 0} \right )(x,x-z) \cos(z)p_m(z) \right | dz  } \\
        & \times \left ( M_{R_0}M_{R_0''} + M_{R_0'}^2 \right ) \\
        &+ \frac{1}{4\pi \Omega}\Ltwo{\int \left |\left (\log(A[v_1]) - \frac{d}{d\lambda} \log(A[\lambda v_1])|_{\lambda = 0} - \log(A[v_0]) + \frac{d}{d\lambda} \log(A[\lambda v_0])|_{\lambda = 0} \right )(x,x-z) \sin(z) \right | dz  } \\
        & \times \left (M_{R_0}^2 + M_{R_0'}^2 \right )  \\
        &\leq \left (\left ( M_{R_0}M_{R_0''} + M_{R_0'}^2 \right ) b_6^C + \left (M_{R_0}^2 + M_{R_0'}^2 \right ) b_6^S \right )(\X{v_0'}+\X{v_1'})\X{v_1'-v_0'}
        \\
        &\leq \frac{C_E'}{\sqrt{2m}} \left(\X{v_0'}+\X{v_1'} \right ) \X{v_1'-v_0'}. 
    \end{align*}
\end{proof}

\section{Further regularity and non-convexity of the solution}
\label{Sec:Regularity}

In this section we are going to prove that our solution is in fact analytic. First we will to use a bootstrapping argument in 2 steps to prove $C^\infty$ regularity. %We are going to see that $\Y{R'}$ is bounded and then that $\Y{\partial^k R}$ is bounded for $k \geq 2$.
To prove analyticity we will use the free boundary elliptic problem formulation, where it follows that if the boundary is $C^2$ then it is analytic.

Throughout this section, let $R = R_0 + v$ be the solution to the equation \eqref{eqR} such that $v \in \Tilde{Y}_{m}^{even, \epsilon}$.

\subsection{Regularity up to $C^\infty$}
We will start by proving some lemmas involving the coefficient in front of the highest derivative in the bootstrap.

\begin{lemma}
    We have the following bounds
    \begin{enumerate}
        \item \label{CvC0infbound} $\Y{\int (C[v]-C[0])(x,y) R(y) dy} \leq  \epsilon b_3^C \sqrt{2m} M_R$,
        \item \label{SvS0infbound}$\Y{\int (S[v]-S[0])(x,y) R'(y) dy} \leq \epsilon\frac{\sqrt{\pi}+2}{\Omega 4 m_R \sqrt{m\pi}}  \left ( 2\pi M_{R_0'} + \epsilon2\sqrt{\pi m} \right )$,
    \end{enumerate}
\end{lemma}
where $b_3^C$ is defined in Lemma \ref{CSvCS0linbound} and $M_{R_0'}$ in \ref{MdR0}.
\begin{proof}
    First we prove \ref{CvC0infbound}. Using the exact same procedure as in the inequalities \eqref{prooflogAvlogA0} from Lemma \ref{CSvCS0linbound} but using $L^{\infty}_x$ estimates instead of $L_x^2$ we get the same inequality
    \begin{equation*}
        \Y{\int (C[v]-C[0])(x,y) R(y) dy} \leq b_3^C \X{v'} \Y{R} \leq \epsilon b_3^C \sqrt{2m} M_R.
    \end{equation*}
    %where we have used also that $\X{v'} \leq \epsilon\sqrt{2m}$ and that $\Y{R} \leq M_R$.

    To prove \ref{SvS0infbound} we cannot use the same argument because we don't know yet that $R'$ is in $L^\infty$. Using the pointwise estimate from Lemma \ref{Sv1Sv2Linfbound} with $v_0=0$  we get
    \begin{align*}
        \left | S[v](x,x-z)-S[0](x,x-z)\right | &\leq b_5 \X{v'} \leq \epsilon \frac{\sqrt{\pi}+2}{4 \Omega m_R \sqrt{m\pi}}, 
    \end{align*}
    Where we have used $\X{v'} \leq \epsilon \sqrt{2m}$ and the definition of $b_5$. Now we bound the integral
    \begin{align*}
        \Y{\int_{-\pi}^\pi (S[v]-S[0])(x,y) R'(y) dy} &\leq \Y{S[v]-S[0]} \X{1} \X{R'} \\
        &\leq \epsilon\frac{\sqrt{\pi}+2}{4 \Omega m_R \sqrt{m\pi}} \sqrt{2\pi} \left ( \sqrt{2\pi} M_{R_0'} + \X{v'} \right ) \\
        &\leq \epsilon\frac{\sqrt{\pi}+2}{4 \Omega m_R \sqrt{m\pi}}  \left ( 2\pi M_{R_0'} + \epsilon2\sqrt{\pi m} \right ),
    \end{align*}
    where we have used that $\Y{R_0} \leq M_{R_0'}$.
\end{proof}

\begin{lemma}
    We have the following bound
    \begin{equation*}
        \left | v(x)+ \int C[v]v(y)- S[v]v'(y) dy\right | \leq \epsilon \left (b_1\frac{\pi}{\sqrt{m}} + b_2 2\sqrt{m\pi} + \sqrt{\frac{\pi}{m}} \right )
    \end{equation*}
    where $b_1, b_2$ are defined in Lemma \ref{CSbounds}.
\end{lemma}
\begin{proof}
    Using the bounds $\Ltwoinf{C}, \Linf{S}$ bounds from Lemma \ref{CSbounds}, the $L^\infty-\dot{H}^1$ bound from Lemma \ref{vLinftyv'Ltwobound} and that $\X{v'} \leq \epsilon\sqrt{2m}$
    \begin{align*}
        \left | v(x)+ \int C[v]v(y)- S[v]v'(y) dy\right | &\leq \left (  \Y{v} + \Ltwoinf{C}\X{1}\Y{v} + \Linf{S} \X{1} \X{v'} \right ) \\
        &\leq \epsilon \left (b_1\frac{\pi}{\sqrt{m}} + b_2 2\sqrt{m\pi} + \sqrt{\frac{\pi}{m}} \right ).
    \end{align*}
\end{proof}

We have the following bound:
\begin{lemma}
\label{CTildeKone}
    The constant $C_{\Tilde{K}_1}$ 
    \begin{equation*}
            C_{\Tilde{K}_1} := C_{K_1} - \epsilon \left (b_1\frac{\pi}{\sqrt{m}} + b_2 2\sqrt{m\pi} + \sqrt{\frac{\pi}{m}} + b_3^C \sqrt{2m} M_R + \frac{\sqrt{\pi}+2}{\Omega 4 m_R \sqrt{m\pi}}  \left ( 2\pi M_{R_0'} + \epsilon2\sqrt{\pi m} \right )\right ) > 0,
    \end{equation*}
    where $C_{K_1}$ is defined in Lemma \ref{K1bound}.
\end{lemma}
\begin{proof}
The proof is computer-assisted and can be found in the supplementary material. We refer to Appendix \ref{appendix:implementation} for the implementation details.
\end{proof}

Building up on the previous Lemma, we can prove the following:

\begin{lemma}\label{K1tilde}
    Let us define $\Tilde{K}_1$ as 
    \begin{equation*}
        \Tilde{K}_1(x) = R(x)+ \int C[v]R(y)- S[v]R'(y) dy.
    \end{equation*}
    Then, we have $|\Tilde{K}_1(x)| \geq C_{\Tilde{K}_1} > 0$.
\end{lemma}
\begin{proof}
    The main idea is to use that $\Tilde{K}_1$ is very close to $K_1$, where $K_1$ was defined in \eqref{defK1}.
    \begin{align*}
        &\left |R(x)+ \int C[v]R(y)- S[v]R'(y) dy \right | \geq \\
        &\geq \left |K_1(x)\right | -\left | v(x)+ \int C[v]v(y)- S[v]v'(y) dy\right |  - \left | \int (C[v]-C[0])R(y) + (S[v]-S[0]) R'(y) dy \right | \\
        &\geq C_{K_1} - \epsilon \left (b_1\frac{\pi}{\sqrt{m}} + b_2 2\sqrt{m\pi} + \sqrt{\frac{\pi}{m}} + b_3^C \sqrt{2m} M_R + \frac{\sqrt{\pi}+2}{\Omega 4 m_R \sqrt{m\pi}}  \left ( 2\pi M_{R_0'} + \epsilon2\sqrt{\pi m} \right )\right ) = C_{\Tilde{K}_1} > 0, 
    \end{align*}
    which follows from Lemma \ref{CTildeKone}.
\end{proof}

\begin{proposition}\label{R_Cinfty}
    $R(x) \in C^\infty$.
\end{proposition}
\begin{proof}
    \textbf{First Step:}
    We can rewrite the equation \eqref{eqR} as
    \begin{align*}
        &\left (R(x)+ \int C[v]R(y)- S[v]R'(y) dy \right ) R'(x) =R(x)\int C[v]R'(y)-S[v]R(y) dy.
    \end{align*}
    Estimating the RHS we get
    \begin{align*}
        \left |R(x)\int C[v]R'(y)-S[v]R(y) dy \right |&\leq \Y R \left ( \Ltwoinf{C}\X{1} \X{R'} + \Linf{S}\Lone{1} \Y{R} \right ) \\
        &\lesssim 1+\X{R'}^2,
    \end{align*}
    then combining it with Lemma \ref{K1tilde} we have
    \begin{align*}
        C_{\Tilde{K}_1} |R'(x)| \lesssim 1+ \X{R'}^2.
    \end{align*}
    We can conclude that $R'$ is bounded in $L^\infty$.
    
    \textbf{Second Step:} Differentiating \eqref{eqR} once and arranging the terms properly, we have
    \begin{equation}\label{Reg2ndStepEq}
    \begin{aligned}
        \Tilde{K}_1(x) R''(x) &= \underbrace{R'(x) \left (-R'(x) + \int C[v]R'(y)-S[v]R(y) dy \right ) + \int C[v]h(x,y) - S[v]g(x,y) dy}_{Q_1(x)} + \\
        &+ \underbrace{\int \frac{\sin(x-y) \partial_x A(x,y)}{A(x,y)}h(x,y) dy}_{Q_2(x)} + \underbrace{\int \frac{\cos(x-y) \partial_x A(x,y)}{A(x,y)}g(x,y) dy}_{Q_3(x)}
    \end{aligned}
    \end{equation}
    where 
    \begin{align*}
        g(x,y) := R(x)R'(y) - R'(x)R(y), \\
        h(x,y) := R(x)R(y) + R'(x)R'(y).
    \end{align*}

    In this step we prove that if $\partial^{k+1} R \in L^\infty$ then $\partial^{k+2} R \in L^\infty$ for $k \geq 0$. The first step is to change to the variable $z = x-y$ inside the integrals as to not generate singularities. We are also going to repeatedly use the embedding $\Y{\partial^j R} \lesssim 1+\Y{\partial^k R}$ for $j\leq k$ all the time without specifying.

    First we prove some useful estimates for the derivatives of $h(x,x-z), g(x,x-z)$:
    
    \begin{equation}\label{hgReg_estimates}
    \begin{aligned}
        \left | \frac{d^j}{dx^j} h(x,x-z) \right | &= \left |\sum_{l= 0}^{j} \binom{j}{l} \left (\partial^l R(x) \partial^{j-l} R(x-z) + \partial^{l+1} R(x) \partial^{j-l+1} R(x-z) \right )\right  | \lesssim 1+\Y{\partial^{k+1} R}^2 \quad \text{ if } j \leq k, \\
        \left |\frac{d^j}{dx^j} g(x,x-z) \right  | &= \left |\sum_{l= 0}^{j} \binom{j}{l}\left(\partial^l R(x) \partial^{j-l+1} R(x-z) + \partial^{l+1} R(x) \partial^{j-l} R(x-z) \right )\right  | \lesssim 1+\Y{\partial^{k+1} R}^2 \quad \text{ if } j \leq k, \\
        \left |\frac{d^j}{dx^j} g(x,x-z) \right  | &= \left |\sum_{l= 0}^{j} \binom{j}{l} \left (\partial^l R(x) (\partial^{j-l+1} R(x-z)- \partial^{j-l+1} R(x) )+ \partial^{l+1} R(x) (\partial^{j-l} R(x)-\partial^{j-l} R(x-z)) \right ) \right  | \\
        &\lesssim \left ( 1+\Y{\partial^{k+1} R}\right )  \left (1+\Y{\partial^{k+2} R} \right ) |z|\quad \text{ if } j \leq k, \\
    \end{aligned}
    \end{equation}
    where in the last inequality we have used the Mean Value Theorem to get the extra factor of $|z|$.

    Next, we estimate the derivatives of $A(x,x-z)$ 
    \begin{equation}\label{Areg_estimates}
    \begin{aligned}
        \left |\frac{d^j}{dx^j} A(x,x-z) \right |&=\left | \sum_{l=0}^j \binom{j}{l}\left [(\partial^l R(x)-\partial^{l}R(x-z))(\partial^{j-l} R(x)-\partial^{j-l}R(x-z)) \right . \right .\\
        &\left . \left .+ 4\partial^l R(x) \partial^{j-l} R(x-z) \sin^2 \left ( \frac{z}{2}\right ) \right ]\right | \lesssim \left (1+\Y{\partial^{k+1} R}^2 \right ) |z|^2 \quad \text{ if } j \leq k, \\
        \left |\frac{d^j}{dx^j} \left (\partial_x A \right)(x,x-z) \right |&= \left |\sum_{l=0}^j \binom{l}{j} \left [(\partial^l R(x)-\partial^{l}R(x-z))\partial^{j+1-l} R(x) + 4\partial^{l+1} R(x) \partial^{j-l} R(x-z) \sin^2 \left ( \frac{z}{2}\right ) \right .\right .\\
        &\left . \left . + 2 \partial^{l+1} R(x) \partial^{j-l}R(x-z) \sin(z)\right ]\right |\lesssim \left (1+\Y{\partial^{k+1} R}^2\right ) |z| \quad \text{ if } j \leq k.
    \end{aligned}
    \end{equation}
    We use the previous calculation to estimate the derivatives of $\log(A(x,x-z))$
    \begin{equation}\label{logA_reg_estimates}
        \begin{aligned}
            \left | \frac{d^j}{dx^j} \log (A(x,x-z)) \right | &= \left |\frac{d^{j-1}}{dx^{j-1}} \frac{\frac{d}{dx} A(x,x-z)}{A(x,x-z)} \right |= \left |\sum\limits_{
            \substack{ \alpha_0+...\alpha_j =2^{j-1}\\
            \alpha_i \geq 0}
            } c_{\alpha_1, ...,\alpha_t} \frac{\prod_{i=0}^j (\frac{d^i}{dx^i} A(x,x-z))^{\alpha_i}}{A(x,x-z)^{2^{j-1}}} \right |\\
            &  \lesssim   \sum\limits_{
            \substack{\alpha_0+...\alpha_j =2^{j-1}\\
            \alpha_i \geq 0}
            }  \prod_{i=0}^j \left ( 1+\Y{\partial^{i+1} R}\right )^{2\alpha_i} \ \lesssim \ 1+\Y{\partial^{j+1} R}^{2^j} \quad \text{ for } j \geq 1,
        \end{aligned}
    \end{equation}
    where we have used \eqref{Areg_estimates}, and the lower bound $A(x,x-z) \gtrsim z^2$. From this we can deduce the same bound for $C(x,x-z), S(x,x-z)$ derivatives since the factors of $\cos(z), \sin(z)$ multiplying in front don't interact with the derivative. 

    Lastly, we estimate the derivatives of $(\partial_x A / A) (x,x-z)$
    \begin{equation}\label{dxA/A_reg_estimates}
        \begin{aligned}
            \left |\frac{d^j}{dx^j} \frac{(\partial_x A)(x,x-z)}{A(x,x-z)} \right |&\leq \sum\limits_{
            \substack{ \ 0 \leq l \leq j \\
            \alpha_0+...\alpha_j =2^{j}-1\\
            \alpha_i \geq 0}
            } \left |c_{l,\alpha_1, ...,\alpha_t} \frac{\frac{d^l}{dx^l} (\partial_x A) (x,x-z)\prod_{i=0}^j (\frac{d^i}{dx^i} A(x,x-z))^{\alpha_i}}{A(x,x-z)^{2^{j}}} \right |\\
            &\lesssim \sum\limits_{
            \substack{ \ 0 \leq l \leq j \\
            \alpha_0+...\alpha_j =2^{j}-1\\
            \alpha_i \geq 0}
            } \frac{1}{|z|}\left  (1+\Y{\partial^{l+1} R}^2\right )\left ( 1 +\Y{\partial^{j+1}R}^{2^{j+1}-2}\right ) \ \lesssim \ \frac{1+\Y{\partial^{j+1}R}^{2^{j+1}}}{|z|},
        \end{aligned}
    \end{equation}
    where we have used the same estimates as in \eqref{logA_reg_estimates} as well as \eqref{Areg_estimates} to estimate derivatives of $\partial_x A(x,x-z)$.

    Finally we proceed to bound $\partial^{k+2} R$, we differentiate $k$ times the equation \eqref{Reg2ndStepEq} from  the second step. But now instead of directly differentiating respect to $x$ inside the integrals we first change variables $z = x-y$, using the notation $Q_1, Q_2, Q_3$ defined in \eqref{Reg2ndStepEq}, we have
    \begin{equation} \label{Reg3rdStepEq}
        \begin{aligned}
            \Tilde{K}_1(x) \partial^{k+2} R(x) = -\frac{d^{k-1}}{dx^{k-1}}\left (\partial^2 R(x)  \frac{d}{dx} \Tilde{K}_1(x) \right ) + \frac{d^k}{dx^k} Q_1(x) + \frac{d^k}{dx^k} Q_2(x) + \frac{d^k}{dx^k} Q_3(x) \quad \forall k \geq 1.
        \end{aligned}
    \end{equation}
    For the case $k = 0$ we have the same equation without the first term. The estimates for the other terms work exactly the same way.
    
    We are going to estimate the RHS term by term. Starting with the first one
    \begin{align*}
        &\left |\frac{d^{k-1}}{dx^{k-1}}\left (\partial^2 R(x)  \frac{d}{dx} \Tilde{K}_1(x) \right )\right | =\\
        &= \left |\sum_{j=1}^{k} \binom{k}{j}\partial^{k-j+2} R \left ( \partial^j R + \sum_{l=0}^j \binom{j}{l} \int \frac{d^l}{dx^l} C(x,x-z) \partial^{j-l}R(x-z) - \frac{d^l}{dx^l} S(x,x-z) \partial^{j-l+1}R(x-z) \right ) \right |\\
        &\lesssim \left (1+\Y{\partial^{k+1}R}\right ) \left ( 1+\Y{\partial^k R} + \left (1+\Y{\partial^{k+1} R} \right ) \left ( 1 + \sum_{l=1}^k \Y{\partial^{l+1} R}^{2^l}\right )\right ) \\
        &\lesssim 1 + \Y{\partial^{k+1}R}^{2^k+2},
    \end{align*}
    where we have used the estimates \eqref{logA_reg_estimates} for $l\geq 1$ for the derivatives of $C, S$ and for the case $l=0$ we used that $C, S$ are integrable because of Lemma \ref{CSbounds}. The second term $\frac{d^k}{dx^k} Q_1(x)$ is bounded analogously.

    We will now bound the fourth term splitting the region of integration $[-\pi, \pi]$ in $[-\pi, -\delta] \cup [\delta, \pi]$,  $(-\delta, \delta)$ and use two different bounds of the integrand in each region. For any $\delta \in (0,1)$ we get
    \begin{align*}
        \left |\frac{d^k}{dx^k} Q_3(x) \right |&= \left |\sum_{j=0}^k \binom{k}{j} \int_{-\pi}^\pi \cos(z) \frac{d^j}{dx^j} \frac{(\partial_x A)(x,x-z)}{A(x,x-z)} \frac{d^{k-j}}{dx^{k-j}} g(x,x-z)dz \right |\\
        &\lesssim \sum_{j=0}^k \int_{-\delta}^\delta \frac{1+\Y{\partial^{j+1} R}^{2^{j+1}}}{|z|} \left (1+\Y{\partial^{k+1}R}\right ) \left (1+\Y{\partial^{k+2}R}\right )|z| dz \\
        &+ \int_{|z|\geq \delta} \frac{1+\Y{\partial^{k+1} R}^{2^{j+1}}}{|z|} \left (1+\Y{\partial^{k+1}R}^2\right ) dz \\
        &\lesssim \left ( 1 + \Y{\partial^{k+1}R}^{2^{k+1}+1}  \right ) \left ( \delta \left (1+\Y{\partial^{k+2} R} \right )+ \frac{1+\Y{\partial^{k+1}R}}{\delta}\right ),
    \end{align*}
    where we have used \eqref{dxA/A_reg_estimates} for the derivatives of $\partial_x A/A$, \eqref{hgReg_estimates} for $g$ in the region $|z|\geq \delta$ and \eqref{hgReg_estimates} for $g$ in the region $|z|<\delta$. Recalling the definition of $Q_2(x)$ \eqref{hgReg_estimates}, to estimate the third term $\frac{d^k}{dx^k} Q_2(x)$ we follow the same procedure we just did for $\frac{d}{dx^k} Q_3(x)$ but we use the only estimate we have for the derivatives of $h(x,x-z)$ \eqref{hgReg_estimates}, which doesn't have any term like $|z|$ but we get this term from the $\sin(z)$ that we have (instead of the $\cos(z)$ we had in the case we have written). We end up getting
    \begin{align*}
        \left |\frac{d^k}{dx^k} Q_2(x)\right | \lesssim \left ( 1 + \Y{\partial^{k+1}R}^{2^{k+1}+2}  \right ).
    \end{align*}
    Notice that in contrast to the inequality for the derivatives of $Q_3(x)$ here the derivative we want to bound $\partial^{k+2}R$ does not appear on the bound. Putting these estimates together, and considering $\delta \ll 1$ to absorb the bounds of the third term into the one of the fourth term, we get that for all $\delta \in (0,1)$
    \begin{equation*}
        C_{\Tilde{K}_1} \Y{\partial^{k+2} R} \lesssim 1 + \Y{\partial^{k+1}R}^{2^k+2} + \left ( 1 + \Y{\partial^{k+1}R}^{2^{k+1}+1}  \right ) \left ( \delta \left(1+ \Y{\partial^{k+2} R} \right )+ \frac{1+\Y{\partial^{k+1}R}}{\delta}\right ),
    \end{equation*}
    where we have used the lower bound on $|\Tilde{K}_1(x)|$ given by Lemma \ref{K1tilde}. Since $C_{\Tilde{K}_1} > 0$, taking $\delta$ small enough we can put the term $\delta \Y{\partial^{k+2} R}$ in the LHS and the coefficient multiplying $\Y{\partial^{k+2} R}$ keeps being positive, we conclude that $\Y{\partial^{k+2}R}$ is bounded if $\Y{\partial^{k+1}R}$ is bounded, for all $k \geq 0$. By induction we have proven that $R(x) \in C^\infty$.
\end{proof}

\subsection{Analyticity}

Considering the stream function $\psi = \frac{1}{2\pi}\log(|\cdot|)\ast\mathbbm{1}_D$ associated to our vortex patch at time 0, $D = \{ x\in \mathbb{R}^2 : |x|^2 \leq R(\alpha(x))\}$, where $\alpha(x)$ is the angle associated to $x$ in polar coordinates,  \eqref{eqR} is equivalent to 
\begin{equation}\label{eq_psi}
    \nabla^{\perp} \psi(x)\cdot \overrightarrow{n}(x) = \Omega x^\perp \cdot \overrightarrow{n}(x) \quad \forall x \in \pa D, 
\end{equation}
where $\overrightarrow{n}(x)$ is a perpendicular vector to $\pa D$ at $x$. Moreover, if we consider the stream function in the rotating frame $\varphi(x) = \frac{1}{2\pi} \log(|\cdot|)\ast \mathbbm{1}_D - \frac{\Omega}{2} |x|^2$, it solves the following elliptic free boundary problem:
\begin{equation*}
    \left \{
    \begin{aligned}
        &\Delta \varphi = 1 - 2\Omega \quad &x \in D\\
        & \Delta \varphi = -2\Omega \quad &x \in \mathbb{R}^2\backslash D \\
        &\varphi = c \quad &x \in \partial D.
    \end{aligned}
    \right.
\end{equation*}

In this setting, we are ready to prove the analyticity of the boundary of $D$.
\begin{proposition}
\label{Ranalytic} 
    The velocity of the particles in the rotating frame $\nabla^\perp \varphi(x)$ is non zero for all $x$ on the boundary $\pa D$, and $\pa D$ is analytic.
\end{proposition}
\begin{proof}
    If we prove that $|\nabla^\perp \varphi (x)| > 0$ for all $x \in \pa D$, then the analyticity of $\partial D$ is an immediate consequence of \cite[Theorem 3.1']{Kinderlehrer-Nirenberg-Spruck:regularity-elliptic-free-boundary}, since we have already proven that the boundary of $D$ is $C^2$.

    We now proceed to bound $\nabla^\perp \varphi (x)$ from below. First we state a useful equality, rewriting $\eqref{eqR}$ we have that %\color{red} I believe there's no $\Omega$ here since the C's and the S's have an $\Omega$ dividing \color{black}
    \begin{equation}\label{useful_R_identity}
        -R(\alpha) \left (\int_0^{2\pi} (C[v]R'(y) + S[v]R(y)) dy \right ) = R'(\alpha) \left ( - R(\alpha) + \int_0^{2\pi} (S[v]R'(y)-C[v]R(y)) dy\right )
    \end{equation}
    for all $\alpha \in [0,2\pi]$. Considering $\overrightarrow{t}(x) := R(\alpha)(-\sin(\alpha), \cos(\alpha)) + R'(\alpha) (\cos(\alpha), \sin(\alpha))$, 
    a tangential vector to $\partial D$ at $x$, we have after a long but straightforward computation that
    \begin{equation}
        \begin{aligned}
            \nabla^\perp \varphi (x) \cdot \overrightarrow{t}(x) &= R(\alpha)\left ( - R(\alpha) + \int_0^{2\pi} S[v]R'(y) - C[v]R(y) dy\right ) - R'(\alpha)\left (\int_0^{2\pi} C[v]R'(y) + S[v]R'(y)dy\right)\\
            &= \left ( R(\alpha) + \frac{R'(\alpha)^2}{R(\alpha)}\right ) \left ( - R(\alpha) + \int_0^{2\pi} S[v] R'(y) - C[v]R(y) dy\right ) \\
            &= \left ( R(\alpha) + \frac{R'(\alpha)^2}{R(\alpha)}\right )\left (-\Tilde{K_1}(x)\right ) \neq 0, 
        \end{aligned}
    \end{equation}    
    for all $x\in \pa D$. %The first equality is obtained by using the same procedure that is used to derive the vortex patch equation from the Euler equations ('long' computations + divergence theorem). 
    The second equality is due to the identity \eqref{useful_R_identity}, the third one is by definition of $\Tilde{K_1}$, and the last line is non zero due to the strict positivity of the first factor and the lower bound of $|\Tilde{K_1}(x)|$ by Lemma  \ref{K1tilde}.
    %\color{red} I would rewrite slightly the last line of the chain of inequalities by saying due to the strict positivity of the first factor and the lower bound of $|\Tilde{K_1}(x)|$ by Lemma  \ref{K1tilde}.
    \color{black}
\end{proof}

\subsection{Non-convexity}

In this subsection we will propagate the quantitative properties of our approximate solution $R_0$ to our real solution $R$. We will enclose $R(x)$ between two explicit functions $R_{inf}, R_{sup}$, to get a hard control on $R$.

\begin{lemma}
    Let $R_{inf}(x) := R_0(x) - \epsilon \sqrt{p_m(x)}$ and $R_{sup}(x) := R_0(x) + \epsilon \sqrt{p_m(x)}$, where $p_m$ is given by Definition \ref{pm_def}.
    Then,  we have $R_{inf}(x) \leq R(x) \leq R_{sup}(x)$ for all $x$.
\end{lemma}

\begin{proof}
    Since $R = R_0 + v$ where $v \in \Tilde{Y}_m^{even,\epsilon}$, using Lemma \ref{bound_v(x)-v(y)} and that $v(0) = 0$:
    \begin{align*}
        |v(x)| = |v(x)-v(0)| \leq \Xm{v'} \sqrt{p_m(x)} \leq \epsilon \sqrt{p_m(x)}.
    \end{align*}
    We can then conclude 
    \begin{equation*}
        R_{inf}(x) \leq R_0(x) - |v(x)| \leq R(x) \leq R_0(x) + |v(x)| \leq R_{sup}(x).
    \end{equation*}
\end{proof}

We now proceed to prove the non-convexity of the patch.
\begin{proposition}\label{NonConvex}
    The vortex patch $D = \{ (x,y) \in \mathbb{R}^2 : x^2+y^2 \leq R(\alpha(x,y))\}$ is non-convex and $m$-fold symmetric (m = 6), where $\alpha(x,y)$ is the angle of the point $(x,y)$ in polar coordinates.
\end{proposition}
\begin{proof}
    The $m$-fold symmetry is straightforward since $R = R_0 + v$, with $v \in \Tilde{Y}_m^{even, \epsilon}$, and both $R_0$ and $v$ are $2\pi/m$-periodic.

    For the non-convexity we  will use that we know exactly the points on $\partial D$ at angles $\alpha_k = \frac{2\pi}{m} k$ for all $k \in \mathbb{Z}$. This is because by construction $v(x) = \int_0^x \Tilde{u}$ so $v(0)=0$ and $v$ is $2\pi/m$-periodic. Hence $R(\alpha_k) = R_0(0)$.
    Let $P_0, P_1 \in \partial D$, defined as
    \begin{align*}
        P_0 := R_0(0)(1, 0), \quad \quad P_1 := R_0(0)(\cos(\alpha_1), \sin(\alpha_1)).
    \end{align*}
    Then if we prove that the midpoint of the segment between $P_0, P_1$ doesn't belong to $D$ we have proven the non-convexity of $D$. We have to show that $\left | \frac{P_0+P_1}{2}\right | > R(\frac{\alpha_1}{2})$. Using that
    \begin{align*}
        R\left (\frac{\alpha_1}{2}\right ) \leq R_{sup}\left (\frac{\alpha_1}{2}\right ) = R_0\left (\frac{\pi}{m}\right ) + \epsilon \sqrt{\frac{\pi}{m}}, \quad
        \left | \frac{P_0+P_1}{2}\right | = R_0(0) \cos\left ( \frac{\pi}{m}\right ),
    \end{align*}
    the result follows from Lemma  \ref{Ineq_NonConvexity}.
    
    %is reduced to the inequality $R_0(\pi/m)+\epsilon \sqrt{\pi/m} < R_0(0) \cos(\pi/m)$ which is proved in Lemma \ref{Ineq_NonConvexity}. 
\end{proof}

\appendix
\section{Auxiliary Lemmas}

\begin{lemma}
\label{FixedPointConstantChecking}
Consider the constants $C_5, C_6$, $\varepsilon_0 =3\cdot 10^{-8}/0.1$ defined in Proposition \ref{Final_nonlinear_bounds_NL}. Then the following bounds are satisfied:  
\begin{equation}\label{numerical_bounds_constants}
    \varepsilon_0 = 3\cdot 10^{-7}, \quad 575 < C_5 < 580, \quad C_6 < 780. 
\end{equation}

Moreover, let us recall from \eqref{def_R0_m_N0_Omega_epsilon} that $\varepsilon = 2\cdot 10^{-5}$ and from Lemma \ref{Linvertible} that $C_1 = 35$. We have that
\begin{align*}\frac{1-\sqrt{1-4C_1C_5C_1\epsilon_0}}{2C_1C_5} \leq \epsilon \leq \frac{1+\sqrt{1-4C_1C_5C_1\epsilon_0}}{2C_1C_5}
    \end{align*}

    and

    \begin{align*}
        \varepsilon < \frac{1}{C_1 C_6}.
    \end{align*}
\end{lemma}
\begin{proof}
    The proof is computer-assisted and can be found in the supplementary material. We refer to Appendix \ref{appendix:implementation} for the implementation details.
\end{proof}

%Estimates of interesting integrals
\begin{lemma} \label{Integralbound1}
    Bounds for $I_1^\alpha(m)$
    \begin{itemize}
    \item Bound for $h_C(z) = \cos(z)$
    \begin{equation}\label{Integralbound1.1} \int_{-\pi}^\pi \frac{|\cos(z)|\sqrt{p_m(z)}}{|\sin\left (\frac{z}{2}\right )|}dz  \leq \sqrt{\frac{\pi}{m}}\left [ \frac{4\pi}{m \sin \left ( \frac{\pi}{2m}\right )} + \sqrt{2}\left (\pi \log \left ( \frac{m}{2}\right ) + 2\right )\right ] =: I_1^C(m) .
    \end{equation}
    
    \item Bound for $h_S(z) = \sin(z)$
    \begin{equation}\label{Integralbound1.2} \int_{-\pi}^\pi \frac{|\sin(z)|\sqrt{p_m(z)}}{|\sin\left (\frac{z}{2}\right )|}dz  \leq 8 \sqrt{\frac{\pi}{m}} =: I_1^S(m). 
    \end{equation}
    \end{itemize}
\end{lemma} 
\begin{proof}
    We start by proving \eqref{Integralbound1.1}. Recalling the definition \ref{pm_def} of $p_m(z) = \text{dist} \left(z, \left \{ \frac{2\pi}{m}k\right \}_{k \in \mathrm{Z}} \right )$, we have $p_m(z) \leq \frac{\pi}{m}$ for all $z$ in $\mathrm{R}$ and $p_m(z) = |z|$ for all $z$ in $[-\pi/m, \pi/m]$. The other inequality we are going to use is that, since $|\sin \left (\frac{z}{2}\right )|$ is concave in $[-\pi, \pi]$ we have that
    \begin{equation}
        \frac{1}{|\sin \left ( \frac{z}{2}\right )|} \leq \frac{a}{\sin \left ( \frac{a}{2}\right )} \frac{1}{|z|} \quad \forall a \in [0, \pi], \ \ \forall z \in [-a, a].
    \end{equation}
    We now proceed by splitting the integral and using different bounds in each interval
    \begin{align*}
        \int_{-\pi}^\pi \frac{|\cos (z)| \sqrt{p_m(z)}}{|\sin\left ( \frac{z}{2} \right )|} dz &\leq \int_{-\frac{\pi}{m}}^{\frac{\pi}{m}} \frac{\sqrt{|z|}}{|\sin\left ( \frac{z}{2}\right )|}dz 
        + 2 \int_{\frac{\pi}{m}}^{\frac{\pi}{2}} \frac{\sqrt{\frac{\pi}{m}}}{\sin\left ( \frac{z}{2}\right )}dz 
        + 2\int_{\frac{\pi}{2}}^\pi \frac{\sqrt{\frac{\pi}{m}}}{\sin\left ( \frac{\pi}{4}\right )} |\cos(z)| dz \\
        & \leq \frac{2\pi}{m|\sin\left(\frac{\pi}{2m}\right)|} \int_0^{\frac{\pi}{m}} \frac{1}{\sqrt{z}}dz + 
        \frac{2\sqrt{ \frac{\pi}{m}} \frac{\pi}{2}}{\sin \left ( \frac{\pi}{4} \right )} \int_{\frac{\pi}{m}}^\frac{\pi}{2} \frac{1}{z} 
        + 2\sqrt{\frac{2\pi}{m}}\int_{\frac{\pi}{2}}^\pi |\cos (z)|\\ 
        &= \frac{4\pi}{m \sin \left ( \frac{\pi}{2m}\right )} \sqrt{\frac{\pi}{m}} + \sqrt{\frac{2\pi}{m}} \pi \log \left ( \frac{m}{2} \right )  + 2\sqrt{\frac{2\pi}{m}} ,
    \end{align*}
    where we have used the previously mentioned inequalities and that $1/\sin(z/2)$ is decreasing in $[0, \pi]$.

    For the second inequality \eqref{Integralbound1.2} it is more straightforward
    \begin{align*}
        \int_{-\pi}^\pi \frac{|\sin(z)|\sqrt{p_m(z)}}{|\sin \left( \frac{z}{2}\right )|} dz =  \int_{-\pi}^\pi 2 |\cos \left(\frac{z}{2} \right )| \sqrt{p_m(z)} dz \leq 2\sqrt{\frac{\pi}{m}} 4.
    \end{align*}
\end{proof}

For the following Lemmas we will not write their proofs since they are very similar to the proof of last Lemma \ref{Integralbound1}.
%Estimates of interesting integrals
\begin{lemma} \label{Integralbound2}
    We have the following integral bounds $I_2^\alpha(m)$
    \begin{itemize}
    
    \item Bound for $h_C(z) = \cos(z) p_m(z)$
        \begin{equation*} 
            \int_{-\pi}^\pi \frac{|\cos(z)|p_m(z)^{3/2}}{\sin^2 \left (\frac{z}{2}\right )}dz  \leq \sqrt{\frac{\pi}{m}} \left [ \frac{4\pi^2}{m^2 \sin^2\left ( \frac{\pi}{2m}\right )} + \frac{\pi^2(m-2)}{m} + \frac{4\pi}{m}\right ]=: I_2^C(m). 
        \end{equation*}
    \item Bound for $h_S(z) = \sin(z)$
        \begin{equation*} 
            \int_{-\pi}^\pi \frac{|\sin(z)|\sqrt{p_m(z)}}{\sin^2 \left (\frac{z}{2}\right )}dz  \leq \sqrt{\frac{\pi}{m}} \left [ \frac{8\pi}{m \sin \left ( \frac{\pi}{2m}\right )} + 2\sqrt{2} \pi \log \left ( \frac{m}{2}\right ) + 8(\sqrt{2}-1)\right ] =: I_2^S(m). 
        \end{equation*}
    \end{itemize}
\end{lemma} 

\begin{lemma} \label{Integralbound3}
    We have the following integral bounds $I_3^\alpha(m)$
    \begin{itemize}
    \item Bound for $h_C(z) = \cos(z) p_m(z)$
        \begin{equation*} 
            \int_{-\pi}^\pi |\cos(z) p_m(z)| dz  \leq  \frac{4\pi}{m} =: I_3^C(m). 
        \end{equation*}
    \item Bound for $h_S(z) = \sin(z)$
        \begin{equation*} 
            \int_{-\pi}^\pi |\sin(z)|dz  = 4 =:  I_3^S(m). 
        \end{equation*}
    \end{itemize}
\end{lemma} 

\begin{lemma} \label{Integralbound4}
    We have the following integral bounds $I_4^\alpha(m)$
    \begin{itemize}
    \item Bound for $h_C(z) = \cos(z) p_m(z)$
        \begin{equation*} 
            \int_{-\pi}^\pi \frac{|\cos(z)|p_m(z)^{3/2}}{|\sin\left (\frac{z}{2}\right )|}dz  \leq \left ( \frac{\pi}{m}\right )^{3/2} \left [ \frac{4 \pi}{3 m \sin \left ( \frac{\pi}{2m} \right )} + \sqrt{2}\left ( \pi \log \left( \frac{m}{2}\right )+ 2\right )\right ]  =: I_4^C(m). 
        \end{equation*}
    \item Bound for $h_S(z) = \sin(z)$
        \begin{equation*} 
            \int_{-\pi}^\pi \frac{|\sin(z)|\sqrt{p_m(z)}}{|\sin\left (\frac{z}{2}\right )|}dz  \leq I_4^S(m) := I_1^S(m) = 8 \sqrt{\frac{\pi}{m}}.
        \end{equation*}
    \end{itemize}
\end{lemma}

\begin{lemma}\label{Integralbound5}
    We have the following integral bounds $I_5^\alpha(m)$
    \begin{itemize}
        \item Bound for $h^C(z) = \cos(z) p_m(z)$
        \begin{equation*}
            \int_{-\pi}^\pi \frac{\cos(z)p_m(z)^2}{\sin^2\left ( \frac{z}{2}\right )} dz \leq \frac{\pi}{m }\left [ 2 \left ( \frac{\pi/m}{\sin \left ( \frac{\pi}{2m}\right )}\right )^2 + \frac{\pi^2(m-2)}{m} + \frac{4\pi}{m}\right ]=: I_5^C(m).
        \end{equation*}
        \item Bound for $h^S(z) = \sin(z)$
        \begin{equation*}
            \int_{-\pi}^\pi \frac{\sin(z)p_m(z)}{\sin^2\left ( \frac{z}{2}\right )} dz \leq \frac{\pi}{m} \left [ \frac{4 \pi}{m \sin \left ( \frac{\pi}{2m}\right )} + 2\sqrt{2} \pi \log \left ( \frac{m}{2} \right ) + 8 (\sqrt{2}-1)\right ] =: I_5^S(m).
        \end{equation*}
    \end{itemize}
\end{lemma}

The following lemma gives bounds of integrals around their singular points.
\begin{lemma}\label{singularities_E0}
    Let $F_K(x,z)$ be
    \begin{align*}
        F_K(x, z) &:= \log(A(x,x-z)) \left [\cos (z) (R_0(x) R_0'(x-z) - R_0'(x)R_0(x-z))  \right .\\
        &\left .+ \sin(z) (R_0(x)R_0(x-z) + R_0'(x)R_0'(x-z))\right ].
    \end{align*}
    Notice that $F[0](x) = \int_{-\pi}^\pi F_K(x,z) dz$ where $F[0]$ is defined in \eqref{eqR}. 
    Then for all $0 \leq \delta < 1/2$, we have the three following bounds:
    \begin{align*}
         \int_{-\delta}^\delta \left | F_K(x,z) \right |dz  &\leq \delta^2  \cdot (M_{R_0}M_{R_0''} + 2M_{R_0'}^2 + M_{R_0}^2)  \left (1-2\log \left (\frac{2m_{R_0}\delta}{\pi}\right )  \right ), 
    \end{align*}
    as well as
    \begin{align*}
         \int_{-\delta}^\delta  \left | \frac{\partial}{\partial x}F_K(x,z) \right | dz  &\leq \delta^2 \cdot\left( M_{R_0} M_{R_0'''} + 2 M_{R_0} M_{R_0'} + 3 M_{R_0'} M_{R_0''} \right)  
         \left(1-2\log\left(\frac{2 m_{R_0} \delta}{\pi}\right) \right) \\
        &+ \delta^2 \cdot \left( M_{R_0} M_{R_0''} + 2 M_{R_0'}^2 + M_{R_0}^2 \right) \frac{1}{2m_{R_0}}\left ( M_{R_0''}\pi + 4 M_{R_0'}\right )
    \end{align*}
    and 
    \begin{align*}
        & \int_{-\delta}^\delta \left | \frac{\partial^2}{\partial^2 x}F_K(x,z) \right | dz  \leq 
        \delta^2 \cdot \left[ M_{R_0} M_{R_0''''} + 2 \left( M_{R_0} M_{R_0''} + M_{R_0'}^2 \right) + 3 M_{R_0''}^2 + 4 M_{R_0'} M_{R_0'''} \right] \left(1 -2\log\left( \frac{2 m_{R_0}\delta}{\pi} \right)\right) \\
        &+ \delta^2 \cdot \left( M_{R_0} M_{R_0'''} + 2 M_{R_0} M_{R_0'} + 3 M_{R_0'} M_{R_0''} \right) \frac{1}{m_{R_0}}\left ( M_{R_0''}\pi + 4 M_{R_0'}\right )  \\
        &+ \left( M_{R_0} M_{R_0''} + 2 M_{R_0'}^2 + M_{R_0}^2 \right) \frac{\delta^2}{m_{R_0}^2} \left[ \frac{\pi^2 \left( M_{R_0''}^2 + M_{R_0'} M_{R_0'''} \right)}{2} + 2 \left( M_{R_0''} M_{R_0} + M_{R_0'}^2 \right)  + \left ( \frac{M_{R_0''} \pi}{2} + 2M_{R_0'}\right )^2  \right] 
    \end{align*}
\end{lemma}
\begin{proof}
    We only do the proof of the first bound, since the rest are done using the same procedure with longer computations. We start by bounding $F_K(x,z)$ pointwise in the region $|z| \leq \delta$ with $\delta \leq \frac{1}{2}$. 
    In this region, $A \leq 1$, since
    \begin{align*}
        A[0](x,x-z) \leq (M_{R_0}-m_{R_0})^2 + 4M_{R_0}^2 \sin^2(\delta/2) \leq 0.2^2 + 8(1/4)^2 \leq 1.
    \end{align*}
    Then, we can use that $A[0](x,x-z) \geq 4m_{R_0}^2 \sin^2(z/2)$ and that $|\sin(z/2)| \geq |z|/\pi$ for $|z| \leq \pi$
    to get:\begin{equation}\label{logA_bound_appendix}
        |\log(A[0](x,x-z))| = -\log(A[0](x,x-z)) \leq -2\log\left ( \frac{2m_{R_0} |z|}{\pi}\right ).
    \end{equation}
    Going back to the full expression of $F_K(x,z)$, adding and subtracting $R_0(x)R_0'(x)\cos(z)\log(A(x,x-z))$ and using the mean value theorem for $R_0, R_0'$, we get the estimate
    \begin{align*}
        |F_K(x,z)| \leq -2\log\left (\frac{2m_{R_0}|z|}{\pi}\right ) \left ( (M_{R_0}M_{R_0''} + M_{R_0'}^2) |z| + (M_{R_0}^2 + M_{R_0'}^2)|z|) \right ). 
    \end{align*}
    Integrating from $-\delta$ to $\delta$ we get the desired estimate
    \begin{align*}
        \int_{-\delta}^\delta \left | F_K(x,z)\right | dz
        &\leq \left ( M_{R_0}^2 + 2M_{R_0'}^2 + M_{R_0}M_{R_0''}\right )2 \int_{0}^\delta -2\log\left (\frac{2m_{R_0}|z|}{\pi}\right ) |z| dz \\
        &= \left ( M_{R_0}^2 + 2M_{R_0'}^2 + M_{R_0}M_{R_0''}\right ) \delta ^2 \left ( 1 - 2\log \left ( \frac{2m_{R_0}\delta}{\pi}\right )\right ).
    \end{align*}
\end{proof}

The following lemma is analogous to the previous one but for functions related to the kernel $K(x,y)$. We are going to define $f, g$ such that $K_1(x) = R_0(x) + \int_0^{2\pi} f(x,z) dz$ and $K_2(x) = R_0'(x) + \int_0^{2\pi} g(x,z) dz$, where $K_1, K_2$ are defined in \eqref{defK1}, \eqref{defK2}, and bound their integrals around a small neighbourhood of the singular point $z=0$.
\begin{lemma}\label{singularities_K1}
    Let $f, g$ be defined as:
    \begin{align*}
        f(x,z) &:= R_0(x-z)C[0](x,x-z) - R_0'(x-z)S[0](x,x-z),\\
        g(x,z) &:= -R_0'(x-z)C[0](x,x-z) - R_0(x-z)S[0](x,x-z) + M_1(x,x-z),
    \end{align*}
    where $C, S$ are defined in \eqref{defCS} and we recall the definition of $M_1$ 
    \begin{equation*}
    \begin{aligned}
        M_1(x,y) &= \frac{1}{4\pi \Omega}  \left [ \frac{\cos(x-y) \left ( R_0(x)R_0'(y) - R_0(y)R_0'(x) \right ) + \sin(x-y) \left ( R_0(x)R_0(y) + R_0'(x)R_0'(y) \right )}{A[0](x,y)}  \right ] \\
        & \times \left [ \left ( 2(R_0(x)-R_0(y)) + 4R_0(y) \sin^2\left ( \frac{x-y}{2}\right ) \right ). \right. 
    \end{aligned}
    \end{equation*}
    
    Then,  for all $0 < \delta < 1/2$ 
    \begin{align*}
        4\pi\Omega\int_{-\delta}^\delta \left | f(x,z)\right | dz \leq \delta\cdot   4M_{R_0}  \left(1 -\log\left(\frac{2m_{R_0} \delta}{\pi}\right) \right) + \delta^2 \cdot M_{R_0'}  \left( 1-2\log\left(\frac{2m_{R_0} \delta}{\pi}\right)  \right),
    \end{align*}
    and 
    \begin{align*}
        4\pi\Omega\int_{-\delta}^\delta \left | \frac{\partial}{\partial x}f(x,z)\right | dz &\leq \delta \cdot 4M_{R_0'} \left( 1-\log\left( \frac{2m_{R_0}  \delta}{\pi} \right) \right) 
        + \delta^2 \cdot M_{R_0''} \left( 1-2\log\left( \frac{2m_{R_0} \delta}{\pi} \right) \right)  \\
        &+ \frac{\delta}{m_{R_0}}\left ( \pi M_{R_0''} + 4 M_{R_0'} \right ) \left (M_{R_0} + \delta M_{R_0'} \right )
        %&+\frac{M_{R_0'}}{m_{R_0}^2} \left( \frac{\pi^2 M_{R_0''}}{2} + 2M_{R_0} \right) \left( 2M_{R_0} \cdot \delta + M_{R_0'} \cdot \delta^2 \right).
    \end{align*}
    Also for $g$ we have
    \begin{align*}
        4\pi\Omega\int_{-\delta}^\delta \left | g(x,z)\right | dz &\leq \delta \cdot 4M_{R_0'}  \left(1 -\log\left(\frac{2m_{R_0} \delta}{\pi}\right) \right) + \delta^2 \cdot M_{R_0}  \left( 1-2\log\left(\frac{2m_{R_0} \delta}{\pi}\right)  \right) \\
        &+ \frac{M_{R_0}M_{R_0''} + 2M_{R_0'}^2 + M_{R_0}^2 }{m_{R_0}} 
        (\pi + \delta) \cdot \delta,
    \end{align*}
    and
    \begin{equation}\label{K2_singularity_bound}
    \begin{aligned}
        4\pi\Omega\int_{-\delta}^\delta \left | \frac{\partial}{\partial x}g(x,z)\right | dz &\leq
        \delta \cdot 4M_{R_0''} \left( 1-\log\left( \frac{2m_{R_0} \delta}{\pi} \right)\right)
        +  \delta^2 \cdot M_{R_0'} \left( 1-2\log\left( \frac{2m_{R_0} \delta}{\pi} \right)  \right)  \\
        &+ \frac{\delta}{m_{R_0}}\left ( \pi M_{R_0''} + 4 M_{R_0'} \right ) \left (M_{R_0'} + \delta M_{R_0} \right ) \\
        &+ \delta \cdot \left [ \left (M_{R_0}M_{R_0'''} + 3M_{R_0'}M_{R_0''} + 2 M_{R_0}M_{R_0'} \right ) \left (\frac{\pi + \delta}{m_{R_0}}\right )\right .\\
        &+ \left ( M_{R_0}M_{R_0''} + 2M_{R_0'}^2 + M_{R_0}^2 \right )\left ( M_{R_0''} \pi + 4 M_{R_0'}\right )\left ( \frac{\pi + \delta}{2 m_{R_0}^2}\right )\\
        &+ \left . \left ( M_{R_0}M_{R_0''} + 2M_{R_0'}^2 + M_{R_0}^2 \right )\left ( \frac{\pi^2 M_{R_0''} + M_{R_0'} \delta}{m_{R_0}^2}\right ) \right ]
        %&+ 2 \delta \cdot \left [  \left( \frac{\pi^2 M_{R_0'} M_{R_0''}}{2m_{R_0}^2} + \frac{2M_{R_0} M_{R_0'}}{m_{R_0}^2} \right)  \frac{\pi^2}{4m_{R_0}^2}  \left( M_{R_0}^2 + 2M_{R_0'}^2 + M_{R_0} M_{R_0''} \right) 2(M_{R_0} + M_{R_0'})\right .\\
        %&+  \frac{\pi^2}{4m_{R_0}^2} \cdot \left( 2M_{R_0} M_{R_0'} + 3M_{R_0'} M_{R_0''} + M_{R_0} M_{R_0'''} \right)  2(M_{R_0} + M_{R_0'}) \\
        %&+ \left . \frac{\pi^2}{4m_{R_0}^2} \cdot  \left( M_{R_0}^2 + 2M_{R_0'}^2 + M_{R_0} M_{R_0''} \right)   2(M_{R_0'} + M_{R_0''}) \right ].
    \end{aligned}
    \end{equation}
\end{lemma}
\begin{proof}
    We are going to prove the bound \eqref{K2_singularity_bound} because it is the most illustrative, the rest are either easier or just longer computations using the same ideas. 
    Starting from the first two terms of $g$, by definition of $C, S$ we have
    \begin{align*}
        4\pi \Omega \left | R_0'(x-z) C[0](x,x-z) + R_0(x-z) S[0](x,x-z)\right | \leq  \left (M_{R_0'}  + M_{R_0} |z| \right )\log \left (A[0](x,x-z) \right ),
    \end{align*}
    using the bound we used for $\log A$ in the proof of the previous Lemma \eqref{logA_bound_appendix} and integrating over the desired interval, we get
    \begin{equation}\label{first_term_K2_singularity_bound}
    \begin{aligned}
        &4\pi \Omega \int_{-\delta}^\delta \left | R_0'(x-z) C[0](x,x-z) + R_0(x-z) S[0](x,x-z)\right | dz \\
        &\leq 4M_{R_0'} \int_0^\delta -\log \left ( \frac{2m_{R_0}z}{\pi}\right ) dz + 4M_{R_0} \int_0^\delta -z\log \left ( \frac{2m_{R_0}z}{\pi}\right ) dz \\
        &= 4M_{R_0'} \delta \left(1 -\log\left(\frac{2m_{R_0} \delta}{\pi}\right) \right) + M_{R_0} \delta^2 \left(1 -2\log\left(\frac{2m_{R_0} \delta}{\pi}\right) \right).
    \end{aligned}
    \end{equation}
    Now we proceed to estimate $4\pi \Omega M_1$. By definition
    \begin{align*}
        \left |4\pi\Omega M_1(x,x-z) \right |&= \left | \cos(z)\left ( R_0(x)R_0'(x-z)-R_0'(x)R_0(x-z)\right ) + \sin(z) \left ( R_0(x)R_0(x-z) + R_0'(x)R_0'(x-z) \right ) \right | \\
        &\times \left |\frac{2(R_0(x)-R_0(x-z)) + 4R(x-z)\sin^2(z/2)}{A[0](x,x-z)}\right |,
    \end{align*}
    adding and subtracting $R_0(x)R_0'(x)$ to the expression multiplying the cosine, applying the mean value theorem and using the bounds on $R_0$ and its derivatives, we get
    \begin{align*}
        \left |4\pi\Omega M_1(x,x-z) \right |&\leq \left (  M_{R_0} M_{R_0''} + M_{R_0'}^2  +  M_{R_0}^2 + M_{R_0'}^2\right ) |z| \left |\frac{2(R_0(x)-R_0(x-z)) + 4R(x-z)\sin^2(z/2)}{A[0](x,x-z)}\right | \\
        &\leq \left (  M_{R_0} M_{R_0''} + 2 M_{R_0'}^2  +  M_{R_0}^2 \right ) |z|\left (\frac{1}{2m_{R_0}|\sin(z/2)|} + \frac{1}{m_{R_0}} \right ) \\
        &\leq \frac{M_{R_0} M_{R_0''} + 2 M_{R_0'}^2  +  M_{R_0}^2}{m_{R_0}} \left ( \frac{\pi }{2}  + |z|\right ),
    \end{align*}
    where in the second inequality we have used the same ideas  as in \eqref{a2/A}, \eqref{a1/A} to bound $A$ from below appropriately for each term of the numerator. Integrating in $z$ from $-\delta$ to $\delta$ we obtain 
    \begin{equation*}
        4\pi\Omega\int_{-\delta}^\delta \left | M_1(x,x-z) \right | \leq \frac{M_{R_0} M_{R_0''} + 2 M_{R_0'}^2  +  M_{R_0}^2}{m_{R_0}} \left ( \pi \delta+ \delta^2 \right ).
    \end{equation*}
    Adding this bound with the previous one \eqref{first_term_K2_singularity_bound} we get the desired estimate.
\end{proof}
The following lemma states that the exact same bounds hold changing $f, g$ by $K_4, K_3$ respectively.
\begin{lemma} \label{singularities_K34}
    Let $K_3, K_4$ be defined as
    \begin{align*}
        K_3(x,y) &:= -R_0(x)S[0](x,y) + R_0'(x)C[0](x,y) + M_2(x,y), \\
        K_4(x,y) &:= -R_0'(x) S[0](x,y) - R_0(x) C[0](x,y),
    \end{align*}
    then, for all $0 < \delta < 1/2$, we have
\begin{align*}
        4\pi\Omega\max \left \{ \int_{x-\delta}^{x+\delta} \left |K_4(x,y)\right | dy, \int_{y-\delta}^{y+\delta} \left |K_4(x,y)\right | dx\right \} &\leq 4 \delta  M_{R_0}  \left(1 -\log\left(\frac{2m_{R_0} \delta}{\pi}\right) \right) + \delta^2  M_{R_0'}  \left( 1-2\log\left(\frac{2m_{R_0} \delta}{\pi}\right)  \right).
    \end{align*}
    Also, for $K_3$ we have
    \begin{align*}
        4\pi\Omega\int_{-\delta}^\delta \left | K_3(x,x-z)\right | dz &\leq \delta \cdot 4M_{R_0'}  \left(1 -\log\left(\frac{2m_{R_0} \delta}{\pi}\right) \right) + \delta^2 \cdot M_{R_0}  \left( 1-2\log\left(\frac{2m_{R_0} \delta}{\pi}\right)  \right) \\
        &+ \frac{M_{R_0}M_{R_0''} + 2M_{R_0'}^2 +  M_{R_0}^2 }{m_{R_0}} 
        (\pi + \delta) \cdot \delta,
    \end{align*}
    and 
    \begin{align*}
        4\pi\Omega\int_{-\delta}^\delta \left | \frac{\partial}{\partial x}g(x,z)\right | dz &\leq
        \delta \cdot 4M_{R_0''} \left( 1-\log\left( \frac{2m_{R_0} \delta}{\pi} \right)\right)
        +  \delta^2 \cdot M_{R_0'} \left( 1-2\log\left( \frac{2m_{R_0} \delta}{\pi} \right)  \right)  \\
        &+ \frac{\delta}{m_{R_0}}\left ( \pi M_{R_0''} + 4 M_{R_0'} \right ) \left (M_{R_0'} + \delta M_{R_0} \right ) \\
        &+ \delta \cdot \left [ \left (M_{R_0}M_{R_0'''} + 3M_{R_0'}M_{R_0''} + 2 M_{R_0}M_{R_0'} \right ) \left (\frac{\pi + \delta}{m_{R_0}}\right )\right .\\
        &+ \left ( M_{R_0}M_{R_0''} + 2M_{R_0'}^2 + M_{R_0}^2 \right )\left ( M_{R_0''} \pi + 4 M_{R_0'}\right )\left ( \frac{\pi + \delta}{2 m_{R_0}^2}\right )\\
        &+ \left . \left ( M_{R_0}M_{R_0''} + 2M_{R_0'}^2 + M_{R_0}^2 \right )\left ( \frac{\pi^2 M_{R_0''} + M_{R_0'} \delta}{m_{R_0}^2}\right ) \right ]
        %&+ 2 \delta \cdot \left [  \left( \frac{\pi^2 M_{R_0'} M_{R_0''}}{2m_{R_0}^2} + \frac{2M_{R_0} M_{R_0'}}{m_{R_0}^2} \right)  \frac{\pi^2}{4m_{R_0}^2}  \left( M_{R_0}^2 + 2M_{R_0'}^2 + M_{R_0} M_{R_0''} \right) 2(M_{R_0} + M_{R_0'})\right .\\
        %&+  \frac{\pi^2}{4m_{R_0}^2} \cdot \left( 2M_{R_0} M_{R_0'} + 3M_{R_0'} M_{R_0''} + M_{R_0} M_{R_0'''} \right)  2(M_{R_0} + M_{R_0'}) \\
        %&+ \left . \frac{\pi^2}{4m_{R_0}^2} \cdot  \left( M_{R_0}^2 + 2M_{R_0'}^2 + M_{R_0} M_{R_0''} \right)   2(M_{R_0'} + M_{R_0''}) \right ].
    \end{align*}
\end{lemma}
\begin{proof}
    For the first bound we realize that the only difference between $f(x,z)$ from the previous Lemma and $K_4(x,x-z)$ is that the $R_0, R_0'$ multiplying $S, C$ are evaluated at $x$ instead of $x-z$. Since in the previous lemma we have used the uniform bounds $M_{R_0}, M_{R_0'}$ for those $R_0, R_0'$, the same procedure applies to proving this bound. Note also that changing variables from $K(x,y)$ to $K_4(y-z, y)$ and integrating in $z$ we would obtain the same bound for the integral in $x$ as stated in this Lemma.

    For the case of $K_4$ the same happens with $g$, since the two differences between $g(x,z)$ and $K_3(x,x-z)$ are on one hand the evaluation of $R_0, R_0'$ multiplying $C,S$ (as before), which doesn't matter in terms of the bound as explained, even when we take the derivative respect to $x$, because afterwards we use again the uniform bounds on $R_0', R_0''$ independently of the evaluation point; and on the other hand that we have $M_2$ instead of $M_1$. Recalling their definitions \eqref{T1E1_def}, their only difference is a sign in a term and having $R_0(x)$ instead of $R_0(x-z)$ in another term. In particular none of these things will be accounted for differently when doing the bounds.
\end{proof}

\section{Implementation details of the computer-assisted part}
\label{appendix:implementation}

In this appendix we discuss the technical details about the
implementation of the different rigorous
numerical computations that appear in the proofs throughout the paper.
We performed the rigorous computations using the Arb library
\cite{Johansson:Arb} and specifically its C/C++ implementation. We attach the code as supplementary material. See Table \ref{table:lemmas_comp_run} for the specific programs/commands to run each Lemma/Proposition and their runtime, with more details in the Supplementary Material.  We have also sacrificed efficiency by readability and some parts of the code could be optimized. Instead, we decided to write a much more modular design with many small functions performing simple tasks, at the price of sometimes duplicating code. To perform this computations we have four main files:
\begin{itemize}
    \item \texttt{methods.cc}: This file's main objective is to perform integrals of given functions. It has two methods, \texttt{integrate} and \texttt{lin\_integrate}. The first one performs the integral with an adaptative size of the intervals where the quadrature is performed, dividing each interval until the given error tolerance is satisfied. The second method subdivides the interval in subintervals of a predetermined size to perform the quadrature. In both methods it is possible to specify the desired quadrature between the following options: 
    \begin{enumerate}
        \item $\mathscr{C}^0$ quadrature:
        \begin{equation*}
            \int_a^b f(x) dx \subseteq f([a,b])\cdot (b-a)
        \end{equation*}

        \item $\mathscr{C}^1$ Taylor quadrature:
        \begin{equation*}
            \int_a^b f(x) dx \subseteq \left (f\left (\frac{a+b}{2} \right ) + f'([a,b]) \cdot \left [ \frac{a-b}{2}, \frac{b-a}{2}\right ]\right )\cdot (b-a)
        \end{equation*}

        \item Midpoint $\mathscr{C}^2$ quadrature:
        \begin{equation*}
            \int_a^b f(x) dx \subseteq f\left (\frac{a+b}{2} \right ) \cdot (b-a) + \frac{(b-a)^3}{24} f''([a,b])
        \end{equation*}

        \item $\mathscr{C}^2$ Taylor quadrature for the function squared:
        \color{black}
        \begin{equation*}
            \int_a^b f^2(x) dx \subseteq \left (f\left (\frac{a+b}{2} \right ) + f' \left ( \frac{a+b}{2} \right ) \cdot\left [ \frac{a-b}{2}, \frac{b-a}{2}\right ] + \frac{1}{2} f''([a,b]) \cdot \left [0, \frac{(b-a)^2}{4} \right ]\right )^2 \cdot(b-a)
        \end{equation*}
    \end{enumerate}
    where the enclosures are taken in the interval arithmetic sense.

    We should remark here that throughout the calculation of all the singular integrals we have tuned the relative tolerance error allowed in the integration, such that the error from the numerical integration and the error from the singularity are of the same order of magnitude in order to balance out the difference sources of error.

    This file also contains a \texttt{bound\_function} method, which given a function, an interval, a step size, a constant and the type of bound (from above, below, absolute value...) it checks that it holds in the whole interval by subdiving it in subintervals of the specified step size and checking the condition there.
    
    \item \texttt{basic\_functions.cc}: This file contains the necessary functions to bound the error $E[0]$ of the approximate solution $R_0$. It contains simple functions like $R_0(x), A(x,x-z)$, their derivatives and other functions such as $F[R_0]$, defined in \eqref{F_def}, that are computed through performing integrals of the simpler functions.

    \item \texttt{complex\_functions.cc}: This file contains the necessary functions to bound the error of the approximate operator $L_F$ and prove the bound of Lemma \ref{L-L_Fbound}. It contains some functions that are algebraic expressions of $x, y, R_0, A$ and of trigonometric functions, like $M_1, M_2$ (defined in \ref{T1E1_def}); and  some functions that also contain integrals of previous functions in their expressions like $K_1, K_2, K_3^{s,o}$, $K_4^{s,o}$, that are necessary to compute $K$. These functions build upon previous ones from this file and some from the \texttt{'basic\_functions'}. It also contains the approximate kernel $K_F$ defined in \eqref{LF_def}. Lastly, a function to efficiently perform the computation in a 2D grid of the error $K-K_F$ is implemented.
   
    \item \texttt{matrix\_utils.cc}: This file contains a fast method to enclose the eigenvalues of given symmetric matrices. It calls the C-library '\texttt{LAPACK}' in order to find an approximation of the eigenvectors and then uses an implementation of \cite[Lemma 2.4]{GomezSerrano-Orriols:negative-hearing-shape-triangle} to find rigorous enclosures for the eigenvalues.
\end{itemize}

The rest of the files are simpler and build upon these ones. The files to prove Lemma \ref{CE0} and Lemma \ref{L-L_Fbound}. Each lemma is split into two files. The first ones ('\texttt{boundE0\_parallel}' and \texttt{'boundK\_parallel'}) are designed to be run with different inputs so each execution does a fraction of the total work, and the second ones ('\texttt{boundE0\_final}' and \texttt{'boundK\_final'}) do the post-processing of the generated outputs to add them together appropriately. The correct way to run the files is explained in the \texttt{'README'} file from the Supplementary Material. The rest of the Lemmas are much quicker and no more than one file is needed.

%\color{red}
%\begin{detailssth}{Lemmas \ref{lemma:aux_WioverZi_7o5},  \ref{lemma:tenthousand_7o5}} The implementation is straightforward; however, due to the large numbers that appear throughout the process ($Z_{10000} \sim 10^{46770}$), ultra-high precision is required to avoid overestimation and to be able to extract the signs out of the relevant quantities. We used 2000 bits to accomplish this.
%\end{detailssth}

\color{black}

\begin{detailssth}{Lemma \ref{boundR0}, \ref{MdR0}, \ref{MddR0}, \ref{MdddR_MddddR}} By periodicity and symmetry (or antisymmetry) we only have to check these conditions in the interval $[0, \pi/m]$. We split it into intervals of size $10^{-3}$ and check the bounds there. We don't use any derivative information to evaluate the functions in the interval: we propagate the intervals directly through the Fourier series, i.e. $R_0([x_n, x_{n+1}]) = \sum_{k=0}^{30} c_k \cos(2mk [x_n, x_{n+1}])$, and the same for further derivatives with their respective Fourier series.
\end{detailssth}

\begin{detailssth}{Lemma \ref{Ineq_NonConvexity}} Using the Fourier series of $R_0$ we compute $R_0(0), R_0(\pi/m)$ and check the desired inequality \eqref{Ineq_NonConvexty_eq}.
\end{detailssth}

\begin{detailssth}{Lemma \ref{CE0}} In order to prove that $\int_0^{\pi/m} E[0]^2(x) dx \leq C_{E_0}^2$ we use the pointwise $\mathscr{C}^2$ enclosure for $E[0](x)$. That is, if $X_k$ is the interval $X_k = [x_k, x_{k+1}]$, $x_k^{\text{mid}} = (x_k+x_{k+1})/2$ and $\delta x = x_{k+1}-{x_k}$, we have
\begin{equation} \label{E0_enclosure}
    E[0](X_k) \subseteq E[0](x_k^{\text{mid}}) + E[0]'(x_k^{\text{mid}})\cdot \left [-\frac{\delta x}{2}, \frac{\delta x}{2} \right ] + \frac{1}{2}E[0]''(X_k)\cdot \left [0 , \left (\frac{\delta x}{2}\right )^2 \right ],
\end{equation}
in the interval arithmetic sense. Using Lemma \ref{ddE0_bound}, we have a uniform enclosure on $E[0]'' \subseteq [-50, 50]$ in all the integration domain $[0,\pi/m]$, therefore, we split the integration domain in intervals of length $\delta x = 1.1 \cdot 10^{-6}$,  i.e $\int_0^{\pi/m} E[0]^2(x) dx \subseteq \sum_{k=0}^n E[0](X_k)^2 \cdot \delta x$ and in each interval we use the enclosure \eqref{E0_enclosure}. To compute this we use \texttt{'lin\_integrate'} from the \texttt{'methods'} file with the quadrature option of the $\mathscr{C}^2$ Taylor enclosure for the function squared. 

The importance of doing the Taylor expansion and getting the uniform enclosure on $E[0]''$ is that we 'only' have to compute $E[0], E[0]'$ at points (or very small intervals) and not at (bigger) intervals, which is not possible with enough precision due to the complexity of this function. To implement $E[0]$ we define the function $F_K(x,z)$ (defined in Lemma \ref{singularities_E0}) which is the integrand of $F[R_0]$ defined in $\eqref{F_def}$, but changing variables $y = x-z$ to have the singularity at $z = 0$ instead of at $y =x$. Due to the periodicity of $R_0$ we can change the domain of integration to still have $F[R_0] = \int_0^{2\pi} F_K(x,z) dz$. Then 
\begin{equation}\label{E0_integral_splitting}
E[0](x) \subseteq -R_0(x)R_0'(x) + \int_{\delta}^{2\pi- \delta} F_K(x,z) dz + \int_{-\delta}^\delta \left | F_K(x,z)\right | dz \cdot [-1, 1],
\end{equation}
where again the inclusion is in the interval arithmetic sense. To compute the first integral we use the $\mathscr{C}^2$ midpoint quadrature (we need to compute $\frac{\partial^2 }{\partial z^2} F_K(x,z)$) with the adaptive step integrator \texttt{'integrate'} from the \texttt{'methods'} file with a relative tolerance of $3\cdot 10^{-8}$ and $\delta = 10^{-6}$. For the second integral we use Lemma \ref{singularities_E0} and therefore have a uniform bound for all $x$.

To compute $E[0]'$ we do the same with $\frac{\partial}{\partial x} F_K(x,z)$ instead of $F_K(x,z)$. The difference is that to compute the first integral we use a $\mathscr{C}^1$ Taylor quadrature so we only need to compute $\frac{\partial^2}{\partial z \partial x} F_K(x,z)$ and do not need to take two derivatives with respect to $z$. The parameters we use are a relative tolerance of $9\cdot 10^{-3}$ and $\delta = 6\cdot 10^{-4}$. The second integral is also bounded by the second claim of Lemma \ref{singularities_E0}.
\end{detailssth}

\begin{detailssth}{Lemma \ref{ddE0_bound}}
To bound the function in the whole interval we will divide it in subintervals of size $10^{-4}$ and check that the bound is satisfied in each of them. We now explain how we compute $\frac{d^2}{dx^2}E[0]$ in each subinterval. By definition  $\frac{d^2}{dx^2} E[0] \subseteq -3R_0'(x)R''(x) - R_0(x)R_0'''(x) + \int_{0}^{2\pi} \frac{\partial^2}{\partial x^2} F_K(x,z) dz $, where we are using the definition of $F_K$ given in Lemma \ref{singularities_E0}. We follow the exact same procedure explained in the details of Lemma \ref{CE0}, splitting the integral as in \eqref{E0_integral_splitting}. The difference is that we use a $\mathscr{C}^0$ enclosure for the first integral in order not to take further derivatives. We use a relative tolerance of $30$ and $\delta = 4 \cdot 10^{-2}$. For the second integral we use the third claim from Lemma \ref{singularities_E0}. 

%To bound the function in the whole interval, we use the procedure we just explained to evaluate the function in subintervals of size $2\cdot 10^{-4}$ and check that the bound is satisfied in each of them. 
\end{detailssth}

\begin{detailssth}{Lemma \ref{K1bound}} We use the implementation of $K_1$ from the file '\texttt{complex\_functions}', explained in detail in \eqref{K1_enclosure} in the details of Lemma \ref{L-L_Fbound}. We use the method '\texttt{bound\_function}', with type $2$, to compute $K_1$ in intervals $[x_n, x_{n+1}]$ of size and in each of them check $K_1([x_n, x_{n+1}]) < -C_{K_1}$
\end{detailssth}

\begin{detailssth}{Lemma \ref{I+Ainvert}} 
To prove that $B := I+A$ is invertible and $\Ltwo{B^{-1}} < C_2$, we use the following observation:
\begin{align*}
    \text{If } \lambda_{\min} (B^t B) > 0  \implies B \text{ is invertible and} \Ltwo{B^{-1}} = \sqrt{\lambda_{\max}((B B^t)^{-1})} = \frac{1}{\sqrt{\lambda_{\min}(B B^t)}}.
\end{align*}
Using this, we use an implementation of  \cite[Lemma 2.4]{GomezSerrano-Orriols:negative-hearing-shape-triangle}, as explained before in the description of the file '\texttt{matrix\_utils}', to check that $\frac{1}{\sqrt{\lambda_{\min}(B B^t)}} \leq C_2$. 
\end{detailssth}

\begin{detailssth}{Lemma \ref{L-L_Fbound}} In order to go from a bound of the $L^2$ operator norm of $L-L_F$ to a bound of integrals of the function $K-K_F$, we use the Generalized Young's inequality
    \begin{equation*}
        \X{(L-L_F)u} = \X{\int_0^\frac{\pi}{m} (K-K_F)(x,y)u(y)} \leq \max \left \{ \|K-K_F\|_{L_x^1 L_y^\infty}, \|K-K_F\|_{L_y^1 L_x^\infty} \right \} \X{u}.
    \end{equation*}
    Therefore the bound $\| L-L_F\|_2 \leq C_3$ from Lemma \ref{L-L_Fbound}  is reduced to proving $\max \left \{ \|K-K_F\|_{L_x^1 L_y^\infty}, \|K-K_F\|_{L_y^1 L_x^\infty} \right \} \leq C_3$.

    To do this we are going to divide our domain $[0, \pi/m] \times [0, \pi/m]$ into $N^{\text{grid}} \times N^{\text{grid}}$ small squares of the same size with $N^{\text{grid}} = 5888$. Then we are going to compute $K(X_i,Y_j)-K_F(X_i,Y_j)$ for all of the squares $X_i\times Y_j = [x_i, x_{i+1}] \times [y_j, y_{j+1}]$ mentioned above.

    We first focus on $K_F$. Let us recall its definition in \eqref{LF_def}: $ K_F(x,y) = \sum_{k,l=1}^N A_{k,l} e_k(x)e_l(y)$ where $e_k(x)$ is proportional to $\sin(2mkx)$ and $e_{2k+1}(x)$ is proportional to $\cos(2mkx)$, so it is a Fourier series. To perform the computation, we use the first order Taylor expansion
    \begin{align*}
        K_F(X_i,Y_j) \subseteq K_F(x_i^{\text{mid}}, y_j^{\text{mid}}) + \frac{\partial }{\partial x} K_F (X_i, y_j^{\text{mid}}) \cdot \left [ -\frac{\delta x}{2}, \frac{\delta x}{2} \right ]+ \frac{\partial}{\partial y} K_F (X_i, Y_j) \cdot \left [ -\frac{\delta y}{2}, \frac{\delta y}{2}\right ], 
    \end{align*}
    where $x_i^{\text{mid}} = (x_i + x_{i+1})/2$, $\delta x = x_{i+1}-x_i$, and analogously for $y_i^{\text{mid}}$. To compute the value $K_F(x_i^{\text{mid}}, y_j^{\text{mid}})$ we use the Fourier series. To compute the partial derivatives of $K_F$ we also differentiate term by term the Fourier series and evaluate it in intervals term by term. For example
    \begin{align*}
        \frac{\partial}{\partial x} K_F(X_i, y_j^{\text{mid}}) = \sum_{t=1}^{(N-1)/2} (2mt) \sum_{l = 1}^N A_{2t, l}  e_{2t+1}(X_i) e_l(y_j^{\text{mid}}) - A_{2t+1, l}  e_{2t}(X_i) e_l(y_j^{\text{mid}}),
    \end{align*}
    where we have used $(e_{2t}(x))' = (C \sin(2mtx))' = 2mt C\cos(2mtx) = 2mt e_{2t+1}(x)$, and analogously $e_{2t+1}(x)' = -2mt e_{2t}(x)$.

    Now we explain how we compute $K$ in these squares. We first recall its definition \eqref{Lu_def}
    \begin{align*}
        K(x,y) = \frac{1}{K_1(x)} \left ( K_2(x) \mathbb{I}_{0 \leq y \leq x} + K_3^{s,o}(x,y) + K_4^{s,o}(x,y)\right ).
    \end{align*}
    We start by computing $K_1, K_2$ in a very similar way to which we have computed $E[0]$. We begin using a first order Taylor expansion 
    \begin{align*}
        K_a(X_i) \subseteq K_a(x_i^{\text{mid}}) + \frac{\partial}{\partial x} K_a(X_i) \cdot \left [ -\delta x, \delta x\right ],
    \end{align*}
    for $a = 1,2$, now need to compute $K_a(x)$ and $\frac{\partial}{\partial x} K_a$. Using their definitions \eqref{defK1}, \eqref{defK2} we enclose

    \begin{equation}\label{K1_enclosure}
        K_1(x) \subseteq R_0(x) + \int_\delta^{2\pi-\delta} |IK_1(x,z)| dz + \int_{-\delta}^\delta |IK_1(x,z)| dz \cdot [-1,1], 
    \end{equation}
    where $IK_1(x,z) = R_0(x-z)C[0](x,x-z)-R_0'(x-z)S[0](x,x-z)$. To compute the first integral we use the $\mathscr{C}^2$ midpoint quadrature (we need to compute $\frac{\partial^2}{\partial z^2} IK_1(x,z)$) with the adaptative step integrator '\texttt{integrate}' from the '\texttt{methods}' file with a relative tolerance of $3\cdot 10^{-5}$ and $\delta = 10^{-6}$. For the second integral we use Lemma \ref{singularities_K1} and therefore we have a uniform bound for all $x$.
    To compute $K_1'(x)$ we do the same as in \eqref{K1_enclosure} with $\frac{\partial}{\partial x} IK_1(x,z)$ instead of $IK_1(x,z)$, but the difference is that we use a $\mathscr{C}^1$ Taylor quadrature so we only need to compute the derivative $\frac{\partial^2}{\partial z \partial x} IK_1(x,z)$. The parameters we use in this case are a relative tolerance of $2$ and $\delta = 10^{-1}$. The singularity bound is given by the second inequality of Lemma \ref{singularities_K1}.
    For $K_2$ we do exactly the same with its respective expression. We also have bounds given for the singularity given by Lemma \ref{singularities_K1}. We use the parameters relative tolerance $2.5\cdot 10^{-4}$, and $\delta = 3\cdot 10^{-5}$. In a similar way, we perform the same splittings and estimates as for $K_1'$ to compute the derivative $K_2'$, using the $\mathscr{C}^1$ Taylor enclosure for the integration method and the last inequality of Lemma \ref{singularities_K1} to bound the singularity, using as integration parameters a relative tolerance of $50$ and $\delta = 2 \cdot 10^{-2}$. This is the way we can compute $K_1, K_2$ in intervals.

    We proceed to explain how to treat $K_4^{s,o}$. An advantage about this term is that we don't need to compute any integral but we shall also note that it is in $L^1$ but not in $L^\infty$ due to the singularity on the diagonal. We define
    \begin{align*}
        A_{\delta}(x,y) &= \max \left \{ A(x,y), 4m_{R_0}^2 \sin^2 \left ( \frac{\delta}{2}\right )\right \}, 
    \end{align*}
    and $K_{4,\delta}^{s,o}$ having the same definition as $K_4^{s,o}$ \eqref{defK4} but changing $A$ by $A_\delta$ in the definition of $S, C$  (cf. \eqref{defCS}) used in $K_4^{s,o}$. Notice that if $A(x,y) < A_\delta$ then $|x-y - 2\pi k|< \delta$ for some $k \in \mathbb{Z}$. Therefore, defining $\Bar{K}_\delta$ the same as $K$ but changing $K_4^{s,o}$ by $K_{4,\delta}^{s,o}$ we have
    \begin{equation}\label{K4_delta_error}
        \| K-K_F\|_{L_x^1 L_y^\infty} \leq \| \Bar{K}_\delta-K_F \|_{L_x^1 L_y^\infty} + \frac{1}{C_{K_1}}\|K_{4,\delta}^{s,o}-K_4^{s,o} \|_{L_x^1 L_y^\infty},
    \end{equation}
    recalling that $C_{K_1}$ is a lower bound of $K_1$. The first norm we are going to bound it using the numerical integration and the second we can bound it using the first inequality of Lemma \ref{singularities_K34}, because notice the lemma gives us a bound of the singularity of $K_4$ not the symmetrized version, but if we unfold the integral of the symmetrized version in the singular set we just get the integral of $K_4$ between $-\delta, \delta$. We use $\delta_{K_4} = 10^{-5}$. To compute $K_{4,\delta}^{s,o}$ in the squares $X_i \times Y_j$ we use its expression, without computing any further derivatives.

    It is only left how to compute $K_3^{s,o}(x,y) = \int_y^{\pi/m} K_3^{s,e}(x,y') dy'$. Actually, we are interested in computing this integral between to arbitrary limits $y_\text{inf}, y_{\text{sup}}$, we will explain later why. Defining 
    \begin{align*}
        K_3(x,y) := - R_0(x)S[0](x,y) + R_0'(x) C[0](x,y) + M_2(x,y),
    \end{align*} and recalling the definition of $K_3^{s,e}$ \eqref{defK3}, we have
    \begin{equation}\label{K3so_integrate_interval}
    \begin{aligned}
        \int_{y_\text{inf}}^{y_{\text{sup}}} K_3^{s,e}(x,y')dy' &=
        \int_{y_\text{inf}}^{y_{\text{sup}}} \sum_{j=-1}^{m-2} K_3(x,y' + 2\pi j/m) + K_3(x,-y' - 2\pi j/m) dy' \\
        &= \sum_{j=-1}^{m-2} \int_{x-y_\text{sup} - \frac{2\pi j}{m}}^{x - y_{\text{inf}} - \frac{2\pi j}{m}} K_3(x, x-z) dz + \int_{x+y_\text{inf} + \frac{2\pi j}{m}}^{x + y_{\text{sup}} + \frac{2\pi j}{m}} K_3(x,x-z) dz.
    \end{aligned}
    \end{equation}
    To compute $K_3^{s,o}$ in a square $X_i\times Y_j$ we are going to use again a first order Taylor expansion
    \begin{equation*}
    \begin{aligned}
        K_3^{s,o}(X_i, Y_j) \subseteq K_3^{s,o}(x_i^{\text{mid}}, y_i^{\text{mid}}) + \frac{\partial}{\partial x} K_3(X_i, y_j^{\text{mid}}) \cdot \left [-\frac{\delta x}{2}, \frac{\delta x}{2}\right ] + \frac{\partial}{\partial y} K_3(X_i, Y_j) \cdot \left [-\frac{\delta y}{2}, \frac{\delta y}{2}\right ].  
    \end{aligned}
    \end{equation*}
    Now we need to compute $K_3^{s,o}$ pointwise and its derivatives in intervals. To do so we are going to exploit the structure of $K_3^{s,o}$ and that we have to compute it in all the squares to be more efficient. To compute it pointwise we are going to use the following recursion formula 
    \begin{equation}\label{K3so_recursion_formula}   
        K_3^{s,o}(x_i^{\text{mid}},y_j^{\text{mid}}) = \int_{y_j^{\text{mid}}}^{y_{j+1}^{\text{mid}}} K_3^{s,e}(x,y') dy' + K_3^{s,o}(x,y_{j+1}^\text{mid})
    \end{equation}
    and to compute the first value $K_3^{s,o}(x_i,y_{N^{\text{grid}}})$ we are going to compute the integral between $y_{N^{\text{grid}}}^{\text{mid}}$ and $\pi/m$. To compute each integral we are going to use the expression \eqref{K3so_integrate_interval} to reduce it to computing 12 very small integrals. Each of these integrals is going to be computed in the same way we have computed the integrals of $K_1, K_2$: we are going to split the integral into the singular domain $-\delta, \delta$ and the rest, where we will integrate using a $\mathscr{C}^2$ quadrature, for which we compute the second derivative of $K_3(x,x-z)$, with the adaptive interval method from the '\texttt{methods}' file.  In order to not have the problem of adding the singularity twice, we add it at the beginning of the iteration procedure, we use the second inequality of Lemma \ref{singularities_K34} to bound the singularity. We use as the integration parameters a relative tolerance of $10^{-4}$ and $\delta_{K_3} = 5 \cdot 10^{-6}$. Notice that we write the summation from $j = -1$ to $m-2$ so the so the singular set in the integrals is only when $z = 0$ and does not include $z = 2\pi$.%, which it would if we had written the summation from $j = 0$ to $m-1$ like in the definition. Because of periodicity this doesn't affect the function, is just a trick to make it easier to code the splitting of the integrals in the singular set.
    
    We are left with computing the derivatives, first using \eqref{K3so_integrate_interval} with the upper limit being $\pi/m$ we compute
    \begin{align*}
        \frac{\partial}{\partial x } K_3^{s,o}(x,y) &= \sum_{j =-1}^{m-2} \left \{ K_3\left (x,y + \frac{2\pi j}{m} \right )-K_3 \left (x,\frac{\pi}{m}+\frac{2\pi j}{m} \right )  +  \int_{x-\frac{\pi}{m} - \frac{2\pi j}{m}}^{x - y - \frac{2\pi j}{m}} \frac{d}{dx}K_3(x, x-z) dz \right .\\
        &\left . K_3 \left (x,\frac{\pi}{m}+\frac{2\pi j}{m} \right ) - K_3\left (x,-y - \frac{2\pi j}{m} \right ) + \int_{x+y+ \frac{2\pi j}{m}}^{x + \frac{\pi}{m} + \frac{2\pi j}{m}} \frac{d}{dx} K_3(x,x-z) dz \right \}, \\
        \frac{\partial}{\partial y } K_3^{s,o}(x,y) &= \sum_{j=-1}^{m-2}  -K_3\left (x, y+\frac{2\pi j}{m} \right ) dz -  K_3\left (x,-y - \frac{2\pi j}{m}\right ) dz.
    \end{align*}
    which we can write as
    \begin{align*}
        \frac{\partial}{\partial x} K_3^{s,o}(x,y) &= K_3^{s,odd}(x,y) - K_3^{s,odd}(x,\pi/m) + \int_{y}^{\frac{\pi}{m}} \frac{d}{dx} \left (K_3(x,x-z) \right)^{s,e} (x,y') dy', \\
        \frac{\partial}{\partial y} K_3^{s,o}(x,y) &= - K_3^{s,e}(x,y),
    \end{align*}
    where $K_3^{s,e}$ is defined in \eqref{defK3} and $K_3^{s, odd}$, $\frac{d}{dx}(K_3(x,x-z))^{s,e}$ are defined as
    \begin{equation*}
    \begin{aligned}
        K_3^{s,odd}(x,y) &= \sum_{j =-1}^{m-2} K_3\left (x,y + \frac{2\pi j}{m} \right ) - K_3\left (x,-y - \frac{2\pi j}{m} \right ), \\
        \frac{d}{dx}(K_3(x,x-z))^{s,e}(x, y) &= \sum_{j =-1}^{m-2}  \frac{d}{dx}(K_3(x,x-z))\left (x,x-y - \frac{2\pi j}{m} \right ) + \frac{d}{dx}(K_3(x,x-z))\left (x,x+y + \frac{2\pi j}{m} \right ).
    \end{aligned}
    \end{equation*}
    Then to compute $\frac{\partial }{\partial y} K_3^{s,o} (X_i, Y_j) \cdot [-\delta y/2, \delta y/ 2]$, we change $A$ by $A_\delta$ in the expression and do the same as in \eqref{K4_delta_error} to bound the singularity, the only difference is that we have to multiply the singularity error by $\delta y/2$ . We use $\delta_{K_3}^2 = 5 \cdot 10^{-5}$. For the first two terms of $\frac{\partial}{\partial x} K_3^{s,o}$ we do the same with $K_3^{s,odd}$ instead of $K_3^{s,e}$, the singularity bound is the same. 
    It is only left to compute the integral part of $\frac{\partial}{\partial x} K_3^{s,o}$ evaluated at $X_i, y_j^{\text{mid}}$. To do that, we use the following recursion formula analogous to \eqref{K3so_recursion_formula}
    \begin{equation*}
    \begin{aligned}
        \frac{\partial}{\partial x} K_3^{s,o}(X_i, y_j^{\text{mid}}) \subseteq \int_{y_j^\text{mid}}^{y_{j+1}^{\text{mid}}}  \frac{d}{dx}(K_3(x,x-z))^{s,e}(X_i, y') dy' + \frac{\partial}{\partial x} K_3^{s,o}(X_i, y_{j+1}^{\text{mid}}).      
    \end{aligned}
    \end{equation*}
    The main difference is that instead of splitting the integral and use a $\mathscr{C}^2$ quadrature in the non singular domain, what we do is just a $\mathscr{C}^0$ quadrature, so we compute the integral as
    \begin{equation*}
        \int_{y_j^\text{mid}}^{y_{j+1}^{\text{mid}}}  \frac{d}{dx}(K_3(x,x-z))^{s,e}(X_i, y') dy' \subseteq \frac{d}{dx}(K_3(x,x-z))^{s,e}(X_i, [y_{j}^\text{mid}, y_{j+1}^{\text{mid}}]) \cdot (y_{j+1}^{\text{mid}} - y_{j}^\text{mid}).
    \end{equation*}
    Notice that since the function has a singularity, we can't do that straightforwardly: instead we use the same trick as before, we substitute $A$ by $A_\delta$ and add the corresponding singularity, given by the last inequality of Lemma \ref{singularities_K34}. We use $\delta = 5 \cdot 10^{-2}$. Here note that there is a difference with the previous uses of $A_\delta$, because we have the integral of a singularity, so it is well defined for all $x,y$, we add the singularity pointwise (in the first iteration of the recursion, as in \eqref{K3so_recursion_formula}). We can do this because we are doing a $\mathscr{C}^0$ enclosure, so truncating the function doesn't affect the integration method. It is the same principle that in the other uses of $A_\delta$, since we are also using a $\mathscr{C}^0$ enclosure to compute the $L^1 L^\infty$ norm.

    We have explained how for every $X_i, Y_j$ we compute $(\Tilde{K}-K_F)(X_i, Y_j)$, where $\Tilde{K}$ is $K$ with the truncated singularities of $K_4^{s,o}, \frac{\partial}{\partial x} K_3^{s,o}, \frac{\partial}{\partial y} K_3^{s,o}$ so to put it all together, what we do is, for every $X_i$ we compute $\Tilde{K}-K_F$ in $X_i, Y_j$ for all $j$, starting with $j= N^\text{grid}$ to be more efficient with the recursive formula of $K_3^{s,o}$. Then accumulating the values appropriately, we bound the corresponding norms
    \begin{align*}
        \| K-K_F\|_{L_y^1 L_x^\infty} &\leq \max_{i=1,..., N^\text{grid}}\sum_{j= 1}^{N^\text{grid}} |\Tilde{K}-K_F|(X_i, Y_j) \cdot \delta y  + \frac{\text{sing}K_4 + 2\text{sing}K_3 \cdot \delta x /2 + \text{sing}K_3 \cdot \delta y / 2 }{C_{K_1}}, \\
        \| K-K_F\|_{L_x^1 L_y^\infty} &\leq \max_{j=1, ..., N^\text{grid}}\sum_{j= 1}^{N^\text{grid}} |\Tilde{K}-K_F|(X_i, Y_j) \cdot \delta x + \frac{\text{sing}K_4 + 2\text{sing}K_3 \cdot \delta x /2 + \text{sing}K_3 \cdot \delta y / 2 }{C_{K_1}},
    \end{align*}
    where $\text{sing}K_i$ are the intervals obtained with the bounds of Lemma \ref{singularities_K34}, that satisfy $\int_{-\delta}^\delta |K_i(x,z)| dz \cdot [-1,1] \subseteq \text{sing}K_i$ for all $x$.
\end{detailssth}

\begin{detailssth}{Lemmas \ref{CTildeKone}, \ref{FixedPointConstantChecking}} The implementation is straightforward, although some constants have long expressions. We define the basic constants (i.e. $\varepsilon, \Omega, M_R, I_i, C_{K_1}, $etc.) and use the definition of more complex constants (i.e. $b_i, C_N, C_E, C_T, $etc.) to obtain their values.
\end{detailssth}

\begin{table}[htbp]
\centering
\begin{tabular}{|l|c|c|c|c|}
\hline
\textbf{Lemma} & \textbf{Compilation} & \textbf{Executable}  & \textbf{CPUs} & \textbf{Avg. time} (HH:MM:SS)  \\ \hline
Lemmas \ref{boundR0}, \ref{MdR0}, \ref{MddR0}, \ref{MdddR_MddddR} & fast\_lemmas & fast\_lemmas & 1 & $\sim$ 00:00:00 \\ \hline
Lemma \ref{Ineq_NonConvexity} & fast\_lemmas & fast\_lemmas & 1 & $\sim$ 00:00:00 \\ \hline
Lemma \ref{CE0} & boundE0 & boundE0\_parallel & 64 & 08:43:01\\ \hline
Lemma \ref{ddE0_bound} & medium\_lemmas& medium\_lemmas & 1 & 00:06:30 \\ \hline
Lemma \ref{K1bound} & medium\_lemmas & medium\_lemmas & 1 & 00:05:56\\ \hline
Lemma \ref{I+Ainvert} & fast\_lemmas & fast\_lemmas & 1 &  00:00:03 \\ \hline
Lemma \ref{L-L_Fbound} & bound\_K-KF& bound\_K-KF\_parallel & 64 & 29:11:11\\ \hline
Lemma \ref{CTildeKone} & constant\_checking & constant\_checking & 1 &  $\sim$ 00:00:00 \\ \hline
Lemma \ref{FixedPointConstantChecking} & constant\_checking & constant\_checking & 1 & $\sim$ 00:00:00  \\ \hline
\end{tabular}
\caption{Commands and performance of the code for different Lemmas} 
\label{table:lemmas_comp_run}
\end{table}

\section{Explicit coefficients of the approximate solution}
\subsection{Coefficients of the solution}
\label{ck_coefs}

\begin{definition}
    The coefficients $\{c_k\}$, representing the approximate solution $R_0$ are given in Table \ref{tab:coefficients}.
   
    \begin{table}[htbp]
    \centering
    \begin{tabular}{cccc}
        \hline
        \( k \) & \( c_k \) & \( k \) & \( c_k \) \\
        \hline
        0 & 1 & 8 & 5.787204471053671e-05 \\
        1 & 6.990356082734282e-02 & 9 & 2.719406328587438e-05 \\
        2 & 1.481485591846559e-02 & 10 & 1.302302594759571e-05 \\
        3 & 4.562581793349477e-03 & 11 & 6.333059955650702e-06 \\
        4 & 1.662067385043346e-03 & 12 & 3.119022088447539e-06 \\
        5 & 6.660065136771192e-04 & 13 & 1.552543629347692e-06 \\
        6 & 2.837414439454572e-04 & 14 & 7.798356164400106e-07 \\
        7 & 1.261501302020939e-04 & 15 & 3.947768555872326e-07 \\
        \hline
        16 & 2.012104965444521e-07 & 24 & 1.084839772995918e-09 \\
        17 & 1.031670056456975e-07 & 25 & 5.732502021897411e-10 \\
        18 & 5.317712460135162e-08 & 26 & 3.036541152512813e-10 \\
        19 & 2.753934675108659e-08 & 27 & 1.612094123238438e-10 \\
        20 & 1.432242380745116e-08 & 28 & 8.576430342713042e-11 \\
        21 & 7.477109926654982e-09 & 29 & 4.571554686889860e-11 \\
        22 & 3.916974811811910e-09 & 30 & 2.441210348018661e-11 \\
        23 & 2.058418356388450e-09 &  &  \\
        \hline
    \end{tabular}
    \caption{Coefficients \( c_k \) for \( k = 0, 1, \ldots, 30 \)}
    \label{tab:coefficients}
\end{table}
  
\end{definition}

\subsection{Coefficients of the finite rank approximation of the operator $K$}
\label{Akl_coefs}

\begin{definition}
    The coefficients of the $N\times N$ matrix $A_{k,l}$ can be found in the supplementary material, in the file \texttt{A.txt}.
\end{definition}

%\printbibliography

\bibliographystyle{abbrv}
\bibliography{references.bib}

\def\cprime{$'$}
\begin{thebibliography}{100}

\bibitem{Arioli-Gazzola-Koch:uniqueness-bifurcation-ns}
G.~Arioli, F.~Gazzola, and H.~Koch.
\newblock Uniqueness and bifurcation branches for planar steady
  {N}avier--{S}tokes equations under {N}avier boundary conditions.
\newblock {\em Journal of Mathematical Fluid Mechanics}, 23, 2021.
\newblock Article 49.

\bibitem{Arioli-Koch:cap-stationary-ks}
G.~Arioli and H.~Koch.
\newblock Computer-assisted methods for the study of stationary solutions in
  dissipative systems, applied to the {K}uramoto-{S}ivashinski equation.
\newblock {\em Arch. Ration. Mech. Anal.}, 197(3):1033--1051, 2010.

\bibitem{Bedrossian-PunshonSmith:chaos-stochastic-2d-galerkin-ns}
J.~Bedrossian and S.~Punshon-Smith.
\newblock Chaos in stochastic 2d {G}alerkin-{N}avier-{S}tokes.
\newblock {\em Comm. Math. Phys.}, 405(4):Paper No. 107, 42, 2024.

\bibitem{Berti-Hassainia-Masmoudi:time-quasiperiodic-vortex-patches}
M.~Berti, Z.~Hassainia, and N.~Masmoudi.
\newblock Time quasi-periodic vortex patches of {E}uler equation in the plane.
\newblock {\em Invent. Math.}, 233(3):1279--1391, 2023.

\bibitem{Bertozzi-Constantin:global-regularity-vortex-patches}
A.~L. Bertozzi and P.~Constantin.
\newblock Global regularity for vortex patches.
\newblock {\em Comm. Math. Phys.}, 152(1):19--28, 1993.

\bibitem{Bogosel-Bucur:polygonal-faber-krahn-validated}
B.~Bogosel and D.~Bucur.
\newblock Polygonal {F}aber-{K}rahn inequality: Local minimality via validated
  computing.
\newblock {\em arXiv preprint arXiv:2406.11575}, 2024.

\bibitem{Breden-Engel:cap-chaos-hopf-systems}
M.~Breden and M.~Engel.
\newblock Computer-assisted proof of shear-induced chaos in stochastically
  perturbed {H}opf systems.
\newblock {\em Ann. Appl. Probab.}, 33(2):1052--1094, 2023.

\bibitem{Buckmaster-CaoLabora-GomezSerrano:implosion-compressible}
T.~Buckmaster, G.~Cao-Labora, and J.~G\'omez-Serrano.
\newblock Smooth imploding solutions for 3d compressible fluids.
\newblock {\em Forum of Mathematics, Pi}, 2024.
\newblock To appear.

\bibitem{Burbea:motions-vortex-patches}
J.~Burbea.
\newblock Motions of vortex patches.
\newblock {\em Lett. Math. Phys.}, 6(1):1--16, 1982.

\bibitem{Cadiot:proofs-existence-stability-capillary-gravity-whitham}
M.~Cadiot.
\newblock Constructive proofs of existence and stability of solitary waves in
  the capillary-gravity whitham equation.
\newblock {\em arXiv preprint arXiv:2403.18718}, 2024.

\bibitem{Cao-Wan-Wang:nonlinear-stability-patches-bounded-domains}
D.~Cao, J.~Wan, and G.~Wang.
\newblock Nonlinear orbital stability for planar vortex patches.
\newblock {\em Proc. Amer. Math. Soc.}, 147(2):775--784, 2019.

\bibitem{Carrillo-Mateu-Mora-Rondi-Scardia-Verdera:dislocation-ellipses}
J.~A. Carrillo, J.~Mateu, M.~G. Mora, L.~Rondi, L.~Scardia, and J.~Verdera.
\newblock The ellipse law: {K}irchhoff meets dislocations.
\newblock {\em Comm. Math. Phys.}, 373(2):507--524, 2020.

\bibitem{Castelli-Gameiro-Lessard:rigorous-numerics-ill-posed-pde}
R.~Castelli, M.~Gameiro, and J.-P. Lessard.
\newblock Rigorous numerics for ill-posed {PDE}s: periodic orbits in the
  {B}oussinesq equation.
\newblock {\em Arch. Ration. Mech. Anal.}, 228(1):129--157, 2018.

\bibitem{Castro-Cordoba-GomezSerrano:analytic-vstates-ellipses}
A.~Castro, D.~C{\'o}rdoba, and J.~G{\'o}mez-Serrano.
\newblock Uniformly rotating analytic global patch solutions for active
  scalars.
\newblock {\em Annals of PDE}, 2(1):1--34, 2016.

\bibitem{Castro-Cordoba-GomezSerrano:global-smooth-solutions-sqg}
A.~Castro, D.~C{\'o}rdoba, and J.~G{\'o}mez-Serrano.
\newblock Global smooth solutions for the inviscid {S}{Q}{G} equation.
\newblock {\em Memoirs of the AMS}, 266(1292):89 pages, 2020.

\bibitem{Cerretelli-Williamson:new-family-vortices}
C.~Cerretelli and C.~H.~K. Williamson.
\newblock A new family of uniform vortices related to vortex configurations
  before merging.
\newblock {\em Journal of Fluid Mechanics}, 493:219--229, 10 2003.

\bibitem{Chemin:persistance-structures-fluides-incompressibles}
J.-Y. Chemin.
\newblock Persistance de structures g\'eom\'etriques dans les fluides
  incompressibles bidimensionnels.
\newblock {\em Ann. Sci. \'Ecole Norm. Sup. (4)}, 26(4):517--542, 1993.

\bibitem{Chen-Hou-Huang:blowup-degregorio}
J.~Chen, T.~Y. Hou, and D.~Huang.
\newblock On the {F}inite {T}ime {B}lowup of the {D}e {G}regorio {M}odel for
  the 3{D} {E}uler {E}quations.
\newblock {\em Communications on Pure and Applied Mathematics},
  74(6):1282--1350, 2021.

\bibitem{Choi-Jeong:stability-instability-kelvin-waves}
K.~Choi and I.-J. Jeong.
\newblock Stability and instability of {K}elvin waves.
\newblock {\em Calc. Var. Partial Differential Equations}, 61(6):Paper No. 221,
  38, 2022.

\bibitem{Choi-Jeong-Sim:existence-sadovskii-vortex-patch}
K.~Choi, I.-J. Jeong, and Y.-J. Sim.
\newblock On existence of sadovskii vortex patch: A touching pair of symmetric
  counter-rotating uniform vortex.
\newblock {\em arXiv preprint arXiv:2406.11379}, 2024.

\bibitem{Constantin-Titi:evolution-nearly-circular-vortex-patches}
P.~Constantin and E.~S. Titi.
\newblock On the evolution of nearly circular vortex patches.
\newblock {\em Comm. Math. Phys.}, 119(2):177--198, 1988.

\bibitem{Dahne:highest-cusped-waves-fkdv}
J.~Dahne.
\newblock Highest cusped waves for the fractional {K}d{V} equations.
\newblock {\em J. Differential Equations}, 401:550--670, 2024.

\bibitem{Dahne-GomezSerrano:highest-wave-burgershilbert}
J.~Dahne and J.~G\'{o}mez-Serrano.
\newblock Highest cusped waves for the {B}urgers-{H}ilbert equation.
\newblock {\em Arch. Ration. Mech. Anal.}, 247(5):Paper No. 74, 55, 2023.

\bibitem{Dahne-GomezSerrano-Hou:counterexample-payne}
J.~Dahne, J.~G\'omez-Serrano, and K.~Hou.
\newblock A counterexample to {Payne's} nodal line conjecture with few holes.
\newblock {\em Commun. Nonlinear Sci.}, 103:105957, Dec. 2021.

\bibitem{Day-Lessard-Mischaikow:validated-continuation-equilibria-pde}
S.~Day, J.-P. Lessard, and K.~Mischaikow.
\newblock Validated continuation for equilibria of {PDE}s.
\newblock {\em SIAM J. Numer. Anal.}, 45(4):1398--1424, 2007.

\bibitem{delaHoz-Hassainia-Hmidi-Mateu:vstates-disk-euler}
F.~de~la Hoz, Z.~Hassainia, T.~Hmidi, and J.~Mateu.
\newblock An analytical and numerical study of steady patches in the disc.
\newblock {\em Anal. PDE}, 9(7):1609--1670, 2016.

\bibitem{Hmidi-delaHoz-Mateu-Verdera:doubly-connected-vstates-euler}
F.~de~la Hoz, T.~Hmidi, J.~Mateu, and J.~Verdera.
\newblock Doubly connected {$V$}-states for the planar {E}uler equations.
\newblock {\em SIAM J. Math. Anal.}, 48(3):1892--1928, 2016.

\bibitem{Deem-Zabusky:vortex-waves-stationary}
G.~S. Deem and N.~J. Zabusky.
\newblock Vortex waves: Stationary "{V}-states", interactions, recurrence, and
  breaking.
\newblock {\em Physical Review Letters}, 40(13):859--862, 1978.

\bibitem{Dhanak:stability-polygon-vertices}
M.~R. Dhanak.
\newblock Stability of a regular polygon of finite vortices.
\newblock {\em J. Fluid Mech.}, 234:297--316, 1992.

\bibitem{Dominguez-Enciso-PeraltaSalas:piecewise-smooth-stationary-euler}
M.~Dom\'{\i}nguez-V\'{a}zquez, A.~Enciso, and D.~Peralta-Salas.
\newblock Piecewise smooth stationary {E}uler flows with compact support via
  overdetermined boundary problems.
\newblock {\em Arch. Ration. Mech. Anal.}, 239(3):1327--1347, 2021.

\bibitem{Dritschel:stability-energetics-corotating-vortices}
D.~G. Dritschel.
\newblock The stability and energetics of corotating uniform vortices.
\newblock {\em Journal of Fluid Mechanics}, 157:95--134, 8 1985.

\bibitem{Dritschel:repeated-filamentation-2d-vorticity}
D.~G. Dritschel.
\newblock The repeated filamentation of two-dimensional vorticity interfaces.
\newblock {\em J. Fluid Mech.}, 194:511--547, 1988.

\bibitem{Dritschel:general-theory-2d-vortex}
D.~G. Dritschel.
\newblock A general theory for two-dimensional vortex interactions.
\newblock {\em J. Fluid Mech.}, 293:269--303, 1995.

\bibitem{Elcrat-Fornberg-Miller:stability-vortices-cylinder}
A.~Elcrat, B.~Fornberg, and K.~Miller.
\newblock Stability of vortices in equilibrium with a cylinder.
\newblock {\em Journal of Fluid Mechanics}, 544:53--68, 2005.

\bibitem{Elcrat-Protas:framework-linear-stability-vortices}
A.~Elcrat and B.~Protas.
\newblock A framework for linear stability analysis of finite-area vortices.
\newblock {\em Proc. R. Soc. Lond. Ser. A Math. Phys. Eng. Sci.},
  469(2152):20120709, 15, 2013.

\bibitem{Enciso-Fernandez-Ruiz:smooth-nonradial-stationary-euler}
A.~Enciso, A.~J. Fern{\'a}ndez, and D.~Ruiz.
\newblock Smooth nonradial stationary euler flows on the plane with compact
  support.
\newblock {\em arXiv preprint arXiv:2406.04414}, 2024.

\bibitem{Enciso-Fernandez-Ruiz-Sicbaldi:schiffer-annuli-stationary-planar-euler}
A.~Enciso, A.~J. Fern{\'a}ndez, D.~Ruiz, and P.~Sicbaldi.
\newblock A schiffer-type problem for annuli with applications to stationary
  planar euler flows.
\newblock {\em arXiv preprint arXiv:2309.07977}, 2023.

\bibitem{Enciso-GomezSerrano-Vergara:convexity-cusped-whitham}
A.~Enciso, J.~G\'omez-Serrano, and B.~Vergara.
\newblock Convexity of cusped {W}hitham waves.
\newblock {\em Arxiv preprint arXiv:1810.10935}, 2018.

\bibitem{Fazekas-Pacella-Plum:nonradial-lane-emden-ball}
B.~Fazekas, F.~Pacella, and M.~Plum.
\newblock Approximate nonradial solutions for the {L}ane-{E}mden problem in the
  ball.
\newblock {\em Adv. Nonlinear Anal.}, 11(1):268--284, 2022.

\bibitem{Figueras-DeLaLLave:cap-periodic-orbits-kuramoto}
J.-L. Figueras and R.~de~la Llave.
\newblock Numerical computations and computer assisted proofs of periodic
  orbits of the {K}uramoto-{S}ivashinsky equation.
\newblock {\em SIAM J. Appl. Dyn. Syst.}, 16(2):834--852, 2017.

\bibitem{Figueras-Gameiro-Lessard-DeLaLLave:framework-cap-invariant-objects}
J.-L. Figueras, M.~Gameiro, J.-P. Lessard, and R.~de~la Llave.
\newblock A framework for the numerical computation and a posteriori
  verification of invariant objects of evolution equations.
\newblock {\em SIAM J. Appl. Dyn. Syst.}, 16(2):1070--1088, 2017.

\bibitem{Fraenkel:book-maximum-principles-symmetry-elliptic}
L.~E. Fraenkel.
\newblock {\em An introduction to maximum principles and symmetry in elliptic
  problems}, volume 128 of {\em Cambridge Tracts in Mathematics}.
\newblock Cambridge University Press, Cambridge, 2000.

\bibitem{Gameiro-Lessard:periodic-orbits-ks}
M.~Gameiro and J.-P. Lessard.
\newblock A posteriori verification of invariant objects of evolution
  equations: periodic orbits in the {K}uramoto-{S}ivashinsky {PDE}.
\newblock {\em SIAM J. Appl. Dyn. Syst.}, 16(1):687--728, 2017.

\bibitem{Gameiro-Lessard-Mischaikow:validated-continuation-large-parameter-ranges-pde}
M.~Gameiro, J.-P. Lessard, and K.~Mischaikow.
\newblock Validated continuation over large parameter ranges for equilibria of
  {PDE}s.
\newblock {\em Math. Comput. Simulation}, 79(4):1368--1382, 2008.

\bibitem{Garcia:Karman-vortex-street}
C.~Garc\'{\i}a.
\newblock K\'{a}rm\'{a}n vortex street in incompressible fluid models.
\newblock {\em Nonlinearity}, 33(4):1625--1676, 2020.

\bibitem{Garcia:vortex-patch-choreography}
C.~Garc\'{\i}a.
\newblock Vortex patches choreography for active scalar equations.
\newblock {\em J. Nonlinear Sci.}, 31(5):Paper No. 75, 31, 2021.

\bibitem{Garcia-Haziot:global-bifurcation-corotating-counter-rotating}
C.~Garc\'{\i}a and S.~V. Haziot.
\newblock Global bifurcation for corotating and counter-rotating vortex pairs.
\newblock {\em Comm. Math. Phys.}, 402(2):1167--1204, 2023.

\bibitem{GomezSerrano:survey-cap-in-pde}
J.~G\'{o}mez-Serrano.
\newblock Computer-assisted proofs in {PDE}: a survey.
\newblock {\em SeMA J.}, 76(3):459--484, 2019.

\bibitem{GomezSerrano-Orriols:negative-hearing-shape-triangle}
J.~G\'{o}mez-Serrano and G.~Orriols.
\newblock Any three eigenvalues do not determine a triangle.
\newblock {\em J. Differential Equations}, 275:920--938, 2021.

\bibitem{GomezSerrano-Park-Shi:nonradial-stationary-solutions}
J.~G\'omez-Serrano, J.~Park, and J.~Shi.
\newblock Existence of non-trivial non-concentrated compactly supported
  stationary solutions of the 2d {E}uler equation with finite energy.
\newblock {\em Memoirs of the AMS}, 2022.
\newblock To appear.

\bibitem{GomezSerrano-Park-Shi-Yao:rotating-solutions-vortex-sheet-rigidity}
J.~G\'{o}mez-Serrano, J.~Park, J.~Shi, and Y.~Yao.
\newblock Remarks on stationary and uniformly-rotating vortex sheets: rigidity
  results.
\newblock {\em Comm. Math. Phys.}, 386(3):1845--1879, 2021.

\bibitem{GomezSerrano-Park-Shi-Yao:radial-symmetry-stationary-solutions}
J.~G\'{o}mez-Serrano, J.~Park, J.~Shi, and Y.~Yao.
\newblock Symmetry in stationary and uniformly rotating solutions of active
  scalar equations.
\newblock {\em Duke Math. J.}, 170(13):2957--3038, 2021.

\bibitem{GomezSerrano-Park-Shi-Yao:rotating-solutions-vortex-sheet-flexibility}
J.~G\'{o}mez-Serrano, J.~Park, J.~Shi, and Y.~Yao.
\newblock Remarks on stationary and uniformly rotating vortex sheets:
  flexibility results.
\newblock {\em Philos. Trans. Roy. Soc. A}, 380(2226):Paper No. 20210045, 15,
  2022.

\bibitem{Guo-Hadzic-Jang-Schrecker:gravitational-collapse-stars-self-similar}
Y.~Guo, M.~Had\v{z}i\'{c}, J.~Jang, and M.~Schrecker.
\newblock Gravitational collapse for polytropic gaseous stars: self-similar
  solutions.
\newblock {\em Arch. Ration. Mech. Anal.}, 246(2-3):957--1066, 2022.

\bibitem{Guo-Hallstrom-Spirn:dynamics-unstable-kirchhoff-ellipse}
Y.~Guo, C.~Hallstrom, and D.~Spirn.
\newblock Dynamics near an unstable {K}irchhoff ellipse.
\newblock {\em Communications in Mathematical Physics}, 245(2):297--354, 2004.

\bibitem{Hamel-Nadirashvili:rigidity-euler-annulus}
F.~Hamel and N.~Nadirashvili.
\newblock Circular flows for the {E}uler equations in two-dimensional annular
  domains, and related free boundary problems.
\newblock {\em J. Eur. Math. Soc. (JEMS)}, 25(1):323--368, 2023.

\bibitem{Hassainia-Hmidi:steady-asymmetric-vortex-pairs-euler}
Z.~Hassainia and T.~Hmidi.
\newblock Steady asymmetric vortex pairs for {E}uler equations.
\newblock {\em Discrete Contin. Dyn. Syst.}, 41(4):1939--1969, 2021.

\bibitem{Hassainia-Hmidi-Roulley:invariant-kam-annular-vortex-patches}
Z.~Hassainia, T.~Hmidi, and E.~Roulley.
\newblock Invariant {KAM} {T}ori {A}round {A}nnular {V}ortex {P}atches for 2{D}
  {E}uler {E}quations.
\newblock {\em Comm. Math. Phys.}, 405(11):Paper No. 270, 2024.

\bibitem{Hassainia-Masmoudi-Wheeler:global-bifurcation-vortex-patches}
Z.~Hassainia, N.~Masmoudi, and M.~H. Wheeler.
\newblock Global bifurcation of rotating vortex patches.
\newblock {\em Comm. Pure Appl. Math.}, 73(9):1933--1980, 2020.

\bibitem{Hassainia-Roulley:quasiperiodic-euler-boundary}
Z.~Hassainia and E.~Roulley.
\newblock Boundary effects on the emergence of quasi-periodic solutions for
  euler equations.
\newblock {\em arXiv preprint arXiv:2202.10053}, 2022.

\bibitem{Hassainia-Wheeler:multipole-vstates-active-scalars}
Z.~Hassainia and M.~H. Wheeler.
\newblock Multipole {V}ortex {P}atch {E}quilibria for {A}ctive {S}calar
  {E}quations.
\newblock {\em SIAM J. Math. Anal.}, 54(6):6054--6095, 2022.

\bibitem{Hmidi:trivial-solutions-rotating-patches}
T.~Hmidi.
\newblock On the trivial solutions for the rotating patch model.
\newblock {\em J. Evol. Equ.}, 15(4):801--816, 2015.

\bibitem{Hmidi-Mateu:bifurcation-kirchhoff-ellipses}
T.~Hmidi and J.~Mateu.
\newblock Bifurcation of rotating patches from {K}irchhoff vortices.
\newblock {\em Discrete and Continuous Dynamical Systems}, 36(10):5401--5422,
  2016.

\bibitem{Hmidi-Mateu:degenerate-bifurcation-vstates-doubly-connected-euler}
T.~Hmidi and J.~Mateu.
\newblock Degenerate bifurcation of the rotating patches.
\newblock {\em Adv. Math.}, 302:799--850, 2016.

\bibitem{Hmidi-Mateu:existence-corotating-counter-rotating}
T.~Hmidi and J.~Mateu.
\newblock Existence of corotating and counter-rotating vortex pairs for active
  scalar equations.
\newblock {\em Comm. Math. Phys.}, 350(2):699--747, 2017.

\bibitem{Hmidi-Mateu-Verdera:rotating-vortex-patch}
T.~Hmidi, J.~Mateu, and J.~Verdera.
\newblock Boundary regularity of rotating vortex patches.
\newblock {\em Archive for Rational Mechanics and Analysis}, 209(1):171--208,
  2013.

\bibitem{Hmidi-Renault:existence-small-loops-doubly-connected-euler}
T.~Hmidi and C.~Renault.
\newblock Existence of small loops in a bifurcation diagram near degenerate
  eigenvalues.
\newblock {\em Nonlinearity}, 30(10):3821--3852, 2017.

\bibitem{Huang-Tong:steady-contiguous-vortex-patch-dipole}
D.~Huang and J.~Tong.
\newblock Steady contiguous vortex-patch dipole solutions of the 2d
  incompressible euler equation.
\newblock {\em arXiv preprint arXiv:2406.09849}, 2024.

\bibitem{Huang:rigidity-vstate-near-rankine}
Y.~Huang.
\newblock On the rigidity of uniformly rotating vortex patch near the rankine
  vortex.
\newblock {\em arXiv preprint arXiv:2312.12711}, 2023.

\bibitem{Jaquette-Lessard-Takayasu:global-dynamics-nonconservative-nls}
J.~Jaquette, J.-P. Lessard, and A.~Takayasu.
\newblock Global dynamics in nonconservative nonlinear {S}chr\"{o}dinger
  equations.
\newblock {\em Adv. Math.}, 398:Paper No. 108234, 70, 2022.

\bibitem{Johansson:Arb}
F.~Johansson.
\newblock Arb: efficient arbitrary-precision midpoint-radius interval
  arithmetic.
\newblock {\em IEEE Transactions on Computers}, 66:1281--1292, 2017.

\bibitem{Kamm:thesis-shape-stability-patches}
J.~R. Kamm.
\newblock {\em Shape and stability of two-dimensional uniform vorticity
  regions}.
\newblock PhD thesis, California Institute of Technology, 1987.

\bibitem{Kinderlehrer-Nirenberg-Spruck:regularity-elliptic-free-boundary}
D.~Kinderlehrer, L.~Nirenberg, and J.~Spruck.
\newblock Regularity in elliptic free boundary problems.
\newblock {\em J. Analyse Math.}, 34:86--119 (1979), 1978.

\bibitem{Kirchhoff:vorlesungen-math-physik}
G.~Kirchhoff.
\newblock {\em Vorlesungen {\"u}ber mathematische {P}hysik}, volume~1.
\newblock Teubner, 1874.

\bibitem{Kiselev-Luo:illposedness-c2-vortex-patches}
A.~Kiselev and X.~Luo.
\newblock Illposedness of {$C^2$} vortex patches.
\newblock {\em Arch. Ration. Mech. Anal.}, 247(3):Paper No. 57, 49, 2023.

\bibitem{Kobayashi:global-uniqueness-stokes}
K.~Kobayashi.
\newblock On the global uniqueness of {S}tokes' wave of extreme form.
\newblock {\em IMA J. Appl. Math.}, 75(5):647--675, 2010.

\bibitem{Love:stability-ellipses}
A.~E.~H. Love.
\newblock On the {S}tability of certain {V}ortex {M}otions.
\newblock {\em Proc. London Math. Soc.}, 25(1):18--42, 1893.

\bibitem{LuzzattoFegiz:bifurcation-stability-opposite-signed-vortex}
P.~Luzzatto-Fegiz.
\newblock Bifurcation structure and stability in models of opposite-signed
  vortex pairs.
\newblock {\em Fluid Dyn. Res.}, 46(3):031408, 14, 2014.

\bibitem{LuzzattoFegiz-Williamson:stability-elliptical-vortices-ivi-diagrams}
P.~Luzzatto-Fegiz and C.~H. Williamson.
\newblock Stability of elliptical vortices from
  “imperfect--velocity--impulse” diagrams.
\newblock {\em Theoretical and Computational Fluid Dynamics}, 24:181--188,
  2010.

\bibitem{LuzzattoFegiz-Williamson:stability-conservative-flows-steady-solutions-variational-argument}
P.~Luzzatto-Fegiz and C.~H.~K. Williamson.
\newblock Stability of conservative flows and new steady-fluid solutions from
  bifurcation diagrams exploiting a variational argument.
\newblock {\em Phys. Rev. Lett.}, 104:044504, Jan 2010.

\bibitem{LuzzattoFegiz-Williamson:efficient-numerical-method-steady-uniform-vortices}
P.~Luzzatto-Fegiz and C.~H.~K. Williamson.
\newblock An efficient and general numerical method to compute steady uniform
  vortices.
\newblock {\em J. Comput. Phys.}, 230(17):6495--6511, 2011.

\bibitem{LuzzattoFegiz-Williamson:resonant-instability-2d-vortex}
P.~Luzzatto-Fegiz and C.~H.~K. Williamson.
\newblock Resonant instability in two-dimensional vortex arrays.
\newblock {\em Proc. R. Soc. Lond. Ser. A Math. Phys. Eng. Sci.},
  467(2128):1164--1185, 2011.

\bibitem{LuzzattoFegiz-Williamson:stability-2d-flows-imperfect-velocity-impulse}
P.~Luzzatto-Fegiz and C.~H.~K. Williamson.
\newblock Determining the stability of steady two-dimensional flows through
  imperfect velocity-impulse diagrams.
\newblock {\em J. Fluid Mech.}, 706:323--350, 2012.

\bibitem{LuzzattoFegiz-Williamson:structure-stability-von-karman-street}
P.~Luzzatto-Fegiz and C.~H.~K. Williamson.
\newblock {Structure and stability of the finite-area von Kármán street}.
\newblock {\em Physics of Fluids}, 24(6):066602, 06 2012.

\bibitem{Mateu-Mora-Rondi-Scardia-Verdera:explicit-minimisers-anisotropic-coulomb-3d}
J.~Mateu, M.~G. Mora, L.~Rondi, L.~Scardia, and J.~Verdera.
\newblock Explicit minimisers for anisotropic {C}oulomb energies in 3{D}.
\newblock {\em Adv. Math.}, 434:Paper No. 109333, 28, 2023.

\bibitem{Mitchell-Rossi:evolution-kirchhoff-elliptic-vortices}
T.~B. Mitchell and L.~F. Rossi.
\newblock The evolution of {K}irchhoff elliptic vortices.
\newblock {\em Physics of Fluids}, 20(5), 2008.

\bibitem{Moore-Bierbaum:methods-applications-interval-analysis}
R.~Moore and F.~Bierbaum.
\newblock {\em Methods and applications of interval analysis}, volume~2.
\newblock Society for Industrial \& Applied Mathematics, 1979.

\bibitem{Nakao:numerical-approach-elliptic-problems}
M.~T. Nakao.
\newblock A numerical approach to the proof of existence of solutions for
  elliptic problems.
\newblock {\em Japan J. Appl. Math.}, 5(2):313--332, 1988.

\bibitem{Nakao:nonlinear-parabolic-problems}
M.~T. Nakao.
\newblock Solving nonlinear parabolic problems with result verification. {I}.
  {O}ne-space-dimensional case.
\newblock In {\em Proceedings of the {I}nternational {S}ymposium on
  {C}omputational {M}athematics ({M}atsuyama, 1990)}, volume~38, pages
  323--334, 1991.

\bibitem{Nakao-Plum-Watanabe:cap-for-pde-book}
M.~T. Nakao, M.~Plum, and Y.~Watanabe.
\newblock {\em Numerical verification methods and computer-assisted proofs for
  partial differential equations}, volume~53 of {\em Springer Series in
  Computational Mathematics}.
\newblock Springer, Singapore, 2019.

\bibitem{Overman:steady-state-euler-local-analysis-vstates}
E.~A. Overman, II.
\newblock Steady-state solutions of the {E}uler equations in two dimensions.
  {II}. {L}ocal analysis of limiting {$V$}-states.
\newblock {\em SIAM J. Appl. Math.}, 46(5):765--800, 1986.

\bibitem{Park:quantitative-estimates-v-states}
J.~Park.
\newblock Quantitative estimates for uniformly-rotating vortex patches.
\newblock {\em Adv. Math.}, 411(part A):Paper No. 108779, 55, 2022.

\bibitem{Pierrehumbert:v-states-cusp}
R.~T. Pierrehumbert.
\newblock A family of steady, translating vortex pairs with distributed
  vorticity.
\newblock {\em Journal of Fluid Mechanics}, 99(1):129–144, 1980.

\bibitem{Plum:H2-estimates-elliptic-bvp}
M.~Plum.
\newblock Explicit {$H_2$}-estimates and pointwise bounds for solutions of
  second-order elliptic boundary value problems.
\newblock {\em J. Math. Anal. Appl.}, 165(1):36--61, 1992.

\bibitem{Plum:numerical-existence-nonlinear-elliptic-bvps}
M.~Plum.
\newblock Numerical existence proofs and explicit bounds for solutions of
  nonlinear elliptic boundary value problems.
\newblock {\em Computing}, 49(1):25--44, 1992.

\bibitem{Ruiz:symmetry-compactly-supported-steady-2d-euler}
D.~Ruiz.
\newblock Symmetry results for compactly supported steady solutions of the 2{D}
  {E}uler equations.
\newblock {\em Arch. Ration. Mech. Anal.}, 247(3):Paper No. 40, 25, 2023.

\bibitem{Saffman-Szeto:equilibrium-shapes-equal-uniform-vortices}
P.~Saffman and R.~Szeto.
\newblock {Equilibrium shapes of a pair of equal uniform vortices}.
\newblock {\em {Physics of Fluids}}, 23(12):2339--2342, 1980.

\bibitem{Saffman-Tanveer:touching-pair-unniform-vortices}
P.~G. Saffman and S.~Tanveer.
\newblock The touching pair of equal and opposite uniform vortices.
\newblock {\em Phys. Fluids}, 25(11):1929--1930, 1982.

\bibitem{Sideris-Vega:stability-L1-patches}
T.~C. Sideris and L.~Vega.
\newblock Stability in {$L^1$} of circular vortex patches.
\newblock {\em Proc. Amer. Math. Soc.}, 137(12):4199--4202, 2009.

\bibitem{Tang:nonlinear-stability-vortex-patches}
Y.~Tang.
\newblock Nonlinear stability of vortex patches.
\newblock {\em Trans. Amer. Math. Soc.}, 304(2):617--638, 1987.

\bibitem{Tucker:validated-numerics-book}
W.~Tucker.
\newblock {\em Validated numerics}.
\newblock Princeton University Press, Princeton, NJ, 2011.
\newblock A short introduction to rigorous computations.

\bibitem{Turkington:corotating-vortices}
B.~Turkington.
\newblock Corotating steady vortex flows with {$N$}-fold symmetry.
\newblock {\em Nonlinear Anal.}, 9(4):351--369, 1985.

\bibitem{vandenBerg-Breden-Lessard-vanVeen:periodic-orbits-ns}
J.~B. van~den Berg, M.~Breden, J.-P. Lessard, and L.~van Veen.
\newblock Spontaneous periodic orbits in the {N}avier--{S}tokes flow.
\newblock {\em Journal of Nonlinear Science}, 31, 2021.
\newblock Article 41.

\bibitem{vandenBerg-Henot-Lessard:radial-solutions-semilinear-elliptic}
J.~B. van~den Berg, O.~H\'{e}not, and J.-P. Lessard.
\newblock Constructive proofs for localised radial solutions of semilinear
  elliptic systems on {$\Bbb R^d$}.
\newblock {\em Nonlinearity}, 36(12):6476--6512, 2023.

\bibitem{vandenBerg-Lessard:chaotic-braided-solutions-swift-hohenberg}
J.~B. van~den Berg and J.-P. Lessard.
\newblock Chaotic braided solutions via rigorous numerics: chaos in the
  {S}wift-{H}ohenberg equation.
\newblock {\em SIAM J. Appl. Dyn. Syst.}, 7(3):988--1031, 2008.

\bibitem{Wan:stability-rotating-vortex-patches}
Y.~H. Wan.
\newblock The stability of rotating vortex patches.
\newblock {\em Comm. Math. Phys.}, 107(1):1--20, 1986.

\bibitem{Wang-Xu-Zhou:degenerate-bifurcations-2fold-doubly-connected-patches}
Y.~Wang, X.~Xu, and M.~Zhou.
\newblock Degenerate bifurcations of two-fold doubly-connected vortex patches.
\newblock {\em arXiv preprint arXiv:2212.01869}, 2022.

\bibitem{Wang-Zhang-Zhou:Boundary-regularity-v-states}
Y.~Wang, G.~Zhang, and M.~Zhou.
\newblock Boundary regularity of uniformly rotating vortex patches and an
  unstable elliptic free boundary problem.
\newblock {\em Arxiv preprint arXiv:2306.03498}, 2023.

\bibitem{Wu-Overman-Zabusky:steady-state-Euler-2d}
H.~M. Wu, E.~A. Overman, II, and N.~J. Zabusky.
\newblock Steady-state solutions of the {E}uler equations in two dimensions:
  rotating and translating {$V$}-states with limiting cases. {I}. {N}umerical
  algorithms and results.
\newblock {\em J. Comput. Phys.}, 53(1):42--71, 1984.

\bibitem{Yudovich:Nonstationary-ideal-incompressible}
V.~I. Yudovich.
\newblock Non-stationary flows of an ideal incompressible fluid.
\newblock {\em \v Z. Vy\v cisl. Mat. i Mat. Fiz.}, 3:1032--1066, 1963.

\bibitem{Zgliczynski:periodic-orbit-kuramoto}
P.~Zgliczy\'nski.
\newblock Rigorous numerics for dissipative partial differential equations.
  {II}. {P}eriodic orbit for the {K}uramoto-{S}ivashinsky {PDE}---a
  computer-assisted proof.
\newblock {\em Found. Comput. Math.}, 4(2):157--185, 2004.

\bibitem{Zgliczynski-Mischaikow:rigorous-numerics-kuramoto}
P.~Zgliczy{\'n}ski and K.~Mischaikow.
\newblock Rigorous numerics for partial differential equations: the
  {K}uramoto-{S}ivashinsky equation.
\newblock {\em Found. Comput. Math.}, 1(3):255--288, 2001.

\end{thebibliography}

\begin{tabular}{l}
  \textbf{Gerard Castro-L\'opez} \\
  {Department of Mathematics} \\
  {Brown University} \\
  {010 Kassar House, 151 Thayer St.} \\
  {Providence, RI 02912, USA} \\ \\
%  {Email: gerard\_castro\_lopez@brown.edu} \\ \\
%  {and} \\ \\
  {Department of Mathematics} \\
  {Universitat Politècnica de Catalunya} \\
  {Carrer Pau Gargallo, 14} \\
  {Barcelona, 08030, Spain} \\ \\
%  {Email: gerard.castro.lopez@estudiantat.upc.edu} \\ \\

  \textit{Current address: } \\
  
  {Department of Mathematics} \\
  {ETH Zurich} \\
  {R\"amistrasse 101} \\
  {Zurich, 8092, Switzerland} \\
  {Email: gcastro@student.ethz.ch} \\ \\
  
  \textbf{Javier G\'omez-Serrano}\\
  {Department of Mathematics} \\
  {Brown University} \\
  {314 Kassar House, 151 Thayer St.} \\
  {Providence, RI 02912, USA} \\
  {Email: javier\_gomez\_serrano@brown.edu} \\
\end{tabular}

\end{document}